\documentclass[final,3p,times,authoryear]{elsarticle}

\usepackage{amssymb}
\usepackage{amsmath}
\usepackage{graphicx}
\usepackage{epstopdf, epsfig}
\usepackage{amsthm}
\usepackage{mathrsfs}
\usepackage{dsfont}
\usepackage{bm}
\usepackage{mathtools}
\usepackage{cases}
\usepackage{booktabs}
\usepackage{multirow}
\usepackage{graphicx}
\usepackage{subcaption}
\usepackage{bbm}

\newcommand{\norm}[1]{\lVert#1\rVert}

\makeatletter
\newcommand*{\rom}[1]{\expandafter\@slowromancap\romannumeral #1@}
\makeatother

\newtheorem{theorem}{Theorem}
\numberwithin{equation}{section}
\newtheorem{lemma}{Lemma}
\newtheorem*{remark}{Remark}

\begin{document}

\begin{frontmatter}



\title{On the accuracy of stiff-accurate diagonal implicit Runge-Kutta methods for finite volume based Navier-Stokes equations}


\author[address1,address2]{Jiawei Wan\corref{cor1}}
\ead{jiawei.wan@outlook.com}

\author[address1]{Ahsan Kareem}
\author[address2]{Haili Liao}
\author[address3]{Yunzhu Cai}

\address[address1]{University of Notre Dame, Nathaz Modelling Laboratory, Notre Dame, United States}
\address[address2]{Southwest Jiaotong University, School of Civil Engineering, Chengdu, China}
\address[address3]{Tongji University, School of Civil Engineering, Shanghai, China}

\cortext[cor1]{Corresponding author.}

\begin{abstract}

The paper aims at developing low-storage implicit Runge-Kutta methods which are easy to implement and achieve higher-order of convergence for both the velocity and pressure in the finite volume formulation of the incompressible Navier-Stokes equations on a static collocated grid. To this end, the effect of the momentum interpolation, a procedure required by the finite volume method for collocated grids, on the differential-algebraic nature of the spatially-discretized Navier-Stokes equations should be examined first. A new framework for the momentum interpolation is established, based on which the semi-discrete Navier-Stokes equations can be strictly viewed as a system of differential-algebraic equations of index 2. The accuracy and convergence of the proposed momentum interpolation framework is examined. We then propose a new method of applying implicit Runge-Kutta schemes to the time-marching of the index 2 system of the incompressible Navier-Stokes equations. Compared to the standard method, the proposed one significantly reduces the numerical difficulties in momentum interpolations and delivers higher-order pressures without requiring additional computational effort. Applying stiff-accurate diagonal implicit Runge-Kutta (DIRK) schemes with the proposed method allows the schemes to attain the classical order of convergence for both the velocity and pressure. We also develop two families of low-storage stiff-accurate DIRK schemes to reduce the storage required by their implementations. Examining the two dimensional Taylor-Green vortex as an example, the spatial and temporal accuracy of the proposed methods in simulating incompressible flow is demonstrated.

\end{abstract}

\begin{keyword}
incompressible Navier-Stokes equations  \sep
finite volume method \sep
static collocated grid \sep
momentum interpolation \sep
differential-algebraic equations \sep
Runge-Kutta method



\end{keyword}

\end{frontmatter}


\section{Introduction}\label{section1}

Runge-Kutta methods have always been a popular choice for the time-marching of the incompressible Navier-Stokes equations. Compared to multi-step methods, they combine high order with good stability, allow for adaptive time stepping and are self-starting. It is worth noting that Runge-Kutta methods are originally conceived for the solutions of ordinary differential equations (ODEs). In the case of the spatially-discretized Navier-Stokes system, the momentum equation is a set of differential equations while the continuity equations is a set of algebraic equations. Hence, the system consisting of those equations is apparently not an ordinary differential problem. In mathematics, the term differential-algebraic equation (DAE) was introduced to comprise differential equations with algebraic constraints. Moreover, the variables in a system of DAEs for which derivatives exist are referred as the differential components (e.g., the velocity), and the variables in the system for which no derivatives exists (e.g., the pressure) are referred as the algebraic components. Compared to a system of ODEs, the system of DAEs have distinctive characteristics and are generally more difficult to solve. It may happen that the orders of local accuracy and global convergence of a particular Runge-Kutta scheme for DAEs are lower than those for ODEs. 

The study of Runge-Kutta methods for differential-algebraic problems did not start until the late 1990s \citep{hairer1989}. However, along with the related theories being established, the application of Runge-Kutta methods in the temporal discretization of the incompressible Navier-Stokes equations was initiated, e.g. \cite{kim1985} and \cite{le1991}. Before the computational fluid dynamics (CFD) community became fully aware of the differential-algebraic nature of the Navier-Stokes equations, it had been the common practice to march the velocity in time as if the semi-discrete Navier-Stokes equations are a system of ODEs, and then solve a Poisson equation to enforce the divergence-free constraint. While such type of approaches usually attain the expected orders of convergence for both the velocity and pressure when applied to flow problems with time-independent boundary conditions (e.g. the periodic boundary condition), their order results can be unsatisfactory for flow problems with time-varying boundary conditions. This issue was seldom considered in the early works \citep{kim1985, le1991, pereira2001, nikitin2006} on the applications of Runge-Kutta methods in CFD. Before we address this issue, it is good to know a term known as the index was introduced in literature to classify the systems of differential-algebraic equations. The DAE systems that frequently arise in practice are of index 1, 2 and 3, and a same Runge-Kutta scheme usually leads to different orders of convergence in the three systems. According to the definition of the index presented in \cite{hairer1989}, the incompressible Navier-Stokes equations, when discretized in space using a finite difference method (or a finite element method, or a finite volume method on a staggered grid), form an index 2 differential-algebraic system. For the finite volume formulation of the incompressible Navier-Stokes equations on a collocated or non-staggered grid, this conclusion needs to be examined.

The most distinctive feature of the finite volume method on a non-staggered grid, compared to the one on a staggered grid, is that it requires the evaluation of velocities at cell-faces during the spatial discretization. The computed cell-face velocities are later utilized in the spatial discretization of the momentum and the continuity equation. It is common practice to compute cell-face velocities by a custom interpolation of the fully-discrete momentum equation known as the Rhie-Chow interpolation \citep{rhie1983}. Refining momentum interpolation based on Rhie-Chow's framework has always been a hot topic in computational fluid dynamics. However, most of the related works \citep{choi1999, yu2002, pascau2011, zhang2014, martinez2017, tukovic2018} mainly focused on obtaining time-step size-independent and under-relaxation factor-independent cell face velocities, which did not address how the momentum interpolation influences the differential-algebraic attribute of the system of the semi-discrete incompressible Navier-Stokes equations. The spatial convergence of the momentum interpolation and the effect of the momentum interpolation on the convergence results of temporal solutions are seldom discussed either. To address these issues, we develop a new framework for computing the cell-face velocity. The proposed framework sheds light on the role of cell-face velocities in semi-discrete incompressible Navier-Stokes system. The finite volume formulation of the incompressible Navier-Stokes equations, in corporate with the cell-face velocity computed by the proposed framework, can be formally viewed as a system of index 2.

The application of Runge-Kutta schemes (especially the explicit ones) for differential-algebraic problems is not as simple as it is for ordinary differential problems. The key issues encountered during the application include (but not limited to) the enforcement of algebraic constraints, the existence and uniqueness of Runge-Kutta solutions and the local accuracy and global convergence analysis. Several approaches were proposed for index 2 problems and the commonly used ones are the direct approach \citep{hairer1989}, the (partitioned) half-explicit methods \citep{brasey1993, murua1997} and the collocation methods \citep{wanner1991}. As its name implies, the direct approach approximates the solution of a differential-algebraic problem by directly applying an implicit Runge-Kutta (IRK) scheme to the discretization of differential equations with the algebraic equations regarding the discretized differential components being satisfied at the same time. In comparison, the half-explicit methods employ an implicit and an explicit Runge-Kutta scheme simultaneously during the solution of index 2 problems. The explicit Runge-Kutta scheme deals with the discretization of differential equations, while the implicit Runge-Kutta scheme handles the enforcement of algebraic constraints. Compared to the half-explicit methods, the direct approach is relatively easier to implement and requires less storage. However, it is in general more computational expensive than the half-explicit methods at the same orders of accuracy.

The accuracy and convergence results of the direct approach and the half-explicit methods were summarized in details in \cite{hairer1989} and \cite{wanner1991}. The related theoretical findings indicate that, in addition to the classical order conditions, extra order conditions on the Runge-Kutta coefficients of a particular scheme needs to be satisfied accordingly for this scheme to guarantee a certain order of accuracy, when it is applied to general non-linear index 2 problems. Also note that the order results for differential and algebraic components may not necessarily be the same. It is usually more difficult for a Runge-Kutta scheme to attain higher-order algebraic components than differential components. For example, diagonal implicit Runge-Kutta (DIRK) schemes (i.e., the IRK scheme with a lower triangle coefficient matrix) can only attain (at most) second order and first order convergences, respectively, for the differential and algebraic components of a general non-linear index 2 problem. This implies that DIRK schemes may not be a good choice for index 2 problems in terms of orders of accuracy. However, DIRK schemes are actually quite popular in solving ordinary differential equations due to their advantages in computational efficiency when compared to the fully implicit IRK schemes. A remedy to improve the order of convergence for DIRK-type of schemes (i.e., the schemes that allow the discretized equations at different Runge-Kutta internal stages to be solved successively) was proposed by \cite{jay1993} in which the diagonal implicit Runge-Kutta schemes with an explicit first-stage (ESDIRK) was considered.

Fortunately, the system of the semi-discrete incompressible Navier-Stokes equations on a stationary mesh can actually be viewed as a special index 2 problem in which the partial derivatives of the algebraic equations with respect to the differential components are constants in this system. An direct consequence of this property is that the order conditions for a scheme to attain a certain order of accuracy are less strict (mainly in quantities) than those for general non-linear index 2 problems. As a result, it now becomes possible for a DIRK scheme to attain the classical order of convergence for velocities, and for pressures as well if the boundary conditions for velocities are time-independent. However, if the boundary conditions for velocities are time-varying, the convergence order for pressures delivered by DIRK schemes can still be first order only. Considering a temporal accurate pressure is of major interest in many CFD applications (e.g., the estimation of the aerodynamic coefficients of a rigid model), an increasing amount of efforts were made to develop Runge-Kutta methods which are capable of delivering higher-order temporal accurate velocities and pressures.

One way to achieve this goal is to treat the semi-discrete incompressible Navier-Stokes equations as a general non-linear index 2 problem and seek for potential prescriptions among the fully-implicit Runge-Kutta methods, the ESDIRK methods and the half-explicit methods. Applying the half-explicit methods to the semi-discrete incompressible Navier-Stokes equations, one plausible attempt by \cite{sanderse12} examined the convergence results of the velocity and pressure. They also investigated the possibility of attaining higher-order accurate pressures without requiring an additional solution of the pressure Poisson equation. The effect of unsteady boundary conditions and deforming meshes on the orders of convergence was also considered in their study. Another class of approaches comes from the idea of projecting the semi-discrete incompressible Navier-Stokes equations onto a discrete divergence-free space. The system of the projected Navier-Stokes equations is a system of ODEs. Consequently, the application of Runge-Kutta schemes becomes straightforward and the convergence results for both the velocity and pressure are determined by the classical order of the Runge-Kutta scheme. In the trial by \cite{colomes2016}, they developed a family of implicit-explicit Runge-Kutta (IMEXRK) schemes and applied it to the projected Navier-Stokes system. Their method works fine if the boundary conditions for the continuity equation are time-independent or $n$-th order differentiable with $n$ lower than the classical order of the Runge-Kutta scheme.  If this assumption on boundary conditions is not satisfied, the divergence-free constraint will no longer be satisfied due to the time integration error, though the numerical results show the right order of convergence.

In \cite{sanderse12}, the incompressible Navier-Stokes equations were discretized in space using a finite volume method on a staggered grid, while \cite{colomes2016} employed a finite element differencing method. An issue that emerges immediately when applying their methods or any other higher-order time-integrators to the incompressible Navier-Stokes equations on a collocated finite volume grid is the computation of cell-face velocities at the Runge-Kutta internal stages. As noted in \cite{choi1999}, the momentum interpolation gets complicated with the amounts of interpolations and additional storage required by higher order methods. The situation may get worse for approaches such as the half-explicit methods and the IMEXRK methods due to the complexities of their formulation. The latest studies on the application of Runge-Kutta methods to the incompressible Navier-Stokes equations on a collocated finite volume grid, e.g., \cite{vuorinen2014} and \cite{valerio2018}, mainly focused on implementing the existing methods using open source codes. Less attention was paid to the momentum interpolation at the Runge-Kutta internal stages. The effect of the momentum interpolation on the differential-algebraic attribute of the semi-discrete Navier-Stokes equations and how to reduce the cost of mathematical calculation required by the momentum interpolation were not discussed. Another issue seldom addressed in the application of Runge-Kutta methods to CFD is the storage consumption which often turns out be a computational bottleneck for the solution of high-dimensional Navier-Stokes systems. Low-storage Runge-Kutta methods for ordinary differential problems were discussed in details in \cite{kennedy2000} for explicit schemes and recently in \cite{cavaglieri2015} for implicit-explicit schemes. Apparently, the same discussion is also necessary in our case.

To address the aforementioned problems successively, we start with tackling the numerical difficulties in  momentum interpolations. As we shall demonstrate later, the Runge-Kutta schemes with non-singular coefficient matrix are desirable, since a simple modification of the conventional formulation of those schemes can make the momentum interpolations at Runge-Kutta internal stages much easier to proceed. This finding leads us to the fully implicit schemes or the DIRK schemes. Considering the implementation of fully implicit schemes is computational demanding and requires a Newton-type of non-linear equation solver (which in practice can be a major obstacle), we prefer to the DIRK schemes here. As we have mentioned earlier, when applied directly to the semi-discrete incompressible Navier-Stokes equations with time-varying velocity boundaries, DIRK schemes can only attain first order temporal accurate pressures at most. To improve the convergence order for the pressure, we propose a new method of applying implicit Runge-Kutta schemes which also incorporates the modification that helps reducing the numerical difficulties in momentum interpolations. The convergence results of the proposed method is then analyzed mathematically. Compared to the direct approach, the proposed method delivers higher-order temporal-accurate pressures without requiring additional computational effort. Moreover, the proposed method allows stiff-accurate DIRK schemes to attain the classical order of convergence for both the velocity and pressure. Finally, two families of low-storage stiff-accurate diagonally implicit Runge-Kutta schemes are developed to further reduce the storage required by the proposed method.

The outline of this paper is as follows. In Section \ref{section2}, we discuss the finite volume formulation of the incompressible Navier-Stokes equations on a stationary collocated grid and introduce a new framework for the momentum interpolation. In Section \ref{section3}, we compare the formulations and the convergence results of the direct approach and the proposed method for the index 2 system of the semi-discrete incompressible Navier-Stokes equations. The low storage implementation of the proposed method is presented afterwards. The discussion on the solution algorithms of the fully-discrete non-linear Navier-Stokes equations at Runge-Kutta internal stages and on the effect of residual errors on convergence results is presented in Section \ref{section4}. In Section \ref{section5}, we take the two dimensional Taylor-Green vortex to as an example to validate the proposed momentum interpolation framework and Runge-Kutta method.

\section{Spatial Discretization of the Incompressible Navier-Stokes Equations} \label{section2}

In this section, we will discuss the spatial discretization of the incompressible Navier-Stokes equations employing a second-order finite volume method on a cell-centered collocated grid arrangement. The main concern will be the effect of the momentum interpolation on the differential-algebraic attribute of the semi-discrete Navier-Stokes system, for which a new framework for the computation of cell-face velocities is developed. The accuracy and convergence of the proposed momentum interpolation framework will also be investigated.

\subsection{Spatial discretization of the momentum and continuity equations} \label{section2.1}

To begin with, we have the three-dimensional incompressible momentum and continuity conservation equations for Newtonian fluids given as follows:

\begin{equation} \label{NS_partial}
\begin{split}
\frac{\partial \boldsymbol{u}}{\partial t} + \nabla \boldsymbol{\cdot} (\boldsymbol{u} \boldsymbol{u}) &= \nabla \boldsymbol{\cdot} (\nu \nabla \boldsymbol{u}) - \nabla p, \\
\nabla \boldsymbol{\cdot} \boldsymbol{u} &= 0,
\end{split}
\end{equation}

\noindent where $\boldsymbol{u}$ is the fluid velocity, $p$ represents the kinematic pressure and $\nu$ denotes the kinematic viscosity. For the spatial discretization of (\ref{NS_partial}), we utilize a second-order finite volume method established on a collocated (or non-staggered) grid system. Thus the system of the semi-discrete equations can be written as :

\begin{subequations} \label{NS_FV}
\begin{gather}
u'(t)= \big\{K + N(\bar{u})\big\} u(t) - Gp(t) + b(t), \label{fdnse}\\
D\bar{u}(t) = r(t), \label{fdce} \\
\intertext{where we use the notation}
\quad \big\{K + N(\bar{u})\big\} \coloneqq \text{blockdiag}\left(K + N(\bar{u}),\ldots, K + N(\bar{u})\right). \nonumber
\end{gather}
\end{subequations}

\noindent In (\ref{NS_FV}) and for later use, $u(t)\in\mathbb{R}^{3m}$ and $p(t)\in\mathbb{R}^{m}$ (where $m$ is the number of the cells for a given mesh) are the vectors storing the average values of the velocity and pressure over the cells. $\bar{u}(t)\in\mathbb{R}^{3\bar{m}}$ (where $\bar{m}$ is the number of the cell-faces) denotes the vector collecting the discrete values of velocities at the cell-faces. The elements of $u$ and $\bar{u}$ are sorted in the form of

\begin{equation}\label{upartition}
u \coloneqq \big( u_{(1)}, u_{(2)}, u_{(3)}\big)^T,\quad \bar{u} \coloneqq \big(\bar{u}_{(1)}, \bar{u}_{(2)}, \bar{u}_{(3)}\big)^T, \nonumber
\end{equation}

\noindent where $u_{(i)}(t) \in \mathbb{R}^{m}$ and $\bar{u}_{(i)}(t) \in \mathbb{R}^{\bar{m}}$ ($i=1,2,3$) are the vectors storing the three components of velocities at the cell centroids and face centroids, respectively. Note that all of the vectors that represent a discrete vector field are sorted in a similar manner in this study. $G\in\mathbb{R}^{3m \times m}$ and $D\in \mathbb{R}^{m \times 3\bar{m}}$ are the matrix representation of the operators that return the gradient and the divergence at cell centroids, respectively. $K$ and $N(\bar{u}) \in \mathbb{R}^{m \times m}$ are the matrices of coefficients resulting from the discretization of diffusion and convection terms. $b(t) \in \mathbb{R}^{3m}$ and $r(t) \in \mathbb{R}^{m}$ are the vectors storing the source terms arising from the velocities at mesh boundaries. Note that $K$, $G$ and $D$ are constant on stationary mesh which will be assumed from now on. We also use the braces $\{*\}$ to denote the operator that returns a block diagonal matrix of the matrix in the braces. One might get confused with the expression (\ref{fdnse}), especially the term related to the block diagonal matrix operator. To make it easier to understand, we rewrite (\ref{fdnse}) as

\begin{equation}\label{fdnseAppend}
\left(
\begin{matrix}
u'_{(1)} \\
u'_{(2)} \\
u'_{(3)}
\end{matrix}
\right)
= 
\left(
\begin{matrix}
K + N & 0 & 0 \\
0 & K + N & 0 \\
0 & 0& K + N
\end{matrix}
\right)
\left(
\begin{matrix}
u_{(1)} \\
u_{(2)} \\
u_{(3)}
\end{matrix}
\right)
+
\left(
\begin{matrix}
G_{(1)} \\
G_{(2)} \\
G_{(3)}
\end{matrix}
\right)
p
+
\left(
\begin{matrix}
b_{(1)} \\
b_{(2)} \\
b_{(3)}
\end{matrix}
\right),
\end{equation}

\noindent where $G_{(i)}\in \mathbb{R}^{m\times m}$ ($i=1,2,3$) is obtained by partitioning $G$ into $G \coloneqq (G_{(1)},G_{(2)},G_{(3)})^T$. 

Now, it is very clear that the equation (\ref{fdnse}) is actually composed of three differential equations with each equation being the semi-discrete momentum equation regarding one component of cell-center velocities. We advocate the use of (\ref{fdnse}) rather than (\ref{fdnseAppend}) for it is simple and addresses the fact that the semi-discrete momentum equations regarding different components of cell-center velocities share the same coefficient matrices $K$ and $N$ resulting from the spatial discretization of diffusion and convection terms. Also note that the conventional finite volume formulation of the semi-discrete momentum equation should have a term $\{V\}$ (where $V\in \mathbb{R}^{m\times m}$ is a diagonal matrix of cell volumes) on the left hand side (LHS) of the $u'$ term, in this study however, we eliminate this term for simplicity by multiplying $\{V^{-1}\}$ to the both sides of the original equation. Consequently, the $K$, $N$ and $b$ terms have included the inverse of the cell volumes in them. Finally, it should be emphasized again the $u$, $\bar{u}$ and $p$ terms in the semi-discrete momentum equation (\ref{NS_FV}) are discrete fields rather than the continuous fields $\boldsymbol{u}$ and $p$ in the conservative form of the momentum equation (\ref{NS_partial}). The system of equations (\ref{NS_FV}) are not sufficient enough to define a closed form solution and cannot be solved at this stage. Additional equations with respect to the cell-face velocity field $\bar{u}$ should be introduced to (\ref{NS_FV}) to make the system solvable.

\subsection{The Rhie and Chow momentum interpolation} \label{section2.2}

It is common practice to compute cell-face velocities by a custom interpolation of the momentum equation known as the Rhie and Chow interpolation. It is noted that in Rhie and Chow's method and also in other momentum interpolation methods involved in literatures, the momentum equation to be interpolated specifically refers to the fully-discrete momentum equation rather than the semi-discrete momentum equation. For example, if we employ a first-order implicit Euler method for the time-marching of (\ref{fdnse}) from $t_{n}$ to $t_{n+1} = t_n +h$ (where $h$ is the time-step size) and do not consider an under relaxation procedure, we will have

\begin{equation} \label{Euler}
u_{n+1}= u_{n} + h \left\{K+N_{n+1}\right\} u_{n+1} - hGp_{n+1} + hb_{n+1},
\end{equation}

\noindent where the subscript $n$ designates the $n$-th time-instance $t_{n}$ here and in the remainder of this paper. For simplicity, we also write $N_{n+1}$ and $b_{n+1}$ for $N(\bar{u}_{n+1})$ and $b(t_{n+1})$, respectively. With the semi-discrete momentum equation discretized temporally, Rhie and Chow's momentum interpolation method approximates the $\bar{u}_{n+1}$ field through a custom interpolation of the fully-discrete momentum equation given by (\ref{Euler}). The interpolation schemes supported by Rhie and Chow's method as well as other momentum interpolation methods are usually expressed locally in terms of a specific element in $\bar{u}_{n+1}$, in this study however, we seek to formulate the interpolation scheme with respect to the whole cell-face velocity field $\bar{u}_{n+1}$. To achieve this purpose some notations are introduced first. 

First, we use the terms cell-centered and face-centered fields to denote the discrete fields defined  at the centroids of the cells and cell-faces of a collocated grid, respectively. Subsequently, we refer the spatial interpolation of a scalar or vector field from cell centroids to cell-face centroids as the ``face interpolation". Then, we denote the elements of discrete scalar fields by superscripts of lower case letters, e.g., the $i$-th element of a cell-centered scalar field $\phi\in\mathbb{R}^m$ is denoted by $\phi^i$ (where $i=1(1)m$). To distinguish with scalar elements, we denote the elements of discrete vector fields (which are vectors) by superscripts of bold face lower case letters.  For example, the $i$-th element of a cell-centered vector field $\varphi\in \mathbb{R}^{3m}$ is denoted by

\begin{equation}
\varphi^{\boldsymbol{i}} \coloneqq \big(\varphi_{(1)}^i,\varphi_{(2)}^i,\varphi_{(3)}^i\big)^T. \nonumber
\end{equation}

\noindent Note that the above mentioned $\varphi^{\boldsymbol{i}}$ and $\phi^i$ terms can be regarded as the vector and scalar stored at the cell with the index $i$ (referred to as cell $i$ hereafter). Similarly, the $i$-th element of a face-centered variable field corresponds to the value of this variable at the cell-face with the index $i$ (referred to as face $i$ hereafter).

To express the conduction of the face interpolation in a more compact way, we define $\eta (\phi): \phi \in \mathbb{R}^{m} \rightarrow\mathbb{R}^{\bar{m}}$ as the function that returns a face-centered scalar field through the spatial interpolation of a cell-centered scalar field $\phi$ with a specific interpolation scheme. The commonly used interpolation schemes include the second-order linear interpolation scheme (also known as the central differencing scheme), the second order FROMM scheme \citep{fromm1968} and the third-order QUICK scheme \citep{leonard1979}. It is noted that those interpolation schemes are based on the polynomial interpolation method. When they are utilized for the $\eta$ function, each element of the output face-centered scalar field will be evaluated as a linear combination of the corresponding elements in the input cell-centered scalar field. Therefore, if the $\eta$ function employs a polynomial-based interpolation scheme, it can be expressed as

\begin{equation} \label{interpolation1}
\eta(\phi) = L\phi \quad \text{or} \quad \eta^i(\phi) = \sum_{j\in S_i}l_{ij}\phi^j \ \  \text{for} \ \ i = 1(1)\bar{m}, 
\end{equation}

\noindent where $L \in \mathbb{R}^{\bar{m}\times m}$ is a matrix of coefficients determined by the geometry of the mesh and the order of the polynomials. $l_{ij}$ is the element on in the $i$-th row and $j$-th column of $L$, and $S_i$ is the set of the indices of the non-zero elements in the $i$-th row of $L$. Note that $S_i$ can also be regarded as the set of the indices of the elements of the $\phi$ field whose values are required for the local interpolation of the $\phi$ field to face $i$.

Based on the definition of the $\eta$ function, we introduce the notation $\boldsymbol{\eta}(\varphi):\varphi \in\mathbb{R}^{3m} \rightarrow\mathbb{R}^{3\bar{m}}$ (the bold $\eta$ symbol) to denote the function that returns a face-centered vector field through the face interpolation of a cell-centered vector field by interpolating its three components respectively

\begin{equation} \label{interpolatedField}
\boldsymbol{\eta}(\varphi) = \big( \eta (\varphi_{(1)}), \eta (\varphi_{(2)}),\eta (\varphi_{(3)}) \big)^T \quad \text{or} \quad \boldsymbol{\eta}^{\boldsymbol{i}}(\varphi) = \big(\eta^i(\varphi_{(1)}), \eta^i(\varphi_{(2)}), \eta^i(\varphi_{(3)})\big)^T \ \  \text{for} \ \ i = 1(1)\bar{m}.
\end{equation}

\noindent If $\eta$ follows (\ref{interpolation1}), the vector field interpolation function $\boldsymbol{\eta}$ can be rewritten as

\begin{equation}
\boldsymbol{\eta}(\varphi) = \left\{L\right\}\varphi \quad \text{or} \quad \boldsymbol{\eta}^{\boldsymbol{i}}(\varphi) = \big(\sum_{j\in S_i}l_{ij}\varphi_{(1)}^j,\sum_{j\in S_i}l_{ij}\varphi_{(2)}^j,\sum_{j\in S_i}l_{ij}\varphi_{(3)}^j\big)^T \ \  \text{for} \ \ i = 1(1)\bar{m}.
\end{equation}

\noindent Finally, we use the notation $\text{diag}(*)$ to denote the operator that returns a diagonal matrix whose diagonal elements equal to the elements of the term in the brackets if the term in the brackets is a vector, or returns a vector whose elements are the diagonal elements of the term in the brackets if the term in the brackets is a matrix. Based on it, we define the symbols

\begin{equation}
A(\bar{u})\coloneqq\text{diag}\left(K+N(\bar{u})\right), \quad H(t,u,\bar{u})\coloneqq\left\{K+N(\bar{u})-\text{diag}\left(A(\bar{u})\right)\right\}u+b(t),
\end{equation}

\noindent where $A\in \mathbb{R}^m$ can be regarded a cell-centered scalar field storing the diagonal coefficients of $K$ and $N$, and $H\in \mathbb{R}^{3m}$ denotes a cell-centered vector field collecting the off-diagonal coefficients of $K$ and $N$ multiplied by the corresponding cell-center velocities and the source term $b$.

With the above notation ready, a matrix representation of Rhie and Chow's momentum interpolation procedure starts with rewriting the fully-discrete momentum equation (\ref{Euler}) as

\begin{equation} \label{EulerModified}
u_{n+1}= \big\{A^*\big\} \left(u_{n}+ hH_{n+1}\right) -h\big\{A^*\big\}Gp_{n+1}, \quad A^* \coloneqq \left(I-h\ \text{diag}(A_{n+1})\right)^{-1},
\end{equation}

\noindent where $A_{n+1} = A(\bar{u}_{n+1})$, $H_{n+1} = H(t_{n+1},u_{n+1},\bar{u}_{n+1})$ and the superscript $^{-1}$ denotes the inverse of a matrix. For the approximation of the face-centered velocity field $\bar{u}_{n+1}$, the Rhie and Chow interpolation gives

\begin{equation}  \label{Rhie&Chow}
\bar{u}_{n+1}= \big\{L\big\}\big\{A^*\big\}\big(u_{n}+ hH_{n+1}\big) -h\big\{\bar{A}^*\big\}\bar{G} p_{n+1},\quad \big\{\bar{A}^*\big\} \coloneqq  \text{diag}  \left(L\ \text{diag}\left(A^*\right)\right),
\end{equation}

\noindent where $\bar{G} \in \mathbb{R}^{3\bar{m}\times m}$ is the matrix representation of the cell-face gradient operator and the coefficients of $L$ are determined by the central differencing scheme. A worth-mentinoning property of the Rhie and Chow interpolation is that when (\ref{Rhie&Chow}) is applied to a steady-state flow problem, the converged solution is time-step size-dependent. This feature has been long realized in literature and a number of attempts, including the work of \cite{choi1999}, \cite{yu2002} and \cite{pascau2011}, were made to tackle it. These methods express the terms related to the cell-face counterparts of the $A$ and $H$ terms in different manners. For example, for the face interpolation of (\ref{Euler}), Choi's method gives

\begin{equation} \label{Choi}
\bar{u}_{n+1} = \big\{\bar{A}^*\big\}\bar{u}_{n}+h\big\{L\big\}\big\{A^*\big\}H_{n+1} - h\big\{\bar{A}^*\big\}\bar{G}p_{n+1},
\end{equation}

\noindent where the expressions of $A^*$ and $\bar{A}^*$ are consistent with their expressions in Rhie and Chow's method. In the meantime, the approach of \cite{yu2002} yields a formula in the form

\begin{equation} \label{Yu}
\bar{u}_{n+1}=  \big\{\bar{A}^*\big\} \left( \bar{u}_{n} +h\big\{L\big\}H_{n+1}-h\bar{G}p_{n+1}\right),\quad \bar{A}^* \coloneqq \left(I-h\ \text{diag}\left(LA_{n+1}\right)\right)^{-1}.
\end{equation}

\noindent Finally, a momentum interpolation scheme follows from the work by \cite{pascau2011} has the form of

\begin{equation}  \label{Pascau}
\bar{u}_{n+1} = \big\{\bar{A}^*\big\} \left[\bar{u}_n + h\Big\{\text{diag}\Big(\big(LA_{n+1}^{\circ \left(-1\right)}\big)^{\circ \left(-1\right)}\Big)\Big\}\Big\{L\Big\}\Big\{\text{diag}\big(A_{n+1}^{\circ \left(-1\right)}\big)\Big\}H_{n+1}-h\bar{G}p_{n+1}\right], \quad \bar{A}^* \coloneqq \Big(I-h\text{diag}\big((LA_{n+1}^{\circ \left(-1\right)})^{\circ\left(-1\right)}\big)\Big)^{-1},
\end{equation}

\noindent where the superscript $\circ\left(-1\right)$ denotes the Hadamard inverse. 

A common feature of Rhie and Chow's method and other related methods is that the main diagonals of the matrices $K$ and $N$ are extracted to $A$ during momentum interpolations. However, it is still questionable if this treatment will affect the convergence of the momentum interpolation or not. As a matter of fact, the issue of the convergence of the momentum interpolation has seldom been addressed in previous studies. Apart from the effect of the momentum interpolation on the spatial accuracy of the finite volume discretization, it is also difficult to say if and how the momentum interpolation influences the differential-algebraic nature of the semi-discrete incompressible Navier-Stokes equations in the aforementioned momentum interpolation methods.

\subsection{A new framework for momentum interpolation} \label{section2.3}

To address the aforementioned problems which were seldom been considered in literatures, we propose a new framework for the computation of the cell-face velocity in a collocated grid system. The proposed framework gives the insight on the role of cell-face velocities in the system of the semi-discrete incompressible Navier-Stokes equations, and consequently the effect of the momentum interpolation on spatial and temporal accuracy can be easily evaluated. The basic ideal behind the proposed framework is that we first compute the time-derivative of cell-face velocities through a customs interpolation of the semi-discrete momentum equation of cell-center velocities, and then integrate the semi-discrete equation of cell-face velocities with respect to time to obtain the temporally discretized cell-face velocity field. With this idea in mind, we now describe the details of the proposed framework following. 

Assume that a pseudo grid system is properly established for the velocities at cell-face centroids and we have all of the information required to formulate a finite volume representation of the semi-discrete momentum equation with respect to the cell-face velocity, thus we will have the time-derivative of the face-centered velocity field $\bar{u}$ given by

\begin{equation} \label{daem00}
\bar{u}' =  \big\{\bar{K}+\bar{N}\big\}\bar{u}+\bar{b} - \bar{G}p,
\end{equation}

\noindent where $\bar{K}$, $\bar{N}\in \mathbb{R}^{\bar{m} \times \bar{m}}$ and $\bar{b} \in \mathbb{R}^{3\bar{m}}$ are, respectively, the counterparts of the $K$, $N$ and $b$ terms in the semi-discrete momentum equation of cell-center velocities. In reality, the $\bar{K}$, $\bar{N}$ and $\bar{b}$ terms cannot be given explicitly, therefore we can only approximate the $\{\bar{K}+\bar{N}\}\bar{u}+\bar{b}$ term (which is a face-centered vector field) through a custom face interpolation of the $\{K+N\}u+b$ term (which is a cell-centered vector field) instead. Before we carry out the interpolation, we define the face-centered scalar field $\bar{A}\in \mathbb{R}^{\bar{m}}$ and the face-centered vector field $\bar{H}\in \mathbb{R}^{3\bar{m}}$ as

\begin{equation} \label{Anotation}
\bar{A} \coloneqq \alpha \circ \text{diag}(\bar{K})+\beta \circ \text{diag}(\bar{N}), \quad \bar{H} \coloneqq \big\{\bar{K}+\bar{N}-\text{diag}(\bar{A})\big\} \bar{u}+\bar{b},
\end{equation}

\noindent where $\alpha$ and $\beta$ are vectors collecting (possibly time-varying) scalar coefficients defined for each cell face with $\alpha^i,\beta^i\in [0,1]$, and the symbol $\circ$ denotes the Hadamard product. Note that $\alpha$ and $\beta$ can also be regarded as face-centered scalar fields. A detailed discussion on their roles and how they affect the accuracy of the proposed momentum interpolation framework will be presented later. At this stage, we shall continue the process. Inserting (\ref{Anotation}) into (\ref{daem00}) yields

\begin{equation}\label{daem0}
\bar{u}' =  \big\{\text{diag}(\bar{A})\big\} \bar{u}+\bar{H} - \bar{G}p.
\end{equation}

\noindent With the above equation ready, we now approximate the $i$-th element of the $\bar{A}$ and $\bar{H}$ fields with

\begin{subequations}\label{daem}
\begin{gather}
\bar{A}^i = \eta^i\big(A_i\big), \quad \bar{H}^{\boldsymbol{i}} = \boldsymbol{\eta}^{\boldsymbol{i}}\big(H_i\big) \\
\intertext{where $A_i\in \mathbb{R}^m$ and $H_i\in \mathbb{R}^{3m}$ are, respectively, the cell-centered scalar and vector fields defined as}
A_i \coloneqq \alpha^i\text{diag}(K)+\beta^i\text{diag}(N), \quad H_i \coloneqq \big\{K+N-\text{diag}(A_i)\big\}u+b.
\end{gather}
\end{subequations}

\noindent (\ref{daem0}) incorporate with (\ref{daem}) is the key formula of the proposed momentum interpolation framework and we name this evaluation of the time-derivatives of cell-face velocities as a generalized face interpolation of the semi-discrete momentum equation (GFISDM). The scheme of GFISDM offers the users the freedom in specifying the values of $\alpha$ and $\beta$, and the interpolation schemes used for the $\eta$ functions in it. Also note that one may use $\eta$ functions based on different schemes for the face interpolation of the $A_i$ and $H_i$ fields, respectively.

Apart from the difference of the stage at which the momentum interpolation is carried out there is one more significant difference between the proposed momentum interpolation framework and other momentum interpolation methods. In the other momentum interpolation methods, the original momentum equation of cell-center velocities, e.g. (\ref{Euler}), normally goes through a series of transformation to deliver a rearranged equation, e.g. (\ref{EulerModified}). Then, the momentum equation of cell-face velocities is assumed to be composed of terms which are identical to those of the rearranged momentum equation of cell-center velocities. Finally, the unknown terms in the momentum equation of cell-face velocities will be approximated through the interpolations of their counterparts in the equation of cell-center velocities. In those methods, the expressions of the terms to be interpolated are identical for each cell-face.

However, the GFISDM scheme starts with assuming that the form of the pseudo semi-discrete momentum equation regarding the cell-face velocity is similar to that of the cell-center velocity. Subsequently, the semi-discrete momentum equation of cell-face velocities is rearranged into the expression (\ref{daem0}) where we introduce the face-centered coefficient fields $\alpha$ and $\beta$ such that the unknown $\bar{A}$ and $\bar{H}$ fields of the rearranged semi-discrete equation of cell-face velocities is $\alpha$- and $\beta$-dependent. To approximate the $i$-th elements of the $\bar{A}$ and $\bar{H}$ fields, we divide the $\{K+N\}u+b$ field into the terms $\{\text{diag}(A_i)\}u$ and $H_i$ with the expression (\ref{daem}) where $A_i$ and $H_i$ are cell-centered fields to be interpolated locally to the centroid of face $i$. Therefore, if $\alpha$ and $\beta$ are not uniform, the $A_i$ and $H_i$ fields are not exactly the same for each cell-face. Note that this does not mean we have to compute the entire $A_i$ and $H_i$ fields respectively for each cell-face. It is because, in most cases, the local interpolation of $A_i$ and $H_i$ only requires their values at a small set of nodes. If the elements of $\alpha$ and $\beta$ are uniform, the $A_i$ and $H_i$ fields will be identical for each cell-face. Thus, in this situation, we can compute the whole $A_i$ and $H_i$ fields once for all, and interpolate them to cell-faces globally in the computational domain.

The proposed framework is not yet complete at this stage, the next step is to integrate the GFISDM scheme with respect to time to deliver the cell-face velocities at the discrete time point. To compare with the formulae of the Rhie-Chow interpolation and other momentum interpolation methods, we consider an implicit Euler method for the temporal discretization of the GFISDM scheme, this yields

\begin{equation}\label{faem}
\bar{u}_{n+1} =\big\{\bar{A}^*\big\}\bar{u}_{n} + h\big\{\bar{A}^*\big\}\bar{H}_{n+1} -h\big\{\bar{A}^*\big\}\bar{G}p_{n+1}, \quad \bar{A}^* \coloneqq \big(I-h\ \text{diag}\big(\bar{A}_{n+1}\big)\big)^{-1}.
\end{equation}

\noindent where $\bar{H}_{n+1}$ denotes the $\bar{H}$ field at $t_{n+1}$. One easily verifies (\ref{faem}) is identical to (\ref{Yu}) when both $\alpha$ and $\beta$ are constants and equal to $\mathbbm{1}$ (where $\mathbbm{1}$ denotes a vector of elements one), and the $\eta$ function has the form of (\ref{interpolation1}). From this point of view, the discretized face-centered velocity field derived from the proposed framework can be regarded as a generalization of the scheme presented in \cite{yu2002}.

The semi-discrete momentum equation of cell-center velocities (\ref{fdnse}), the discretized continuity equation (\ref{fdce}) together with the GFISDM scheme define a system of differential-algebraic equations that has a closed form solution. For simplicity, we use the notation $F(t,u,\bar{u},p)$ and $\bar{F}(t,u,\bar{u},p)$ to denote the right hand sides of (\ref{fdnse}) and (\ref{daem}) respectively, and write the governing equations for all velocities and pressures in a compact form as

\begin{equation} \label{system}
u' = F(t,u,\bar{u},p), \quad \bar{u}' = \bar{F}(t,u,\bar{u},p), \quad D\bar{u} = r(t).
\end{equation}

\noindent Note that $D\bar{G}$ is non-singular and its inverse is bounded for a properly implemented finite volume method. Consequently, the system of equations given by (\ref{system}) incorporate with appropriate initial conditions for velocities and pressures clearly defines an index 2 differential-algebraic problem. In practice, the analytical solution of this index 2 problem on a given time domain is not available. Therefore, a numerical method is required for the time-marching of this system, and the temporal accuracy of the numerical method can be examined by applying convergence theories developed for index 2 problems. It is noted that the uniqueness and convergence of the numerical solution of (\ref{system}) usually requires the initial conditions for $u$, $\bar{u}$ and $p$ satisfying the equities

\begin{equation} \label{NSconsistency}
D\bar{u}=r(t),\quad   D\bar{F}(t,u,\bar{u},p) = r'(t).
\end{equation}

\noindent If the consistency constraint (\ref{NSconsistency}) is not satisfied, it is suggested to discard the original pressure field and compute a new one by solving the second relation of (\ref{NSconsistency}) instead. Another issue is when the semi-discrete incompressible Navier-Stokes system (\ref{system}) is formulated to solve a steady-state flow problem, the converged solution will satisfy

\begin{equation} \label{convergedSolution}
F(t_0,u,\bar{u},p) = 0, \quad \bar{F}(t_0,u,\bar{u},p) = 0, \quad D\bar{u} =r(t_0),
\end{equation}

\noindent where we replace the variable $t$ with a fixed time instant $t_0$ considering the source terms are constant for steady-state flow problems. Now it is easy to tell the converged solution obtained by solving (\ref{convergedSolution}) is time-step size-independent if the system (\ref{system}) itself or more specifically the $\eta$ functions in a GFISDM scheme does not depend on time-step size. From this point of view, the proposed momentum interpolation framework actually provides an alternative approach to examine if a momentum interpolation method can attain time-step size-independent converged solution for steady-state flows. To carry out the examination, we first seek for specific $\alpha$ and $\beta$, and interpolation schemes for $\eta$, with which the GFISDM scheme discretized temporally will reproduce the scheme of the momentum interpolation method that needs to be examined. Then, by examining the time-step size-dependency of the $\eta$ function, we can easily determine if the converged solutions are unconditionally time-step size-dependent or not.

As mentioned earlier, the scheme (\ref{Yu}) following from the work of \cite{yu2002} can be reproduced if we consider a GFISDM scheme with $\alpha = \beta = \mathbbm{1}$ and the $\eta$ function in the form of (\ref{interpolation1}), and we write GFISDM-Yu for this scheme hereafter. Meanwhile, to reproduce the scheme (\ref{Pascau}) based on \cite{pascau2011}, we consider the scheme (denoted by GFISDM-Pascau hereafter) with $\alpha = \beta = \mathbbm{1}$ and $\eta^i$ given by

\begin{equation} \label{IPascau}
\begin{cases}
\eta^i\big(A\big)=1/\sum_{j\in S_i} l_{ij}/A^j, \\
\boldsymbol{\eta}^{\boldsymbol{i}}\big(H\big)=\left(\sum_{j\in S_i} l_{ij}H^{\boldsymbol{j}}/A^j\right) / \sum_{j\in S_i} l_{ij}/A^j.
\end{cases}
\end{equation}

\noindent It is easy to tell both GFISDM-Yu and GFISDM-Pascau utilize time-step size-independent $\eta$ functions, we can therefore conclude that the methods of \cite{yu2002} and \cite{pascau2011} deliver time-step size-independent solution for steady-state flow problems. Finally, let's discuss the scheme (\ref{Choi}) following from the work of \cite{choi1999}. Note that Choi's method is still being used in the latest public version (i.e., version 6) of the OpenFOAM\textregistered. To reproduce (\ref{Choi}), we consider the GFISDM scheme with $\alpha = \beta = \mathbbm{1}$ and $\eta^i$ in the form of:

\begin{equation} \label{IChoi}
\begin{cases}
\eta^i\big(A\big) = 1/h-1/\sum_{j\in S_i}l_{ij}/\big(1/h-A^j\big),\\
\boldsymbol{\eta}^{\boldsymbol{i}}\big(H\big) = \sum_{j\in S_i}l_{ij}H^{\boldsymbol{j}}\big(1/h-A^j\big)/\sum_{j\in S_i}l_{ij}\big(1/h-A^j\big).
\end{cases}
\end{equation}

\noindent Apparently, the $\bar{A}$ and $\bar{H}$ fields computed by (\ref{IChoi}) are functions of time-step size $h$. Hence, GFISDM-Choi as well as the semi-discrete incompressible Navier-Stokes system based on it depends on the time-step size $h$. Therefore, we get Choi's method is not unconditionally time-step size-dependent for steady state flow problems. This finding is nothing new and were addressed in detail in \cite{yu2002} and \cite{pascau2011}. However, the effect of Choi's method on the convergence of temporal solutions for unsteady flow problems was not involved in these researches. 

Examining (\ref{IChoi}) more carefully, we find when the elements of $A$ are not identical, GFISDM-Choi can be regarded as a perturbed system of GFISDM-Yu with the perturbations of $O(h)$, and it converges to GFISDM-Yu as $h$ approaches to zero. If we consider two semi-discrete incompressible Navier-Stokes system (\ref{system}) based on GFISDM-Yu and GFISDM-Choi respectively, and employ a higher-order time-integrator for the temporal solutions of the two systems. While the time-integrator attains its designed order of convergence for the system with GFISDM-Yu, it only delivers first-order results for the system with GFISDM-Choi since this system contains terms of $O(h)$ initially. In contrast, if the elements of $A$ are identical, GFISDM-Choi is equivalent to GFISDM-Yu, and consequently the convergence results of the time-integrator will not be damaged. Also note that the elements of $A$ are not identical in general, unless the incompressible Navier-Stokes equations are formulated on a mesh system with equal grid spacing and periodic boundaries and the central differencing scheme is employed for spatial discretization of convection terms. In practice, such a numerical setup is usually adopted in the simulation of the two-dimensional Taylor-Green Vortex, a widely-used test problem used for the examination of temporal order of accuracy of a numerical integration method.

\subsection{Spatial convergence of the momentum interpolation} \label{section2.4}

Now, let's discuss the convergence of the errors of GFISDM schemes and how they affect the order of convergence of the finite volume discretization. It is noted that the error of a GFISDM scheme specifically refers to the error in approximating the $\{\bar{K}+\bar{N}\}\bar{u}+\bar{b}$ term of (\ref{daem00}). Prior to the analysis, we assume that the finite volume representation of the semi-discrete momentum equation is convergent of order 2. That is, by reducing the mesh grid spacing with the same proportion in the computational domain, the evaluations (\ref{fdnse}) and (\ref{daem00}) for the time-derivatives of velocities at cell centroids and face centroids (i.e., $u'$ and $\bar{u}'$) will converge with a second-order rate once the grid spacing becomes sufficiently small. The pressure obtained by solving the third relation of (\ref{system}) will also be convergent of order two based on the boundedness of the inverse of $D\bar{G}$. In reality, the $\{\bar{K}+\bar{N}\}\bar{u}+\bar{b}$ term in (\ref{daem00}) cannot be given explicitly, hence we evaluate $\bar{u}'$ with a GFISDM scheme instead. As a result, additional error, i.e., the error of the GFISDM scheme, is introduced to semi-discrete momentum equation of cell-face velocities. Thus, the second-order convergence of $u'$, $\bar{u}'$ and $p$ for the semi-discrete incompressible Navier-Stokes system based on a GFISDM scheme can be maintained if the error of the GFISDM scheme is convergent of order two at least. For simplicity, we mainly consider the GFISDM schemes on a Cartesian or a curvilinear structured grid system with the face interpolation functions $\eta$ following from (\ref{interpolation1}) in the following discussion.

To examine the accuracy and convergence of the GFISDM schemes without being distracted by the complexity of the mathematics, we first consider a one-dimensional problem stated as follows: given two continuous and sufficiently differentiable functions $\phi(\zeta)$, $\lambda(\zeta):\mathbb{R}\rightarrow\mathbb{R}$ defined on a curvilinear axis $\zeta$ and their values at a set of $k$ discrete points on the $\zeta$-axis with the increasing coordinates $\zeta_1$, $\zeta_2$, $\ldots$, $\zeta_k$ where $k \geq 2$. We further assume that the distance between any two adjacent points is proportional to $\varDelta \zeta$. The task of this problem is to approximate the value of $\lambda(\zeta)\phi(\zeta)$ at a point namely $e$ with the coordinate $\zeta_e \in [\zeta_1,\zeta_k]$. Apparently, $\lambda(\zeta_e)\phi(\zeta_e)$ can be approximated by a $k$-th order Lagrange interpolation formula in the form of:

\begin{equation} \label{method1}
\lambda_e \phi_e = \sum_{j=1}^k \lambda(\zeta_j)\phi(\zeta_j) l_j(\zeta_e),
\end{equation}

\noindent where $l_j(\zeta)$ is the Lagrange basis polynomials. If we assume the value of $\phi(\zeta_e)$ is given as well, we can approximate $\lambda(\zeta_e)\phi(\zeta_e)$ alternately using the expression

\begin{equation} \label{method2}
\lambda_e \phi_e = \phi(\zeta_e)\sum_{j=1}^k \lambda(\zeta_j)l_j(\zeta_e).
\end{equation}

\noindent Both approximations satisfy $|\lambda(\zeta_e)\phi(\zeta_e)-\lambda_e \phi_e| = O(\varDelta \zeta^k)$. However, additional constraints are required for them to be convergent of order $k$ (i.e., the approximation $\lambda_e \phi_e$ converges to $\lambda(\zeta_e)\phi(\zeta_e)$ with a $k$-th order rate as $\varDelta \zeta$ approaches to zero). More specifically, the $k$-th order convergence of the approximation (\ref{method1}) relies on the boundedness of the $k$-th order derivatives of $\lambda(\zeta)\phi(\zeta)$ for sufficiently small $\varDelta \zeta$. Meanwhile, for (\ref{method2}) to be convergent of order $k$, the $k$-th order derivatives of $\lambda(\zeta)$ should be bounded for sufficiently small $\varDelta \zeta$. Now, we have finished the discussion of this simplified problem and let's discuss what can be learnt from it. 

To start with, we would like to show that the interpolations involved in the GFISDM schemes on a Cartesian or a curvilinear structured grid system are identical to the simplified problem discussed above. For those two types of grid systems, the interpolation of a cell-centered variable field to a particular cell-face (let's say face $i$) employs the values of this variable at the adjacent downstream and upstream cells. Apparently, the centroids of the cells involved in the interpolation can be used to establish a Cartesian or a curvilinear axis $\zeta$ (see Figure \ref{pseudoCell}). Now, consider a pseudo control volume cell with its centroid moving along the $\zeta$ axis within the interval $X_i = [\zeta_{FD}, \zeta_{FU}]$ where $\zeta_{FD}$ and $\zeta_{FU}$ are the coordinates of the centroids of the farthest downstream cell (denoted by $FD$) and upstream cell (denoted by $FU$) on the $\zeta$ axis (again see Figure \ref{pseudoCell}). We also let the boundaries of the pseudo cell vary smoothly and accordingly such that they coincide with the boundaries of the real cells when the centroid of the pseudo cell is located at the centroids of the corresponding real cells. If we discretize the incompressible Navier-Stokes equation on the pseudo cell with an finite volume method, we can have the discretized diffusion and convection terms of this cell expressed as

\begin{figure}
\centering
\includegraphics[width=0.7\linewidth]{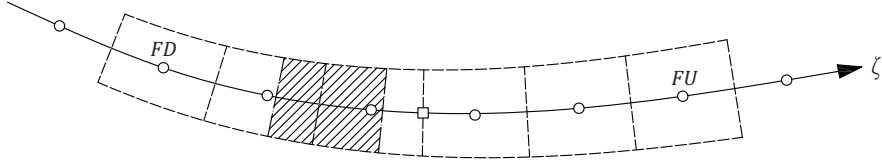} 
\caption{Illustration of the pseudo finite volume cell moving along the $\zeta$ axis in which the patched region denotes the pseudo cell, the dash lines denote the mesh grid lines and the symbols $\bigcirc$ and $\Box$ denote the centroids of the cells and cell-face are respectively}\label{pseudoCell}
\end{figure}

\begin{equation} \label{continous}
\big(\tilde{A}_{i,K}\left(\zeta\right)+\tilde{A}_{i,N}\left(\zeta\right)\big)\tilde{u}(\zeta)+\tilde{H}_{i,K}\left(\zeta\right)+\tilde{H}_{i,N}\left(\zeta\right),
\end{equation}

\noindent where $\tilde{A}_{i,K}(\zeta)$, $\tilde{A}_{i,N}(\zeta):X_i\rightarrow \mathbb{R}$ are the functions of the disretized coefficients regarding the velocity of the pseudo cell resulting from the diffusion and convection terms, respectively. $\tilde{H}_{i,K}(\zeta)$, $\tilde{H}_{i,N}(\zeta):X_i\rightarrow \mathbb{R}^3$ denotes the function collecting the contribution of the velocities of neighbouring cells or boundaries to the diffusion and convection terms, respectively. $\tilde{u}(\zeta):X_i\rightarrow \mathbb{R}^3$ represents the value of the continuous velocity field $\boldsymbol{u}$ along the $\zeta$ axis. By doing so, we actually obtain the continuous counterpart of the discretized diffusion and convection terms $\{\bar{K}+\bar{N}\}\bar{u}+\bar{b}$. We can further rearrange the expression (\ref{continous}) into

\begin{subequations} \label{continousM}
\begin{gather}
\tilde{A}_i(\zeta)\tilde{u}(\zeta)+\tilde{H}_i\left(\zeta\right), \\
\intertext{where $\tilde{A}_i(\zeta):X_i\rightarrow \mathbb{R}$ and $\tilde{H}_i(\zeta):X_i\rightarrow \mathbb{R}^3$ are functions defined as}
\tilde{A}_i(\zeta) \coloneqq \alpha^i\tilde{A}_{i,K}\left(\zeta\right)+\beta^i\tilde{A}_{i,N}\left(\zeta\right), \quad \tilde{H}_i\left(\zeta\right) \coloneqq (1-\alpha^i)\tilde{A}_{i,K}\left(\zeta\right)\tilde{u}\left(\zeta\right)+(1-\beta^i)\tilde{A}_{i,N}\left(\zeta\right)\tilde{u}\left(\zeta\right)+\tilde{H}_{i,K}\left(\zeta\right)+\tilde{H}_{i,N}\left(\zeta\right).
\end{gather}
\end{subequations}

\noindent Apparently, the $\tilde{A}_i$ and $\tilde{H}_i$ functions can be regarded as the continuous counterpart of the discrete $A_i$ and $H_i$ fields of (\ref{daem}). Note that for each cell-face, the corresponding $\tilde{A}_{i,K}$, $\tilde{A}_{i,N}$, $\tilde{H}_{i,K}$, $\tilde{H}_{i,N}$, $\tilde{A}_i$, $\tilde{H}_i$ functions and their domain $X_i$ can be constructed in the same manner. In the following analysis, we also assume that the term (\ref{continous}) or its equivalent form (\ref{continousM}) is bounded and sufficiently differential on $X_i$ for $i=1(1)\bar{m}$. It can be shown that this assumption holds if the grid lines of the mesh are smooth enough. With the above notation and assumptions ready, we can now analyze the accuracy and convergence of the GFISDM schemes as the way we do for the simplified problem. 

Recall that the essential problems of the proposed momentum interpolation framework is the approximation of the face-centered $\{\bar{K}+\bar{N}\}\bar{u}+\bar{b}$ field which is equivalent to a vector collecting the values of (\ref{continous}) at the centroid of face $i$ for $i=1(1)\bar{m}$. In GFISDM schemes, we use the coefficients $\alpha^i$ and $\beta^i$ to divide the $i$-th element of the $\{\bar{K}+\bar{N}\}\bar{u}+\bar{b}$ field into the components $\bar{A}^i\bar{u}^{\boldsymbol{i}}$ and $\bar{H}^{\boldsymbol{i}}$ and then approximate them separately. The $\bar{A}^i\bar{u}^{\boldsymbol{i}}$ term (equivalent to the value of $\tilde{A}_i\tilde{u}$ at the centroid of face $i$) is approximated with $\eta^i(A_i)\bar{u}^{\boldsymbol{i}}$ which corresponds to the scheme of (\ref{method2}). Meanwhile, the $\bar{H}^{\boldsymbol{i}}$ term (equivalent to the value of $\tilde{H}_i$ at the centroid of face $i$) is approximated with $\eta^{\boldsymbol{i}}(H_i)$ which corresponds to the scheme of (\ref{method1}). Both approximations have the same order of accuracy if they utilize the same interpolation scheme for $\eta^i$ and if the $\tilde{A}_i$ and $\tilde{H}_i$ functions are continuous on $X_i$. Therefore, we conclude that a GFISDM scheme with a $k$-th order polynomial-based interpolation scheme for $\eta$ has the $k$-th order of local accuracy if the $\tilde{A}_i$ and $\tilde{H}_i$ functions are continuous on $X_i$ for $i=1(1)\bar{m}$.

Another lesson we learn from the simplified problem is that a $k$-th order local accuracy does not necessaries lead to a $k$-th order convergence. By extending the aforementioned requirements on the orders of convergence to the case of the GFISDM schemes, we conclude that the error in approximating the $\{\bar{K}+\bar{N}\}\bar{u}+\bar{b}$ field with a $k$-th order interpolation scheme will be convergent of order $k$, if and only if the $\tilde{A}_i$ and $\tilde{H}_i$ functions (for $i=1(1)\bar{m}$) and their $k$-th order spatial derivatives are bounded and sufficiently differentiable on their domain as the grid spacing approaches to zero. If those requirements are not satisfied, the error of the GFISDM schemes will be convergent of order less than $k$ which may in return damage the convergence results of the finite volume discretization. Recall we have assumed the $\tilde{A}_i\tilde{u}+\tilde{H}_i$ term is bounded and sufficiently differentiable on $X_i$ for all $i$. Thus, if the $\tilde{A}_i$ function is discontinuous or non-differentiable or unbounded on its domain, neither will be the $\tilde{H}_i$ function. Therefore, the accuracy and convergence of the GFISDM schemes largely depends on the smoothness and boundedness of the $\tilde{A}_i$ function. 

As discussed in Section \ref{section2.2}, isolating the main diagonals of the coefficient matrices $K$ and $N$ during the momentum interpolation is a procedure that has been utilized in almost all of the available momentum interpolation methods available. If we also adopt this strategy in a GFISDM scheme by setting $\alpha = \beta = \mathbbm{1}$, we end up with the scheme of GFISDM-Yu in which $A_i =A$ and $H_i = H$. Meanwhile, a GFISDM scheme with $\alpha$ and $\beta$ equalling to $0$ (write GFISDM-Z for this scheme hereafter) leads to $A_i = 0$ and $H_i = \{K+N\}u+b$. Note that the elements of the $\text{diag}(K)$ term are negative regardless of the numerical schemes used for the discretization of diffusion terms. If we further assume the discretization of convection terms employs a deferred correction approach \citep{khosla1974,moukalled2016}, the elements of the $\text{diag}(N)$ term will be non-positive. As a result, the elements of the $A_i$ term in GFISDM-Yu will have larger absolute values than the rest of GFISDM schemes. Potential benefits of doing this may include the faster convergence rate for the iterative solution of the discretized Navier-Stokes system, higher spatial accuracy and better numerical stability. However, this procedure can be a potential threat to the accuracy and convergence of the momentum interpolation. For example, let us consider a two-dimensional Taylor-Green vortex on a uniform non-staggered grid with a Dirichlet boundary condition for velocities and a deferred correction formulation of the FROMM scheme for the discretization of convection terms. We demonstrate the $\{K+N\}u+b$ and $A$ fields resulting from the finite volume discretization in Figure \ref{figureAU}. It is easy to see the $\{K+N\}u+b$ field is sufficiently smooth over the entire computational domain, while the $A$ field is apparently not. It can be shown that GFISDM-Z is able to attain the second order convergence in this example, while GFISDM-Z is not. To gain a deeper understanding of similar issues, we summarize two types of situations in which GFISDM schemes will suffer from the loss of accuracy and the order reduction.

\begin{figure}
\centering
    \begin{subfigure}[b]{0.45\linewidth}
        \includegraphics[width=\linewidth]{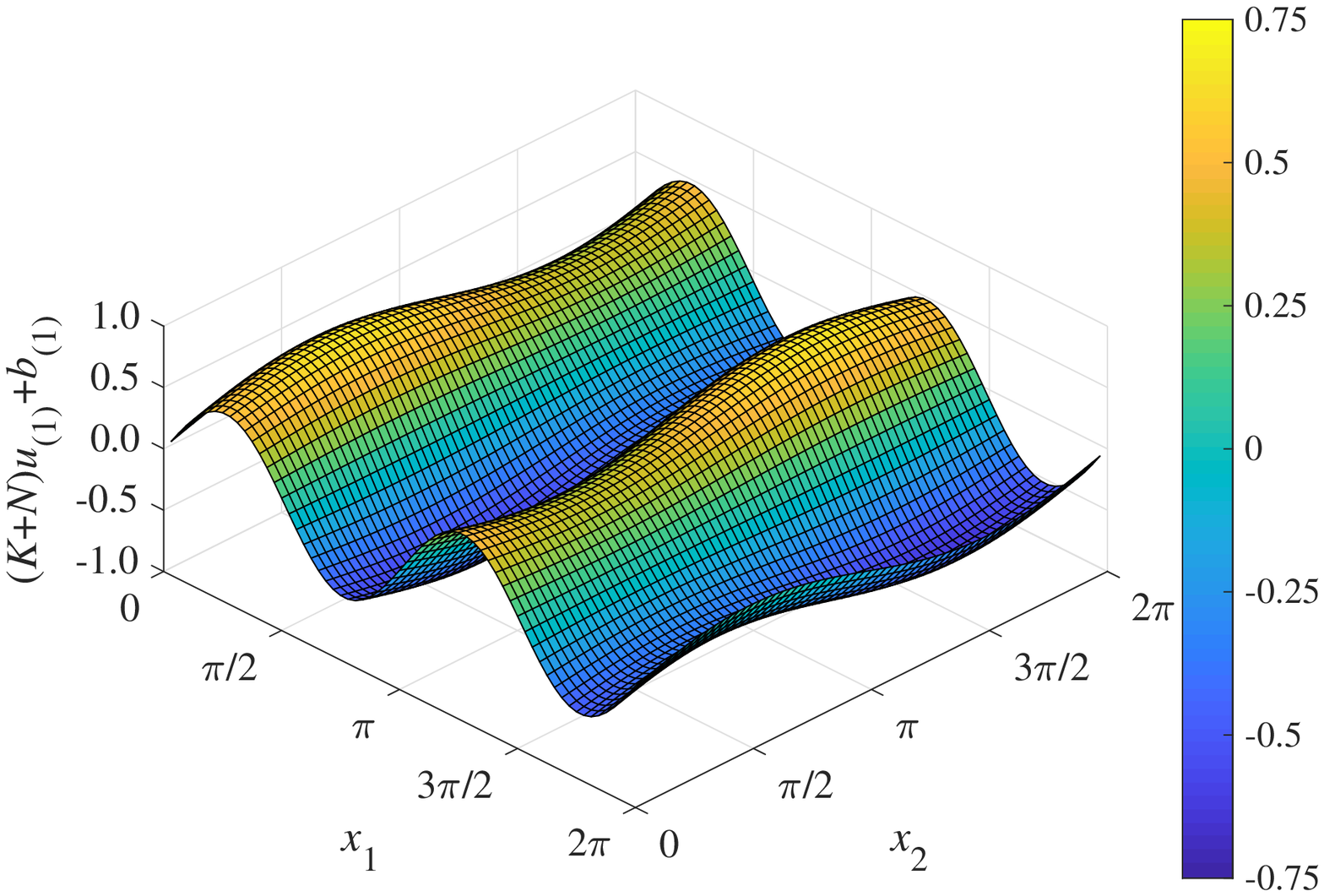}
        \caption{$(K+N)u_{(1)}+b_{(1)}$}
    \end{subfigure}
    \begin{subfigure}[b]{0.45\linewidth}
        \includegraphics[width=\linewidth]{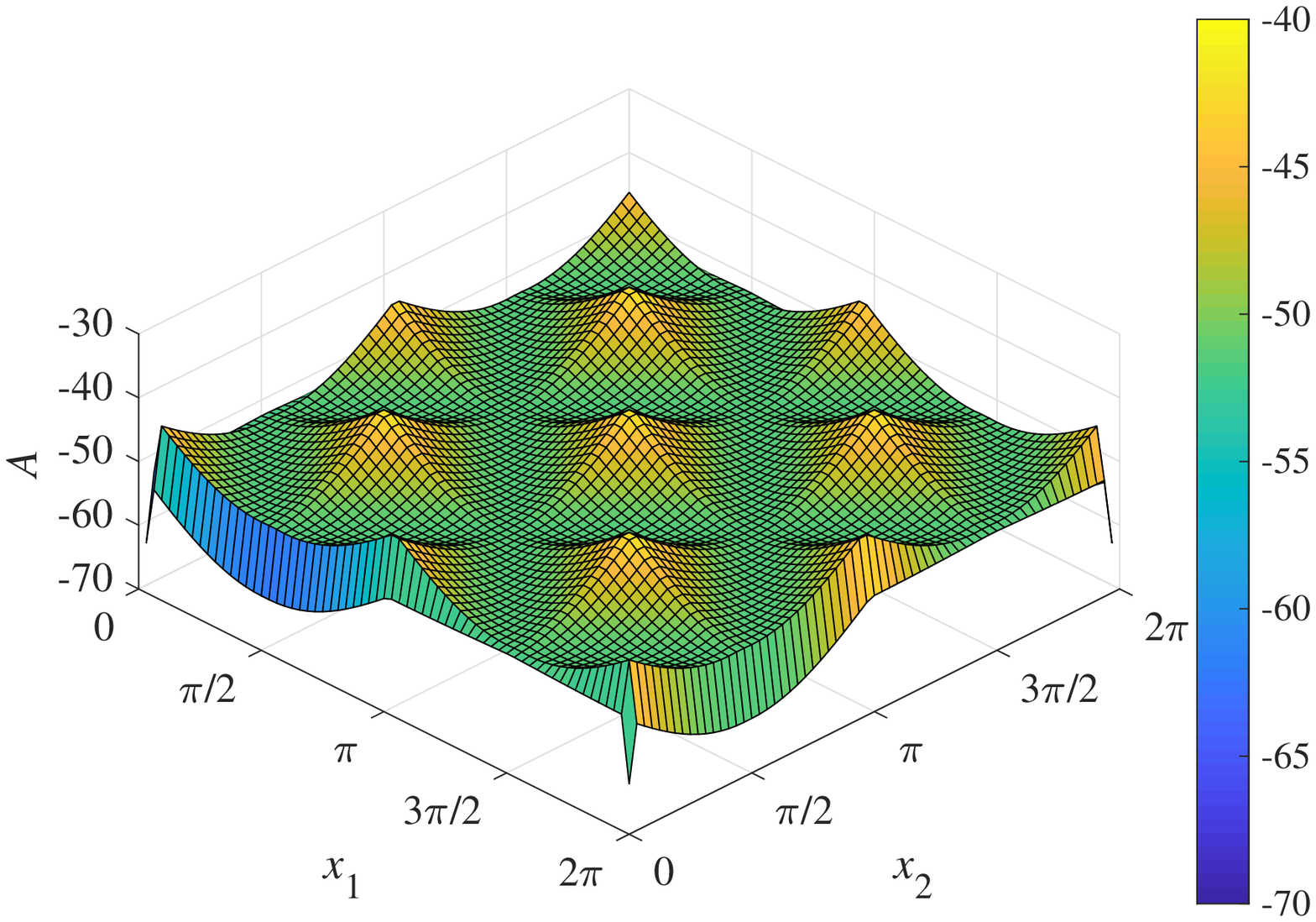}
        \caption{$A=\text{diag}(K+N)$}
     \end{subfigure}
      \caption{The cell-centered $(K+N)u_{(1)}+b_{(1)}$ and $A$ fields of a two dimensional Taylor-Green Vortex on a uniform grid with $64\times64$ control volumes, a Dirichlet boundary condition for velocities and a deferred correction formulation of the FROMM scheme for the discretization of convection terms} \label{figureAU}
\end{figure}

The first situation is the $\tilde{A}_i$ function is discontinuous or non-differentiable on its domain $X_i$. When it happens, the $\eta^i(A_i)$ and $\boldsymbol{\eta}^{\boldsymbol{i}}(H_i)$ terms cannot deliver accurate approximations for the $\bar{A}^i$ and $\bar{H}^{\boldsymbol{i}}$ terms. Consequently, the accuracy and convergence of the GFISDM schemes will be damaged regardless of how accurate the interpolation scheme used for $\eta$ is. Note that, unless both $\alpha^i$ and $\beta^i$ equal to zero (which leads to $\tilde{A}_i=0$), the smoothness of the $\tilde{A}_i$ function depend on that property of its two components i.e., the $\tilde{A}_{i,K}$ and $\tilde{A}_{i,N}$ functions. When one of the nodes required for the evaluations of $\eta^i(A_i)$ and $\boldsymbol{\eta}^{\boldsymbol{i}}(H_i)$ belongs to a boundary cell (i.e., the cell with at least one of its cell-face coincide with the boundary of the computational domain) and the velocity field employs a non-periodic type of boundary condition (e.g., the Dirichlet boundary condition or the Von Neumann boundary condition) at this boundary cell, both $\tilde{A}_{i,K}$ and $\tilde{A}_{i,N}$ are very likely to be discontinuous on $X_i$. Moreover, if an upwind-type of scheme or a deferred correction approach is employed for the discretization of convection terms, the corresponding $\tilde{A}_{i,N}$ function depends on the sign of the $\zeta$-axis component of $\tilde{u}$, and it will be non-differentiable if the $\zeta$-axis component of $\tilde{u}$ change the sign on $X_i$.

The second situation is that the spatial derivatives of the $\tilde{A}_i$ or $\tilde{H}_i$ function are unbounded while the order of the interpolation scheme used for $\eta$ is not high enough. Recall we assumed the expression $\tilde{A}_i\tilde{u}+\tilde{H}_i$, as an integral term, is bounded on $X_i$. However, its two components, i.e., the $\tilde{A}_i\tilde{u}$ and $\tilde{H}_i$ terms, may not be bounded as the grid spacing approaches to zero, since both the $\tilde{A}_{i,K}$ and $\tilde{H}_{i,N}$ functions increase with the grid spacing decreasing. Furthermore, if we, without loss of generosity, assume that the grid spacing is proportional to a particular value, let's say $\varDelta \zeta$. It can be shown that, in this case, $\tilde{A}_{i,K}$ and $\tilde{A}_{i,N}$ are functions of $\varDelta \zeta^{-2}$ and $\varDelta \zeta^{-1}$, respectively. Consequently, the convergence orders of the $\eta^i(A_i)$ and $\boldsymbol{\eta}^{\boldsymbol{i}}(H_i)$ terms are actually lower than the order of accuracy of the interpolation scheme used for $\eta$ unless both $\alpha^i$ and $\beta^i$ equals to zero. More specifically, the $\eta^i(A_i)$ and $\boldsymbol{\eta}^{\boldsymbol{i}}(H_i)$ terms can only show $(k-2)$-th order of convergence if $\alpha^i>0$ (where $k$ is the order of accuracy of the interpolation scheme for $\eta$). Meanwhile, when $\alpha^i=0$ and $\beta^i>0$, the $\eta^i(A_i)$ and $\boldsymbol{\eta}^{\boldsymbol{i}}(H_i)$ terms can only be convergent of order $(k-1)$.

The above discussion indicates that GFISDM-Z turns out to be the easiest and safest choice to attain the second-order convergence in space. It approximate the face-centered $\{\bar{K}+\bar{N}\}\bar{u}+\bar{b}$ field through the direct face interpolation of the cell-centered $\{K+N\}u+b$ field, and neither of the aforementioned situations occurs in this scheme as along as the term (\ref{continous}) is differentiable on $X_i$ for all $i$. Consequently, GFISDM-Z can deliver the second order convergence if the finite volume discretization of the momentum equation itself is convergent of order two and an interpolation scheme with order of accuracy no less than two is utilized in it.

In comparison, it is much more difficult to attain the second order convergence for the GFISDM schemes with $\alpha$ or $\beta$ not being $0$. For this type of schemes to achieve this goal, the first priority is to compensate order reduction caused by the unboundedness of $\tilde{A}_i$ and $\tilde{H}_i$ using higher order interpolation schemes. If $\alpha \neq 0$, the most straightforward approach is to employ a fourth-order interpolation scheme universally for the face interpolations of the $A_i$ and $H_i$ fields. If $\alpha = 0$ while $\beta \neq 0$, a third-order scheme is sufficient. Instead of using the same interpolation scheme universally, we can also specify the interpolation schemes locally according to the requirement on their orders of accuracy so that $\bar{A}^i$ and $\bar{H}^{\boldsymbol{i}}$ can be convergent of order two. In this case, the coefficients $l_{ij}$ on the each row of $L$ are determined by different interpolation schemes accordingly. Note that in the GFISDM scheme given by (\ref{daem}), we treat either the $A_i$ or $H_i$ field as an integral part during the interpolation. Alternatively, we can divide both the $A_i$ and $H_i$ fields into two parts, i.e., the parts from diffusion and convection terms, respectively. Then, we conduct the face interpolation of those parts using the interpolation schemes with the minimum orders of accuracy required, and add up the results afterwards. This approach corresponds to a modified GFISDM scheme with (\ref{daem}) replaced by

\begin{subequations} \label{modifieddaem}
\begin{gather}
\bar{A}^i = \eta^i(A_{i,K})+\eta^i(A_{i,N}),\quad \bar{H}^{\boldsymbol{i}}= \boldsymbol{\eta}^{\boldsymbol{i}}(H_{i,K})+\boldsymbol{\eta}^{\boldsymbol{i}}(H_{i,N}), \\
\intertext{where}
A_{i,K} \coloneqq \alpha^i\text{diag}(K), \quad A_{i,N} \coloneqq \beta^i\text{diag}(N), \quad H_{i,K} \coloneqq \{K-\text{diag}(A_{i,K})\}u+b_K, \quad H_{i,N} \coloneqq \{N-\text{diag}(A_{i,N})\}u+b_N.
\end{gather}
\end{subequations}

\noindent $b_K$ and $b_N\in \mathbb{R}^{m}$ are the components of $b$ arising from the discretized diffusion and convection fluxes at the boundary, respectively. Each $\eta^i$ in the above equation is then determined by a interpolation scheme with the order of accuracy no less than the requirement given in Table \ref{orderofinterpolation}. It is also worth mentioning that implementing a higher-order interpolation scheme is quite challenging on a finite volume mesh especially for unstructured grids. Even for structured grids, it is still difficult to find enough data points to perform the interpolation at the location adjacent to the boundary. If this type of situation happens during the evaluation of $\bar{A}^i$ and $\bar{H}^{\boldsymbol{i}}$, we can simply set $\alpha^i$ or $\beta^i$ or both to zero, so that a linear interpolation scheme (which only use two data points) is sufficient. 

The next priority is to make sure the $A_i$ function is sufficiently differentiable on its domain for $i=1(1)\bar{m}$. This goal can be achieved easily by setting $\alpha^i$ to zero if $\tilde{A}_{i,K}$ is discontinuous or non-differentiable on $X_i$; and setting $\beta^i$ to zero if $\tilde{A}_{i,N}$ is discontinuous or non-differentiable on $X_i$. Note that the $\tilde{A}_{i,N}$ function is time-dependent for unsteady flow problems. In practice, it is very likely that the $\tilde{A}_{i,N}$ function is sufficiently differentiable within a particular time interval e.g. $[t_0,t_k)$, while it is not for the time interval $(t_k,t_n]$ (where $0<k<n$). In this case, we can not just set $\alpha^i(t)=1$ and $0$ for $t\in[t_0,t_k)$ and $(t_k,t_n]$, respectively, since $\alpha^i$ will be discontinuous and non-differentiable at the time instant $t_k$ by doing so. Also note that if $\alpha$ and $\beta$ are time-varying, they should be sufficiently differentiable in time for any applied time-integrator to attain its designed order of accuracy. Therefore, if we insist on utilizing time-varying $\alpha^i$ coefficient, we should let it varies smoothly between 1 and 0. Apparently, this approach is only applicable if we know or can somehow estimate the time instants at which $\tilde{A}_{i,N}$ becomes non-differentiable. A more realistic approach is to use a constant value for $\alpha^i$ and set it to zero if $\tilde{A}_{i,N}$ is likely to be non-differentiable in the whole solution time domain. Based on the above requirements, we develop a GFISDM scheme which employs (\ref{modifieddaem}) and constant $\alpha$, $\beta$ coefficients such that the properties of GFISDM-Yu can be preserved as much as possible without suffering from the order reduction.

\begin{table}
\centering
\begin{tabular}{*7c}
\toprule
Cases & $A_i$ & $A_{i,K}$ & $A_{i,N}$ & $H_i$ & $H_{i,K}$ & $H_{i,N}$ \\
\midrule
$\alpha^i>0,\ \beta^i>0$ & 4 & 4 & 3 & 4 & 4 & 3 \\
$\alpha^i>0,\ \beta^i=0$ & 4 & 4 & - & 4 & 4 & 2 \\
$\alpha^i=0,\ \beta^i>0$ & 3 & - & 3 & 3 & 2 & 3 \\
$\alpha^i=0,\ \beta^i=0$ & 2 & 2 & 2 & 2 & 2 & 2 \\
\bottomrule
\end{tabular}
\caption{Minimum requirements on the orders of accuracy of the interpolation schemes for different terms in GFISDM to be convergent of order 2}
\label{orderofinterpolation}
\end{table}

Now, we have finished the discussion on the convergence results of the GFISDM schemes on structured Cartesian or curvilinear gird systems with $\eta$ in the form of (\ref{interpolation1}). It is reasonable to assume the discontinuity, non-smoothness or unboundedness of the cell-centered field that interpolated will also damage the accuracy and convergence of the momentum interpolation in other methods available in literature as well. The order of spatial convergence is not all of the game in a finite volume method. It is however easier to summarize and more difficult to obtain than other aspects. The Taylor-Green Vortex we mentioned earlier can be used to examine the convergence of the spatial errors. As a matter of fact, such an examination has always been included in studies which utilize a finite differencing method \citep{kim1985, le1991} or a finite volume method on a staggered grid \citep{sanderse12}. However, the convergence results of the spatial error was seldom presented in the studies of the finite volume method on a collocated grid, and for those that presented, the boundary conditions and the schemes employed for the discretization of convection terms were not reported. Instead, we will present a detailed verification on the accuracy and convergence of the spatial errors for the semi-discrete incompressible Navier-Stokes system based on GFISDM-Z and GFISDM-H later in Section \ref{section5}. Here, we shall continue the discussion on the solution of the semi-discrete incompressible Navier-Stokes system.

\section{Runge-Kutta Methods for Differential-Algebraic Problems} \label{section3}

In this section, we first briefly discuss how the standard direct approach employs an implicit Runge-Kutta scheme for solution of the semi-discrete incompressible Navier-Stokes system, the corresponding convergence results and order conditions. Subsequently, we propose a new way of applying the implicit Runge-Kutta schemes which aims at improving the orders of convergence for pressures and enhancing the computational efficiency The convergence result of the proposed method is analyzed mathematically afterwards. Finally, a discussion on the low-storage implementation of stiff-accurate diagonal implicit Runge-Kutta schemes in the proposed method and the determination of Runge-Kutta coefficients is presented.

\subsection{The direct approach of applying implicit Runge-Kutta schemes}

The solution of the semi-discrete incompressible Naiver Stokes system (\ref{system}) can be analysed by considering an equivalent index 2 DAE problem regarding $y(t)\in \mathbb{R}^{m_y}$ and $z(t)\in \mathbb{R}^{m_z}$ specified as follows:

\begin{equation} \label{I2DAE}
y' =  f(y,z), \quad gy = r(t),
\end{equation}

\noindent where $f:\mathbb{R}^{m_y}\times\mathbb{R}^{m_z}\rightarrow\mathbb{R}^{m_y}$ is a function of $y$ and $z$, and $g \in \mathbb{R}^{m_z \times m_y}$ is matrix of constant coefficients. The integers $m_y$ and $m_z$ denote the dimensions of $y$ and $z$, respectively. In order that ($\ref{I2DAE}$) is an index 2 problem, we assume that $\norm{(gf_z(y,z))^{-1}}$ (where $f_z(y,z)$ is the partial derivative of $f$ with respect to $z$) exists and is bounded in a neighbourhood of the solution. We further assume that $f$ and $r$ are sufficiently differentiable, and the initial values $(y_0,z_0)$ satisfy the consistency condition

\begin{equation} \label{consistency}
gy_0 = r(t_0), \quad gf(y_0,z_0) = r'(t_0).
\end{equation}

\noindent For the approximation of $(y_{n+1},z_{n+1})$ from an approximation $(y_n,z_n)$ at $t_n$ using an $s$-stage implicit Runge-Kutta (IRK) scheme with a non-singular coefficient matrix, the direct approach gives

\begin{subequations} \label{originalRK}
\begin{gather}
y_{n+1} = y_n + \textstyle\sum_{i,j=1}^s b_i \omega_{ij}(Y_{nj}-y_n), \quad z_{n+1} = z_n + \textstyle\sum_{i,j=1}^s b_i \omega_{ij}(Z_{nj}-z_n), \label{originalRKfinalStage} \\  
\intertext{with the Runge-Kutta internal stages for $i=1(1)s$ given by}
Y_{ni}  = y_n + h\textstyle\sum_{j=1}^s a_{ij} f(Y_{nj}, Z_{nj}), \quad gY_{ni}  = r(t_n+c_ih), \label{originalRKinternalStage}
\end{gather} 
\end{subequations}

\noindent where $a_{ij}$, $b_i$ and $c_i$ are coefficients that characterize a Runge-Kutta scheme, and $\omega_{ij}$ denote an element of the inverse the matrix of coefficients $a_{ij}$. Runge-Kutta coefficients are usually displayed in a tableau of the form:

\begin{equation} \label{RKtableau}
\begin{array}{ l | l  l  l  l }		
  c_1 & a_{11} & a_{12} & \ldots & a_{1s} \\
  c_2 & a_{21} & a_{22} & \ldots & a_{2s} \\
  \vdots & \vdots  & \vdots & \ddots & \vdots \\
  c_s & a_{s1} & a_{s2} & \ldots & a_{ss} \\ \cline{1-5}
  & b_1 & b_2 & \ldots & b_s \\
\end{array}
\end{equation}

\noindent  Hereafter for simplicity, we write $t_{ni}$ for $t_n+c_i h$, $A \in \mathbb{R}^{s\times s}$ for the matrix of elements $a_{ij}$, $b \in \mathbb{R}^s$ for the vector of $b_i$ and $c \in \mathbb{R}^s$ for the vector of $c_i$ in this section. In is noted that the symbols $A$ and $b$ in this section are different from those in the rest of the sections.

The implementation of an $s$-stage IRK scheme equipped with a full matrix $A$ to a differential-algebraic system of $m_y$ differential equations and $m_z$ algebraic equations requires the simultaneous solution of $(m_y+m_z)\times s$ implicit equations at every time step. However, for schemes equipped with a lower triangular coefficient matrix, the system (\ref{originalRKinternalStage}) can be readily solved in $s$ successive stages with $(m_y+m_z)$ implicit equations to be solved at each stage. Such a scheme has a significant advantage in computational efficiency compared to the one with a full matrix $A$, and is called the diagonally implicit Runge-Kutta (DIRK) scheme. If, in addition, the elements of the main diagonal of $A$ are identical, we speak of an singly diagonally implicit Runge-Kutta (SDIRK) scheme. 

An important reason we prefer the IRK schemes with invertible coefficient matrices over the ESDIRK schemes or the half-explicit methods for the semi-discrete incompressible Naver-Stokes system considered in this study is because the non-singularity of $A$ suggests that the first relation of (\ref{originalRKinternalStage}) can be replaced by

\begin{equation}  \label{originalRKinternalStageModified}
Y_{ni} = y_n + \textstyle\sum_{\substack{j,k=1 \\ j \neq i}}^{s} a_{ij} \omega_{jk}(Y_{nk} - y_n) + ha_{ii} f(Y_{ni},Z_{ni}).
\end{equation}

\noindent Considering the second term on the right hand side (RHS) of (\ref{originalRKinternalStageModified}) is easy to compute, such a modification avoids the repeat evaluation of $f(y,z)$ at every Runge-Kutta internal stage which can be rather difficult and time-consuming if the expression of $f$ cannot be explicitly given easily or the system of equations is high-dimensional.  When it comes to the semi-discrete incompressible Navier-Stokes equations considered in this study, an implicit Runge-Kutta scheme with (\ref{originalRKinternalStage}) replaced by (\ref{originalRKinternalStageModified}) clearly avoids the interpolations involved the momentum interpolation to the most extent. Furthermore, we only consider the diagonal implicit schemes in this study since the solutions of the non-linear Runge-Kutta internal stage equations given by a fully-implicit scheme is way more difficult.

\subsection{Conditions on orders of convergence for the direct approach}

Before we discuss the order results of the direct approach (\ref{originalRK}) for the solution of (\ref{I2DAE}), we briefly introduce some basic properties of a Runge-Kutta scheme that affect its order of convergence when applied to a differential-algebraic problem first. To start with, we use the symbols $B$ and $C$, in agreement with \cite{butcher1964}, to denote a Runge-Kutta scheme whose coefficients satisfy:

\begin{equation}
\begin{split}
B(\xi): &\textstyle\sum_{i=1}^s b_i c_i^{k-1} = \frac{1}{k}, \quad k = 1(1)\xi, \\
C(\xi): &\textstyle\sum_{i=1}^s a_{ij} c_j^{k-1} = \frac{c_i^k}{k}, \quad i=1(1)s,\ k = 1(1)\xi.
\end{split}
\end{equation}

\noindent Recall that we say a Runge-Kutta scheme is of classical order $p$ if, for a sufficiently smooth problem $y' =f(t,y)$, the estimate $\norm{y(t_n+h)-y_{n+1}} \leq Ch^{p+1}$ holds. It should be addressed that throughout this section the variable $p$ denotes an integer rather than the pressure. If we use the symbol $\rho(p)$ to denote the order conditions for a Runge-Kutta scheme to be of classical order $p$ in addition to the conditions $\rho(r)$ with $r=0(1)p-1$ (where $\rho(0) = \emptyset$), the order conditions for a Runge-Kutta scheme to have classical order up to 3 are given by

\begin{equation} \label{classicalOrder3}
\begin{split}
\rho(1):&\ \textstyle\sum_i^s b_i = 1, \\
\rho(2):&\ \textstyle\sum_i^s b_i c_i = 1/2, \\
\rho(3):&\ \textstyle\sum_i^s b_i c_i^2 = 1/3, \quad \sum_{i,j}^s b_i a_{ij} c_j = 1/6.
\end{split}
\end{equation}

\noindent The stage order of a Runge-Kutta scheme, on the other hand, is the largest value of $q$ such that both $B(q)$ and $C(q)$ are satisfied. As its name implies, the stage order is related to the order of accuracy of the Runge-Kutta internal stage values, and it has a significant impact on the accuracy of the method when applied to differential-algebraic problems.

Another property for an implicit Runge-Kutta scheme with non-singular $A$ which plays a decisive role in the order of convergence is the limit of its stability function at $\infty$ (the definition of the stability function can be found in \cite{wanner1991}:

\begin{equation}\label{Rinfty}
R(\infty) = 1-b^TA^{-1}\mathbbm{1}.
\end{equation}

\noindent In a general non-linear index 2 problem, the condition $|R(\infty)|<1$ is required to ensure the boundedness of the global error of the implicit Runge-Kutta scheme. Note that if the coefficients of Runge-Kutta schemes satisfy for $i=1(1)s$

\begin{equation} \label{stiffaccurate}
b_i = a_{si},
\end{equation}

\noindent we will have $|R(\infty)| = 0$. Such schemes are known as the stiff-accurate methods in literatures. A direct implementation of (\ref{stiffaccurate}) leads to $Y_{ns} = y_{n+1}$ and $Z_{ns} = z_{n+1}$ which turns out be a favourable property for the solution of differential-algebraic problems. Moreover, it can be shown that $z_{n+1}$ does not depend on $z_n$ in still-accurate Runge-Kutta schemes. This statement can be validated by rewriting the second-relation of (\ref{originalRKfinalStage}) as

\begin{equation} \label{zcomponent}
z_{n+1} = R(\infty) z_n + \textstyle\sum_{i,j=1}^s b_i \omega_{ij}Z_{nj}.
\end{equation}

\noindent Inserting $R(\infty) = 0$ to the above relation and considering the fact that $Z_{ni}$ are independent of $z_n$, the in-dependency of $z_{n+1}$ on $z_n$ is now straightforward. The feature that the algebraic component obtained within the current time-step does not depend on its information at previous time-steps is very attractive especially considering the initial values of $z$ may not necessarily satisfy the second relation of the consistency constraint (\ref{consistency}). Finally, it is good to know that a method is said to be convergent of order $p$ if the error, i.e., the difference between the exact and the numerical solution, is bounded by $Const \cdot  h^p$ uniformly on bounded intervals for sufficiently small step sizes $h$. 

Similar to the way we define the order conditions of classical order, we use the symbol $\rho_y(p)$ to denote the order conditions for the $y$-component to show order of convergence $r$ in addition to the conditions $\rho(r)$ with $r=0(1)p$ and $\rho_y(r)$ with $r=0(1)p-1$ (where $\rho_y(0)=\emptyset$), and similarly for $z$-component. Based on the discussions on index 2 differential-algebraic problems presented in \cite{wanner1991}, the conditions for $y$-component in (\ref{originalRK}) to be convergent of order 3, in addition to (\ref{classicalOrder3}) and $|R(\infty)| <1$,  are

\begin{equation} \label{conditionY}
\begin{split}
\rho_y(1)&:\ \emptyset, \\
\rho_y(2)&:\ \textstyle\sum_{i,j}^s b_i \omega_{ij} c_j^2 = 1, \\
\rho_y(3)&:\ \textstyle\sum_{i,j}^s b_i \omega_{ij} c_j^3 = 1.
\end{split}
\end{equation}

\noindent In the meantime, the condition for the $z$-component to show a second-order convergence is

\begin{equation} \label{conditionZ}
\begin{split}
\rho_z(1)&:\ \emptyset, \\
\rho_z(2)&:\ \textstyle\sum_{i,j,k}^s b_i \omega_{ij} \omega_{jk} c_k^2 = 2.
\end{split}
\end{equation}

\noindent It can be easily shown that $\rho_y(2)$ and $\rho_y(3)$ reduce to $c_s^2 =1$ and $c_s^3 =1$ respectively. These two conditions hold for all stiff-accurate implicit Runge-Kutta schemes and consequently for all stiff accurate DIRK schemes as well. However, it is impossible to construct a DIRK scheme with non-singular $A$ satisfying $\rho_z(2)$. As a result, the $z$-component of a DIRK method given by (\ref{originalRK}) is first-order accurate only.

\subsection{A new way of applying implicit Runge-Kutta schemes} \label{section3.3}

The algebraic component in (\ref{I2DAE}) corresponds to the pressure in the semi-discrete incompressible Navier-Stokes system. Considering discrete pressure fields of higher-order temporal accuracy are of interest in many CFD applications, the feature that a DIRK scheme implemented with (\ref{originalRK}) delivering first-order temporal-accurate pressures only is very unpleasant. This motivate us to investigate the possibility of improving the order of convergence for the algebraic component in (\ref{I2DAE}) through the modifications regarding the standard direct approach (\ref{originalRK}). In our attempt to do so, we consider an implicit Runge-Kutta method that combines the benefit of (\ref{originalRKinternalStageModified}) in numerical efficiency with a new way of applying algebraic constraints. For the time-marching of (\ref{I2DAE}) from $t_n$ to $t_{n+1}$, the proposed method yields

\begin{subequations}\label{RKformulation}
\begin{gather}
y_{n+1} = y_n + \textstyle\sum_{i,j=1}^s b_i \omega_{ij}(Y_{nj}-y_n), \quad z_{n+1} = z_n + \textstyle\sum_{i,j=1}^s b_i \omega_{ij}(Z_{nj}-z_n),  \label{RKfinalstage} \\ 
\intertext{with the Runge-Kutta internal stages for $i=1(1)s$ given by}
\begin{split}
Y_{ni} &= y_n + \textstyle\sum_{\substack{j,k=1 \\ j \neq i}}^{s} a_{ij} \omega_{jk}(Y_{nk} - y_n) + h a_{ii} f(Y_{ni},Z_{ni}), \\
g Y_{ni} &= r(t_n) + h\textstyle\sum_{j=1}^s a_{ij} \big(r'(t_{ni}) + \theta_{i}\big),
\end{split} \label{RKinternalstage}
\end{gather} 
\end{subequations}

\noindent where $\theta_{i}$ are functions of $t_n$ and $h$, and satisfies

\begin{equation} \label{RKConstraint}
r(t_n+h) = r(t_{n}) + h\textstyle\sum_{i=1}^s b_{i}\big( r'(t_{ni})+\theta_i \big).
\end{equation}

\noindent The major difference between the direct approach and the proposed method, in addition to replacing the first relation of (\ref{originalRKinternalStage}) by (\ref{originalRKinternalStageModified}), relates to the specification of the source term of the continuity equation at the Runge-Kutta internal stage. Instead of taking the explicitly prescribed value $r(t_{ni})$, the source term of the continuity equation now consists of two components: a Runge-Kutta approximation of the $r(t_{ni})$ term from the starting value $r(t_n)$ (i.e., $r(t_n) + h\sum_{j=1}^s a_{ij} r'(t_{ni})$) and a pre-determined perturbation term (i.e., $\theta_{i}$) which ensures the satisfaction of the equality $g y_{n+1} = r(t_{n+1})$ at the new time-step through the constraint (\ref{RKConstraint}). For the convergence of the method (\ref{RKformulation}), we state the following theorem:

\begin{theorem} \label{theorem1}

Consider the problem (\ref{I2DAE}) and suppose that the initial conditions are consistent and that $(gf_z)^{-1}(y,z)$ exists and is bounded in a neighbourhood of the solution. Let $(y_n,z_n)$ be given by the implicit Runge-Kutta method (\ref{RKformulation}) of classical order $p$ and stage order $q$ with $p \geq q+1$, having an invertible coefficient matrix $A=(a_{ij})$ and satisfying $|R(\infty)| < 1$ where $R(\infty)$ is given by (\ref{Rinfty}).  If the errors in solving the first relation of ($\ref{RKinternalstage}$) are of $O(h^{p+1})$, then the global errors for $y$- and $z$-components satisfy

\begin{equation}
y_n-y(t_n) = O(h^{p}), \quad z_n-z(t_n) = O(h^{q+1}), \nonumber
\end{equation}

\noindent for $h\leq h_0$, $t_n = nh \leq C$ and $\theta_{i} = O(h^{p})$. If $a_{si} = b_i$ for $i=1(1)s$ in addition, the global error for the $z$-component can be sharpened with

\begin{equation}
z_n-z(t_n) = O(h^{p}). \nonumber
\end{equation}

\end{theorem}

\begin{remark}
Replacing the first relation of (\ref{RKinternalstage}) with that of (\ref{originalRKinternalStage}) doest not change the results of the theorem.
\end{remark}

\begin{proof}

Given that the errors in solving the first relation of (\ref{RKinternalstage}) are $O(h^{p+1})$, we have the numeral solution satisfying for $i=1(1)s$

\begin{equation} \label{numericalSolution}
\begin{split}
Y_{ni} &= y_n + \textstyle\sum_{\substack{j,k=1 \\ j \neq i}}^{s} a_{ij} \omega_{jk}(Y_{nk} - y_n) + ha_{ii}\big( f(Y_{ni},Z_{ni}) + \delta_{i}\big), \\
g Y_{ni} &= r(t_n) + h\textstyle\sum_{j=1}^s a_{ij} \big(r'(t_{ni})+\theta_{i}\big),
\end{split}
\end{equation}

\noindent where both $\delta_i$ and $\theta_i$ are of $O(h^{p})$. If we introduce the notation $\delta \coloneqq (\delta_{1},\ldots,\delta_{s})^T$, $\theta \coloneqq (\theta_{1},\ldots,\theta_{s})^T$,

\begin{equation}
Y_n  \coloneqq \big(Y_{n1},\ldots,Y_{ns}\big)^T, \quad Y'_n  \coloneqq \big(f(Y_{n1},Z_{n1}),\ldots,f(Y_{ns}, Z_{ns})\big)^T, \quad r'_n  \coloneqq \big(r'(t_{n1}),\ldots,r'(t_{ns})\big)^T, \nonumber
\end{equation}

\noindent and $\hat{A} \coloneqq \text{diag}(\text{diag}(A))$, the relations in (\ref{numericalSolution}) can be replaced by

\begin{equation} \label{numericalSolution2}
\begin{split}
Y_n &= \mathbbm{1} \otimes y_n + \big((A-\hat{A})A^{-1} \otimes I\big)\big(Y_n - \mathbbm{1} \otimes y_n\big) + h\big(\hat{A} \otimes I\big)\big(Y'_n +\delta\big), \\
\big\{g\big\}Y_n &= \mathbbm{1} \otimes r(t_n) + h\big(A\otimes I\big)(r'_n + \theta),
\end{split}
\end{equation}

\noindent where $\otimes$ denotes the Kronecker product. By deduction, the first relation of (\ref{numericalSolution2}) can further be rewritten as

\begin{equation} \label{csolution1}
Y_n = \mathbbm{1} \otimes y_n + h\big(A \otimes I\big)\big(Y'_n +\delta\big).
\end{equation}

\noindent Multiplying $\{g\}$ to the both sides of the above equation and using the equality $gy_n = r(t_n)$ and the second relation of (\ref{numericalSolution2}), we arrive at

\begin{equation} \label{csolution2}
\big\{g\big\}Y'_n = r'_n + \theta - \big\{g\big\}\delta.
\end{equation}

\noindent By (\ref{csolution1}) and (\ref{csolution2}), we can then rewrite (\ref{numericalSolution}) as

\begin{equation} \label{equation2}
Y_{ni} = y_n + h\textstyle\sum_{j=1}^s a_{ij}\big( f(Y_{nj}, Z_{nj})+\delta_{j}\big), \quad g f(Y_{ni}, Z_{ni}) = r'(t_{ni}) + \theta_{i}-g\delta_{i}, \quad i=1(1)s.
\end{equation}

\noindent Now consider an index one differential-algebraic system equivalent to (\ref{I2DAE}) in the form of 

\begin{equation} \label{I1DAE}
y' =  f(y,z), \quad gf(y,z) = r'(t).
\end{equation}

\noindent The above system is obtained by differentiating the second relation of (\ref{I2DAE}) with respect to $t$. It can be easily shown that the exact solution of (\ref{I1DAE}) is identical to that of (\ref{I2DAE}) if the initial conditions for $y$- and $z$-components are the same in both systems. The application of an $s$-stage IRK method to (\ref{I1DAE}) yields

\begin{equation} \label{I1RKformulation}
Y_{ni} = y_n + h\textstyle\sum_{j=1}^s a_{ij} f(Y_{nj}, Z_{nj}), \quad gf(Y_{ni},Z_{ni}) = r'(t_{ni}), \quad i = 1(1)s,
\end{equation}

\noindent with $y_{n+1}$ and $z_{n+1}$ given by (\ref{RKfinalstage}). According to Theorem 3.1 of \cite{hairer1989}, the solution of (\ref{I1RKformulation}) is bounded by

\begin{equation}
y_n - y(t_n) = O(h^p), \quad z_n - z(t_n) = O(h^{q+1}), \nonumber
\end{equation}

\noindent If $a_{si} = b_i$ for all $i$, the $z$-component in (\ref{I1RKformulation}) further satisfies

\begin{equation}
z_n-z(t_n) = O(h^{p}). \nonumber
\end{equation}

\noindent By comparing (\ref{I1RKformulation}) with (\ref{equation2}) and noticing the fact that $\delta_{i}$ and $\theta_{i}$ are of $O(h^{p})$, the statement of the theorem follows from a direct implementation of the lemmas given below. More specifically, Lemma \ref{lemma1} gives the existence and uniqueness of the solutions of (\ref{csolution1}) and (\ref{I1RKformulation}). Lemma \ref{lemma3} shows that if the values of $\delta_i$ and $\theta_i$ in (\ref{equation2}) are sufficiently small, the solution of (\ref{equation2}) has the same orders of convergence for both $y$-and $z$-components as those of (\ref{I1RKformulation}). Lemma \ref{lemma2} is established to support the proof of Lemma \ref{lemma3}.

\end{proof}

\begin{lemma} \label{lemma1}

Consider a system of equations regarding $y(t)\in \mathbb{R}^{m_y}$, and $z(t)\in \mathbb{R}^{m_z}$ given by

\begin{equation} \label{I1DAE2}
y' =  f(y,z), \quad g(y,z) = 0,
\end{equation}

\noindent where $f:\mathbb{R}^{m_y}\times\mathbb{R}^{m_z}\rightarrow\mathbb{R}^{m_y}$ and $g:\mathbb{R}^{m_y}\times\mathbb{R}^{m_z}\rightarrow\mathbb{R}^{m_z}$ are sufficiently smooth functions of $y$ and $z$, and $g_z^{-1}$ exists and satisfies

\begin{equation} \label{boundedness}
\norm{(g_z)^{-1}(y,z)} \leq M
\end{equation}

\noindent in a neighbourhood of the solution. Let $(\eta,\zeta)$ satisfies $g(\eta,\zeta)=O(h)$ and suppose that the inequality (\ref{boundedness}) holds in a neighbourhood of $(\eta,\zeta)$. Then, the non-linear system given by

\begin{equation} \label{unperturbedInternalStage}
Y_{i} = \eta + h\textstyle\sum_{j=1}^s a_{ij} f(Y_{j}, Z_{j}), \quad g(Y_{i},Z_{i}) = 0, \quad i =1(1)s,
\end{equation}

\noindent where $a_{ij}$ are Runge-Kutta coefficients, possesses a locally unique solution for $h \leq h_0$ (where $h_0$ is sufficiently small) and the solution satisfies

\begin{equation}
Y_i - \eta = O(h), \quad Z_i - \zeta = O(h).
\end{equation}

\end{lemma}

\begin{proof}

Applying Newton's iteration method to the solution of the non-linear system (\ref{perturbedInternalStage}), for starting values $Y_i^{(0)}=\eta$ and $Z_i^{(0)}=\zeta$, the resulting Jacobian matrix has the form

\begin{equation} \label{Jacobian}
\left(\begin{matrix}
I - h\big(A\otimes I\big)\big\{f_y(\eta,\zeta)\big\} & - h\big(A\otimes I\big)\big\{f_z(\eta,\zeta)\big\} \\
\big\{g_y(\eta,\zeta)\big\} & \big\{g_z(\eta,\zeta)\big\}
\end{matrix}\right).
\end{equation}

\noindent The existence and boundedness of $(g_z)^{-1}(y,z)$ in a neighbourhood of $(\eta,\zeta)$ suggests that, for a sufficiently small time-step size $h$,  the matrix (\ref{Jacobian}) is non-singular and its inverse can be expressed as

\begin{equation} \label{inverseMatrix}
\left(
\begin{matrix}
I + O(h) & O(h) \\
O(1) & O(1)
\end{matrix}
\right).
\end{equation}

\noindent Meanwhile, the assumption $g(\eta,\zeta)=O(h)$ implies that the first increment of Newton's iteration is of $O(h)$. Consequently, the existence and unique of the solution of (\ref{perturbedInternalStage}) can be deduced from a direct implementation of the theorem of Newton-Kantorovich \citep{kantorovich1964, ortega1970}.

\end{proof}

\begin{lemma} \label{lemma2} 

Consider the problem (\ref{I1DAE2}), and let $(Y_{ni},Z_{ni})$ and $(\hat{Y}_{ni},\hat{Z}_{ni})$ be,  respectively, given by (\ref{unperturbedInternalStage}) and a perturbed scheme

\begin{equation} \label{perturbedInternalStage}
\hat{Y}_{i} = \hat{\eta} + h\textstyle\sum_{j=1}^s a_{ij} f(\hat{Y}_{j}, \hat{Z}_{j})+h\delta_{i}, \quad g(\hat{Y}_{i},\hat{Z}_{i}) = \theta_{i}, \quad i = 1(1)s,
\end{equation}

\noindent In addition to the assumptions of Lemma \ref{lemma1}, suppose that

\begin{equation} \label{lemmaAssumption}
\hat{\eta}-\eta=O(h),\quad \hat{\zeta}-\zeta=O(h), \quad \delta_i = O(1), \quad \theta_i=O(h), \nonumber
\end{equation}

\noindent Then we have for $h \leq h_0$ the estimates

\begin{equation} \label{estimate1}
\begin{split}
\norm{\hat{Y}_i-Y_i} &\leq \norm{\varDelta \eta}+C\big(h\norm{\varDelta \eta}+h\norm{\delta}+h\norm{\theta}\big), \\
\norm{\hat{Z}_i-Z_i} &\leq C\big(\norm{\varDelta \eta}+h\norm{\delta}+\norm{\theta}\big),
\end{split}
\end{equation}

\noindent where $\varDelta \eta \coloneqq \hat{\eta}-\eta$, $\delta \coloneqq (\delta_1,\ldots,\delta_s)^T$, $\theta \coloneqq (\theta_1,\ldots,\theta_s)^T$. If the coefficient matrix $A=(a_{ij})$ is non-singular, we further have

\begin{equation} \label{estimate2}
\begin{split}
\norm{\varDelta y} &\leq \norm{\varDelta \eta} + C(h\norm{\eta}+h\norm{\delta}+h\norm{\theta}), \\
\norm{\varDelta z} &\leq \alpha \norm{\varDelta \zeta} + C(\norm{\eta}+h\norm{\delta}+\norm{\theta}),
\end{split} 
\end{equation}

\noindent where $\alpha = |R(\infty)|$, $\varDelta y = \hat{y}-y$, $\varDelta z = \hat{z}-z$, and the pairs of values $(y,z)$ and $(\hat{y},\hat{z})$ are respectively given by

\begin{equation} \label{finalStage}
\begin{split}
y = \eta+\textstyle\sum_{i=1}^s b_i w_{ij}(Y_j-\eta), \quad z = \zeta+\textstyle\sum_{i=1}^s b_i w_{ij}(Z_j-\zeta), \\
\hat{y} = \hat{\eta}+\textstyle\sum_{i=1}^s b_i w_{ij}(\hat{Y}_j-\hat{\eta}), \quad \hat{z} = \hat{\zeta}+\textstyle\sum_{i=1}^s b_i w_{ij}(\hat{Z}_j-\hat{\zeta}).
\end{split} 
\end{equation}

\end{lemma}

\begin{proof}

Similar to the proof of Lemma \ref{lemma1}, we conclude from (\ref{lemmaAssumption}) the system (\ref{perturbedInternalStage}) possesses a locally unique solution $(\hat{Y}_{i},\hat{Z}_{i})$ satisfying $\hat{Y}_{i}-\hat{\eta}=O(h)$ and $\hat{Z}_{i}-\hat{\zeta}=O(h)$. To ingestive the difference between $(Y_i,Z_i)$ of (\ref{unperturbedInternalStage}) and $(\hat{Y}_i,\hat{Z}_i)$ of (\ref{perturbedInternalStage}), we consider the homotopy

\begin{equation} \label{homotopy}
Y_{i} = \eta + h\textstyle\sum_{j=1}^s a_{ij} f(Y_{j}, Z_{j}) + \tau(\hat{\eta}-\eta+h\delta_i), \quad g(Y_{i},Z_{i})  = \tau\theta_i, \quad i=1(1)s.
\end{equation}

\noindent It can be easily seen that when $\tau = 0$, the system (\ref{homotopy}) reduces to the unperturbed Runge-Kutta method (\ref{unperturbedInternalStage}), and when $\tau = 1$, it is equivalent to the perturbed Runge-Kutta scheme (\ref{perturbedInternalStage}). If we consider $Y_{i}$ and $Z_{i}$ in (\ref{homotopy}) as functions of $\tau$, differentiating (\ref{homotopy}) with respect to $\tau$ gives for $i=1(1)s$

\begin{equation} \label{dhomotopy}
\begin{split}
&Y'_{i} = h\textstyle\sum_{j=1}^s a_{ij}\big(f_y(Y_{j}, Z_{nj})Y'_{j} + f_z(Y_{j}, Z_{nj})Z'_{j}\big) + \big(\hat{\eta}-\eta+h\delta_i\big), \\
&g_y(Y_{i},Z_{i})Y'_{i} + g_z(Y_{i},Z_{i})Z'_{i}  = \theta_i,
\end{split}
\end{equation}

\noindent If we define $Y \coloneqq (Y_{1},\ldots,Y_{s})^T$, $Z' \coloneqq (Z_{1},\ldots,Z_{s})^T$, $\delta \coloneqq (\delta_1,\ldots,\delta_s)^T$, $\theta \coloneqq (\theta_1,\ldots,\theta_s)^T$ and use the notation

\begin{equation}
\{ f_y \} \coloneqq \text{blockdiag} \big( f_y(Y_{n1}, Z_{n1}), \ldots, f_y(Y_{ns}, Z_{ns})\big),
\end{equation}

\noindent and similarly for $f_z$, $g_y$ and $g_z$, (\ref{dhomotopy}) can therefore be rewritten as

\begin{equation} \label{dhomotopyCompact}
\left(\begin{matrix}
I - h(A\otimes I)\{f_y\} & - h(A\otimes I)\{f_z\} \\
\{g_y\} & \{g_z\}
\end{matrix}\right)
\left(\begin{matrix}
Y' \\
Z'
\end{matrix}\right) 
=
\left(
\begin{matrix}
\mathbbm{1} \otimes \varDelta \eta + h\delta \\
\theta
\end{matrix}\right),
\end{equation}

\noindent The inverse of $\{g_z\}$ exists and is bounded by (\ref{boundedness}) provided that all $(Y_i,Z_i)$ remain in a small neighbourhood of $(\eta,\zeta)$. Consequently, the matrix on the LHS of (\ref{dhomotopyCompact}) has a bounded inverse of the form (\ref{inverseMatrix}) for a sufficiently small time-step size $h$. We now introduce the notation $\varDelta Y = (\varDelta Y_1,\ldots,\varDelta Y_2)^T$, $\varDelta Y_{i} \coloneqq \hat{Y}_{i}-Y_{i}$ and similarly for the $z$-component. Thus the fact

\begin{equation} \label{Yintegral}
\varDelta Y = \textstyle\int_0^1 Y' \ \mathrm{d}\tau,\quad \varDelta Z = \textstyle\int_0^1 Z' \ \mathrm{d}\tau,
\end{equation}

\noindent together with the mean value theorem implies that

\begin{equation} \label{dhomotopyBounded}
\begin{split}
\varDelta Y &= \mathbbm{1} \otimes \varDelta \eta + O(h\norm{\varDelta \eta}+h\norm{\delta}+h\norm{\theta}), \\
\varDelta Z &= O(\norm{\varDelta \eta}+h\norm{\delta}+\norm{\theta}).
\end{split}
\end{equation}

\noindent which proves the estimates (\ref{estimate1}) of the lemma. (\ref{estimate2}) remains to be proved. Using (\ref{Rinfty}) and (\ref{finalStage}) we obtain

\begin{equation} \label{differenceFinalStage}
\begin{split}
\varDelta y &= \varDelta \eta + \big(b^TA^{-1}\otimes I\big)\big(\varDelta Y - \mathbbm{1}\otimes \varDelta \eta\big), \\
\varDelta z &= R(\infty) \varDelta \zeta + \big(b^TA^{-1}\otimes I\big) \varDelta Z.
\end{split}
\end{equation}

\noindent Inserting (\ref{dhomotopyBounded}) into (\ref{differenceFinalStage}), the estimates (\ref{estimate2}) of the lemma now becomes straightforward.

\end{proof}

\begin{lemma} \label{lemma3}

Consider the problem (\ref{I1DAE2}) and let $(y_n,z_n)$ and $(\hat{y}_n,\hat{z}_n)$ be given by an implicit Runge-Kutta method

\begin{subequations} \label{lemma3_1}
\begin{gather}
Y_{ni} = y_n + h\textstyle\sum_{j=1}^s a_{ij} f(Y_{nj}, Z_{nj}), \quad g(Y_{ni},Z_{ni})=0, \quad i=1(1)s, \\
y_{n+1} = y_n + \textstyle\sum_{i,j=1}^s b_i \omega_{ij}(Y_{nj}-y_n), \quad z_{n+1} = z_n + \textstyle\sum_{i,j=1}^s b_i \omega_{ij}(Z_{nj}-z_n),
\end{gather}
\end{subequations}

\noindent and a perturbed scheme

\begin{subequations} \label{lemma3_2}
\begin{gather}
\hat{Y}_{ni} = \hat{y}_n + h\textstyle\sum_{j=1}^s a_{ij} f(\hat{Y}_{nj}, \hat{Z}_{nj})+h\delta_{i}, \quad g(\hat{Y}_{ni},\hat{Z}_{ni}) = \theta_{i}, \quad i = 1(1)s, \\
\hat{y}_{n+1} = \hat{y}_n + \textstyle\sum_{i,j=1}^s b_i \omega_{ij}(\hat{Y}_{nj}-\hat{y}_n), \quad \hat{z}_{n+1} = \hat{z}_n + \textstyle\sum_{i,j=1}^s b_i \omega_{ij}(\hat{Z}_{nj}-\hat{z}_n),
\end{gather}
\end{subequations}

\noindent respectively. Suppose that the initial values of the method (\ref{lemma3_1}) are consistent and those of (\ref{lemma3_2}) satisfy

\begin{equation} \label{lemma3_3}
\hat{y}_0-y_0=O(h), \quad \hat{z}_0-z_0 =O(h),
\end{equation}

\noindent Also suppose the Runge-Kutta method (\ref{lemma3_1}), satisfying $\alpha=|R(\infty)|<1$, is of classical order $p$ and stage order $q$ with $p \geq q+1$ and $q\geq 1$. Then the estimates

\begin{equation} \label{lemma3_4}
\begin{split}
\norm{\varDelta y_n} &\leq C\big(\norm{\varDelta y_0}+\norm{\delta}+\norm{\theta}\big), \\
\norm{\varDelta z_n} &\leq C\big(\norm{\varDelta y_0}+\alpha^n\norm{\varDelta z_0}+\norm{\delta}+\norm{\theta}\big),
\end{split}
\end{equation}

\noindent holds for $\delta_i = O(h)$, $\theta_{i}=O(h)$, $h \leq h_0$ and $nh<Const$.

\end{lemma}

\begin{proof}

According to Theorem 3.1 of \cite{hairer1989}, the assumption that the Runge-Kutta method (\ref{lemma3_1}), satisfying $|R(\infty)|<1$, is of classical order $p$ and stage order $q$ with $p \geq q+1$ and $q\geq 1$ implies

\begin{equation}
y_n - y(t_n) =O(h), \quad z_n - z(t_n) = O(h), \quad g(y_n,z_n) = O(h).
\end{equation}

\noindent In order to be able to apply Lemma \ref{lemma2}, we suppose that

\begin{equation} \label{assumption}
\quad \norm{\hat{y}_n-y_n} \leq C_0 h, \quad \norm{\hat{z}_n-z_n} \leq C_1 h,
\end{equation}

\noindent which will be justified later. The application of Lemma \ref{lemma2} yields the recurrence relations

\begin{subequations}
\begin{align}
\norm{\varDelta y_{n+1}} &\leq \norm{\varDelta y_{n}} + C\big(h\norm{\varDelta y_{n}}+h\norm{\delta}+h\norm{\theta}\big),  \label{recurrence_a} \\
\norm{\varDelta z_{n+1}} &\leq \alpha \norm{\varDelta z_{n}} + C\big(\norm{\varDelta y_{n}}+h\norm{\delta}+\norm{\theta}\big). \label{recurrence_b}
\end{align}
\end{subequations}

\noindent By deduction, (\ref{recurrence_a}) implies that

\begin{equation} \label{lemma3_5}
\begin{split}
\norm{\varDelta y_{n}} &\leq (1+Ch)\norm{\varDelta y_{n-1}} + C\big(h\norm{\delta}+h\norm{\theta}\big) \\
& \leq (1+Ch)\norm{\varDelta y_{n-2}} + C(1+Ch)\big(h\norm{\delta}+h\norm{\theta}\big)+ C\big(h\norm{\delta}+h\norm{\theta}\big) \\
& \cdots \\
& \leq (1+Ch)^n\norm{\varDelta y_{0}} + \sum_{i=1}^n (1+Ch)^n \cdot C\big(h\norm{\delta}+h\norm{\theta}\big).
\end{split}
\end{equation}

\noindent Considering the fact that $(1+Ch)^n = 1 + O(h)$ for $nh \leq Const$, we obtain from (\ref{lemma3_5}) that

\begin{equation} \label{lemma3_6}
\norm{\varDelta y_{n}} \leq C\big(\norm{\varDelta y_{0}}+\norm{\delta}+\norm{\theta}\big).
\end{equation}

\noindent Inserting (\ref{lemma3_6}) into (\ref{recurrence_b}), this yields

\begin{equation} \label{lemma3_7}
\norm{\varDelta z_{n+1}} \leq \alpha \norm{\varDelta z_{n}} + C\big(\norm{\varDelta y_0}+h\norm{\delta}+\norm{\theta}\big).
\end{equation}

\noindent By using the technique similar to that of (\ref{lemma3_5}), we also obtain

\begin{equation} \label{lemma3_8}
\norm{\varDelta z_n} \leq C\big(\norm{\varDelta y_0}+\alpha^n\norm{\varDelta z_0}+\norm{\delta}+\norm{\theta}\big),
\end{equation}

\noindent which proves the statement of the lemma. We now justify the assumption (\ref{assumption}). On easily verifies that the constant in (\ref{recurrence_a}) can be chosen independently of $C_0$ or $C_1$ if we restrict the step size $h$ such that $hC_0$ is bounded by an $h$-independent constant. Thus the constants in (\ref{lemma3_6})$\sim$(\ref{lemma3_8}) do not depend on $C_0$ or $C_1$ either. The assumption (\ref{assumption}) then follows from (\ref{lemma3_6}), (\ref{lemma3_8}) and (\ref{lemmaAssumption}) provided that $C_0$ and $C_1$ are chosen sufficiently large.

\end{proof}

The convergence results given in Theorem \ref{theorem1} highlights the efficiency of the stiff-accurate DIRK schemes over the non-stiff-accurate ones. For they not only attain the classical order of convergence for both the velocity and pressure, but also do not require the previous information of pressures. In comparison, a DIRK scheme which is not stiff-accurate can only deliver second-order of convergence for pressures at most and requires the previous information of the pressure. It is noted that the equality (\ref{RKConstraint}) dose not define $\theta_i$ uniquely. Therefore, the specification of this variable can be quite arbitrary. For still-accurate Runge-Kutta methods we can simply set $\theta_i$ as

\begin{equation} \label{stiffaccurateRKConstraint}
h\theta_i =
\begin{cases}
0 & i < s \\
r(t_{n+1}) - r(t_n) - h\sum_{j=1}^s b_{i} r'(t_{ni}) & i = s
\end{cases}.
\end{equation}

\noindent It is easy to see the resulting $\theta_i$ are of $O(h^p)$ which satisfies the requirements of Theorem \ref{theorem1}.

\subsection{Low storage implementation of stiff-accurate DIRK methods} \label{section3.4}

The direct approach of applying an $s$-stage DIRK scheme to the differential-algebraic problem (\ref{I2DAE}) requires the full storage of internal stage $Y_{ni}$ and $Z_{ni}$ values. In contrast, if the DIRK scheme is implemented with the proposed method (\ref{RKformulation}), the number of the storage locations needed for the algebraic component reduce to two with one storage location for the algebraic component at the current internal stage and the other one for the contribution of the algebraic components from the previous time-step and internal stages to the terms on the RHS of (\ref{zcomponent}). If the DIRK scheme considered is stiff-accurate, the number of the storage locations required for the algebraic component further reduces to one. 

In the meantime, it is in general difficult for a non-stiff-accurate DIRK scheme to further reduce the number of the storage locations required by differential components.  However, if the DIRK schemes are stiff accurate and their coefficients possess certain regular patterns, the low-storage implementation of these schemes are possible. In the light of the work of \cite{van1972} and \cite{kennedy2000} on the low-storage explicit Runge-Kutta schemes for ordinary differential equations, we first consider a family of stiff-accurate DIRK schemes taking the tableau of the form

\begin{equation} \label{DIRKfamily1}
\begin{array}{ l | l  l  l  l  l}		
  c_1 & a_{11} &  &  &  &  \\
  c_2 & b_1 & a_{22} &  &  &  \\
  c_3 & b_1 & b_2 & a_{33} &  &  \\
 \vdots  & \vdots &  \vdots & \ddots & \ddots &  \\
  c_s & b_1  & b_2 &  \ldots & b_{s-1}  & a_{ss} \\ \cline{1-6}
  c_s & b_1  & b_2 &  \ldots & b_{s-1}  & a_{ss}
\end{array}
\end{equation}

\noindent where $a_{ii}$ are non-zero values. It allows $2s-1$ degrees of freedom (DOF), where $s$ is the stage numbers, to satisfy the order conditions. More specifically, we have 3 DOFs for a two-stage scheme and 5 DOFs for a three-stage scheme. Considering the number of order conditions for order two and order three are 2 and 4 respectively, therefore we are able to attain the classical order of 2 using two stages and the classical order of 3 using three stages among the schemes (\ref{DIRKfamily1}). For example, a two-stage second-order scheme of (\ref{DIRKfamily1}) in which all the diagonal coefficients are identical takes the form

\begin{equation} \label{SDIRK2}
\begin{array}{ c | c  c }
  \gamma & \gamma & \\
  1 & 1-\gamma & \gamma \\ \cline{1-3}
  & 1-\gamma & \gamma \\
\end{array}
\end{equation}

\noindent where $\gamma = 1 - \sqrt{2}/2$. It is noted that the coefficients of this DIRK scheme (denoted by SDIRK2 hereafter) was first given by \cite{alexander1977}. For the low-storage implementation of this family of schemes in the method (\ref{RKformulation}), we propose a scheme that uses two register for $y$ (namely $y$, $y^* \in \mathbb{R}^{N_y}$) and $r$ (namely $r$, $r^* \in \mathbb{R}^{N_z}$), and one register for $z$ (namely $z\in\mathbb{R}^{N_z}$):

\begin{equation} \label{lowStorage1}
\begin{array}{l}
\text{for } i = 1:s \\
\quad \quad \text{if } i == 1, \  y^* \leftarrow y, \ r^* \leftarrow r, \ \text{else}  \\
\quad \quad \quad \quad y^* \leftarrow y^* + b_{i-1} \omega_{i-1,i-1}(y-y^*) \\
\quad \quad \quad \quad r^* \leftarrow r^* + b_{i-1} \omega_{i-1,i-1}(r-r^*) \\
\quad \quad \text{end} \\
\quad \quad \text{if } i == s,\ r \leftarrow r(t_{n+1}), \ \text{else}\\
\quad \quad \quad \quad r \leftarrow r^* + ha_{ii}r'(t_{ni}) \\
\quad \quad \text{end} \\
\quad \quad \text{solve } y = y^* + ha_{ii}f(y,z), \quad g(z) = r \\
\text{end}
\end{array}
\end{equation}

\noindent Next we consider a family of stiff-accurate DIRK schemes in the form of

\begin{equation} \label{DIRKfamily2}
\begin{array}{ l | l  l  l  l  l}		
  c_1 & a_{11} &  &  &  &  \\
  c_2 & a_{21} & a_{22} &  &  &  \\
  c_3 & b_1 & a_{32} & a_{33} &  &  \\
 \vdots  & \vdots &  \vdots & \ddots & \ddots &  \\
  c_s & b_1  & b_2 &  \ldots & a_{s,s-1}  & a_{ss} \\ \cline{1-6}
  c_s & b_1  & b_2 &  \ldots & a_{s,s-1}  & a_{ss}
\end{array}
\end{equation}

\noindent where both $a_{ii}$ and $a_{i,i-1}$ are non-zero. Different from the first family of schemes, the schemes given by (\ref{DIRKfamily2}) allows $3s-3$ DOFs to satisfy the order conditions, that is to say we have 3 DOFs for a two-stage scheme and 6 DOFs for a three-stage scheme. Considering the number of order conditions for the fourth-order accuracy is 8 which is larger than the DOFs of a three-stage scheme. Therefore, the highest classical order we can achieve with a two-stage or a three-stage (\ref{DIRKfamily2}) is actually identical to that of (\ref{DIRKfamily1}). Also note that we have four order conditions for the classical order 3, thus we happen to be able to construct a three-stage third order scheme of (\ref{DIRKfamily2}) with all $a_{ii}$ identical (which takes up 2DOFs), and this scheme can be written as:

\begin{equation} \label{SDIRK3}
\begin{array}{ c | c  c  c}
  \gamma & \gamma &  & \\
  c_2 & c_2-\gamma & \gamma  &\\
  1 & 1-\gamma-b_2 & b_2 & \gamma \\ \cline{1-4}
  & 1-\gamma-b_2 & b_2 & \gamma \\
\end{array}
\end{equation}

\noindent with

\begin{equation} 
\begin{array}{l}
  c_2 = (1+\gamma)/2, \\
  b_2 = (6\gamma^2-20\gamma+5)/4,
\end{array}\nonumber
\end{equation}

\noindent and $\gamma$ being the root of $6\gamma^3 - 18\gamma^2 + 9\gamma -1 = 0$ lying in $(1/6,1/2)$. Also note that the Runge-Kutta coefficients of (\ref{SDIRK3}) (denoted by SDIRK3 hereafter) was also first introduced in \cite{alexander1977}.  For the low-storage implementation of the second family of schemes in the proposed method (\ref{RKformulation}), we propose a scheme that uses three register for $y$ (namely $y$, $y^*$, $y^{**} \in \mathbb{R}^{N_y}$) and $r$ (namely $r$, $r^*$, $r^{**} \in \mathbb{R}^{N_z}$), and one register for $z$ (namely $z\in\mathbb{R}^{N_z}$):

\begin{equation} \label{lowStorage2}
\begin{array}{l}
\text{for } i = 1:s \\
\quad \quad \text{if } i == 1 \\
\quad\quad \quad \quad y^* \leftarrow y,\ y^{**} \leftarrow y, \ r^* \leftarrow r, \  r^{**} \leftarrow r \\
\quad \quad \text{else} \\
\quad \quad \quad \quad y^* \leftarrow y^{**} + a_{i,i-1} \omega_{i-1,i-1}(y-y^*) \\
\quad \quad \quad \quad r^* \leftarrow r^{**} + a_{i,i-1} \omega_{i-1,i-1}(r-r^*) \\
\quad \quad \quad \quad \text{if } i < s \\
\quad \quad \quad \quad \quad \quad y^{**} \leftarrow y^{**} + b_{i-1} (y^*-y^{**})/a_{i,i-1} \\
\quad \quad \quad \quad \quad \quad r^{**} \leftarrow r^{**} + b_{i-1} (r^*-r^{**})/a_{i,i-1} \\
\quad \quad \quad \quad \quad \quad r \leftarrow r^* + ha_{ii}r'(t_{ni}) \\
\quad \quad \quad \quad \text{else} \\
\quad \quad \quad \quad \quad \quad r \leftarrow r(t_{n+1}) \\
\quad \quad \quad \quad \text{end} \\
\quad \quad \text{end} \\
\quad \quad \text{solve } y = y^* + ha_{ii}f(y,z), \quad g(z) = r \\
\text{end}
\end{array}
\end{equation}

\noindent Higher-order stiff-accurate DIRK schemes with more internal stages can be designed for both (\ref{DIRKfamily1}) and (\ref{DIRKfamily2}). However, we shall only consider SDIRK2 and SDIRK3 in this study. This is because our major attention is focused on if the stiff accurate DIRK schemes implemented with (\ref{RKformulation}) achieves higher-order accuracy for algebraic components when compared to the direct approach, rather than investigating the highest classical order a stiff-accurate DIRK scheme can achieve. In the meantime, both SDIRK2 and SDIRK3 have reached the optimal classical order concerning the number of the internal stages and the storage they use, and achieve a good balance between accuracy and efficiency. Their orders of convergence for the solution of the semi-discrete incompressible Navier-Stokes system will be validated in our numerical experiments presented later.

\section{Solution of the internal Stage Navier-Stokes Equations} \label{section4}

The objective of this section is to discuss the solution algorithm of the discretized incompressible Navier-Stokes equations at every Runge-Kutta internal stage. For generosity, we consider the original implementation of the developed Runge-Kutta method, i.e., (\ref{RKformulation}), rather than its low storage formulation, i.e., (\ref{lowStorage1}) or (\ref{lowStorage2}), for a particular type of DIRK schemes. Thus, an $s$-stage stiff-accurate DIRK scheme for the time-marching of the semi-discrete incompressible Navier-Stokes system (\ref{system}) from $t_{n}$ to $t_{n+1}$ can be specified as follows:

\begin{subequations} \label{subndse}
\begin{gather}
U_{ni} = u_n + \textstyle\sum_{j,k=1}^{i-1} a_{ij} \omega_{jk}(U_{nk} - u_n) + h a_{ii} F(t_{ni},U_{ni},\bar{U}_{ni},P_{ni}), \label{subndses1}\\
\bar{U}_{ni} = \bar{u}_n + \textstyle\sum_{j,k=1}^{i-1} a_{ij} \omega_{jk}(\bar{U}_{nk} - \bar{u}_n) + h a_{ii} \bar{F}(t_{ni},U_{ni},\bar{U}_{ni},P_{ni}), \label{subndses2}\\
D\bar{U}_{ni} = r(t_n) + h\textstyle\sum_{j=1}^s a_{ij}\big( r'(t_{ni}) + \theta_{i} \big), \label{subndses3}
\end{gather}
\end{subequations}

\noindent where $i=1(1)s$, $u_{n+1} = U_{ns}$, $\bar{u}_{n+1} = \bar{U}_{ns}$, $p_{n+1} = P_{ns}$ and $\theta_{i}$ is given by (\ref{stiffaccurateRKConstraint}). $U_{ni}$, $\bar{U}_{ni}$ and $P_{ni}$ are the cell-centered, face-centered velocity fields and the cell-centered pressure field at the time instant $t_{ni}$, respectively. Note that the discretized momentum equations given by (\ref{subndse}) are easy to implement on the basis of the existing finite volume codes since they resemble the discretized momentum equation, e.g. (\ref{Euler}), given by an implicit multi-step method. The only difference is that the terms on the RHS of (\ref{subndses1}) and (\ref{subndses2}), which represent the contribution of last velocities to the next ones, include not only the velocities at the previous time-step but also the velocities at the previous Runge-Kutta internal stages.

The Runge-Kutta internal stage incompressible Navier-Stokes equations given by (\ref{subndse}) are non-linear and should be solved using a fixed-point iteration method. The Newton-Raphson method can be a good choice for its quadratic convergence. However, the complexity of a Newton's solver and the computational effort it requires are the major obstacles in practice especially for the incompressible Navier-Stokes system considered in this study. If the Newton iterations are terminated before the exact solutions are reached in a Newton's method, residual errors will arise in the internal stage continuity equation (\ref{subndses3}) and consequently the numerical solutions will fail to satisfy the algebraic constraint. In the meantime, a Picard linearization method for non-linearity in combination with a splitting method for pressure-velocity coupling is a more appealing solution algorithm. This is because, on one hand, this algorithm enforces the equality (\ref{subndses3}) even if the Picard iterations are terminated before the exact solutions are reached. On the other hand, it is much easier to implement than a Newton-type solver considering the resemblance between (\ref{subndse}) and the discretized equations given by an implicit multi-step method. Details of this approach are given as follows: to start with, we write $\bar{r}_{ni}$ for the RHS of (\ref{subndse3}), and insert (\ref{interpolation1}) into (\ref{subndses1}) and (\ref{subndses2}) which leads to

\begin{subequations} \label{subndse2}
\begin{gather}
U_{ni} = \big\{A_{ni}^*\big\} \big(u_{ni}+ha_{ii}H_{ni}-ha_{ii}GP_{ni}\big), \label{subndse2s1}  \\
\bar{U}_{ni} = \big\{\bar{A}_{ni}^*\big\} \big(\bar{u}_{ni}+ha_{ii}\bar{H}_{ni}-ha_{ii}\bar{G}P_{ni}\big), \label{subndse2s2} \\
\intertext{with}
u_{ni} = u_n + \textstyle\sum_{j,k=1}^{i-1} a_{ij} \omega_{jk}\big(U_{nk} - u_n\big), \quad \bar{u}_{ni} = \bar{u}_n + \textstyle\sum_{j,k=1}^{i-1} a_{ij} \omega_{jk}\big(\bar{U}_{nk} - \bar{u}_n\big), \\
A_{ni}^* = \big(I - ha_{ii}\text{diag}\big(A_{ni}\big)\big)^{-1}, \quad \bar{A}_{ni}^* = \big(I - ha_{ii}\text{diag}\big(\bar{A}_{ni}\big)\big)^{-1},
\end{gather}
\end{subequations}

\noindent where $A_{ni}=A(\bar{U}_{ni})$ and $H_{ni} = H(t_{ni},U_{ni},\bar{U}_{ni})$ (see (\ref{Anotation}) for the definitions of the $A$ and $H$ symbols). $\bar{A}_{ni}$ and $\bar{H}_{ni}$ are computed by the GFISDM scheme (\ref{daem}) or (\ref{modifieddaem}) using the values of $A_{ni}$ and $H_{ni}$. Multiplying $D$ to the both side of (\ref{subndse2s2}) and using the relation (\ref{subndses3}) yields

\begin{equation}\label{subndse3}
ha_{ii}D\big\{\bar{A}_{ni}^*\big\}\bar{G} P_{ni} =  D\big\{\bar{A}_{ni}^*\big\}\big(\bar{u}_{ni}+ha_{ii}\bar{H}_{ni}\big)-\bar{r}_{ni}.
\end{equation}

\noindent The equations of (\ref{subndse2s1}) and (\ref{subndse3}) can be written together in a compact form as

\begin{equation} \label{PISOalgorithm}
\left(
\begin{matrix}
I & ha_{ii}\big\{A_{ni}^*\big\}G \\
0 &  ha_{ii}D\big\{\bar{A}_{ni}^*\big\}\bar{G}  \\
\end{matrix}
\right)
\left(
\begin{matrix}
U_{ni} \\
P_{ni}
\end{matrix}
\right)
=
\left(
\begin{matrix}
\big\{A_{ni}^*\big\}\big(u_{ni}+ha_{ii}H_{ni}\big) \\
D\big\{\bar{A}_{ni}^*\big\}\big(\bar{u}_{ni}+ha_{ii}\bar{H}_{ni}\big)-\bar{r}_{ni}
\end{matrix}
\right).
\end{equation}

\noindent Note that the non-linearity of (\ref{subndse2}) or (\ref{PISOalgorithm}) results from that fact that the $A_{ni}$ term (or more specifically the matrix of coefficients resulting from the convection term) is a function of $\bar{U}_{ni}$. In a Picard linearization procedure for (\ref{PISOalgorithm}), this term is computed with a good approximation of $\bar{U}_{ni}$. This approximation can be the cell-face velocity field $\bar{u}$ obtained in the previous Picard iteration loop or the previous Runge-Kutta internal stage. As a result, (\ref{PISOalgorithm}) becomes a system of linear equations regarding the discretized velocities and pressures after the Picard linearization. 

Now, if the linearized Navier-Stokes system is to be solved directly using methods such the Gauss elimination (which is computational demanding), we still need move the $U_{ni}$-dependent components in the $H_{ni}$ and $\bar{H}_{ni}$ terms of (\ref{PISOalgorithm}) to its left hand side before the solution. However, if we start with computing the RHS of (\ref{PISOalgorithm}) with an existing internal stage cell-centered velocity field $U_{ni}$, (\ref{PISOalgorithm}) can then be solved in a staggered manner in which the pressure equation (\ref{subndse3}) is solved first and cell-center velocities are updated with (\ref{subndse2s1}) afterwards. Subsequently, by iterating the related steps until the results converge, we will acquire the exact solution of the linearized system of (\ref{PISOalgorithm}). This method is known as the PISO algorithm \citep{issa1986} in literature. The implementation of the PISO algorithm requires an initial guess for cell-center velocities, in order to trigger the iteration process, which can be obtained by solving (\ref{subndse2s1}). Considering the internal stage pressure field $P_{ni}$ is not available at the beginning of the first Picard iteration loop of the current internal stage, it is approximated with the information of the pressure at the previous internal stage or time-step. Then, in the following Picard iterations, it takes the value of the pressure obtained in the last Picard loop. Finally, after a series of Picard and PISO iterations, the obtained results are used to approximate the exact solution of the system (\ref{subndse}). In practice, either the Picard or the PISO iteration is usually terminated before the exact solutions are reached. As suggested by Theorem \ref{theorem1} and Lemma \ref{lemma3}, if the residual errors are constrained within certain magnitudes, the order of convergence for the numerical results will not be affected.

\section{Taylor-Green Vortex} \label{section5}

The Taylor-Green vortex is an unsteady flow of a decaying vortex, which has an exact closed form solution of the incompressible Navier-Stokes equations on a two dimensional Cartesian grid system:

\begin{equation} \label{taylorGreen}
\begin{split}
u_{(1)}(x_1,x_2,t) &= -\cos x_1 \sin x_2 e^{-2\nu t}, \\
u_{(2)}(x_1,x_2,t) &= \sin x_1 \cos x_2 e^{-2\nu t}, \\
p(x_1,x_2,t) &= -0.25(\cos 2x_1 + \sin 2x_2)  e^{-4\nu t}.
\end{split}
\end{equation}

\noindent It has been widely used as a benchmark problem for evaluating the accuracy of numerical methods for simulating incompressible flows. In this study, we simulate the Taylor-Green vortex numerically to examine the spatial accuracy of the proposed momentum interpolation framework, and the temporal accuracy of the developed Runge-Kutta method. Supplied with the initial conditions given by (\ref{taylorGreen}), we compute the solutions of (\ref{system}) on a square domain $[0,2\pi]\times[0,2\pi]$ with either unsteady Dirichlet or periodic boundary conditions for velocities. When velocities use periodic boundaries, the boundary condition we employ for pressures is periodic as well.  For velocities using unsteady Dirichlet boundaries, we employ the zero-gradient condition for pressures since we have $\nabla p = 0$ on the entire boundaries for the computational domain considered in this study. The domain on which we discretize the incompressible Navier-Stokes equations also ensures the elements of the source term in the discretized continuity equation (\ref{fdce}) are non-zero when the velocity field employs the unsteady Dirichlet boundary condition. 

To demonstrate the performance of the proposed methods in different numerical setup, we consider two kinematic viscosity, i.e., $\nu = 1$ and $\nu = 0.1$, and use the central differencing scheme and the deferred correction formulation of the FROMM scheme for the spatial discretization of convection terms, respectively. Four cases are included in the numerical study: case \rom{1} and case \rom{2} consider the periodic boundary situation using the central differencing scheme and the FROMM scheme, respectively; case \rom{3} and case \rom{4} consider the unsteady Dirichlet boundary situation using the central differencing scheme and the FROMM scheme, respectively. For the solution of the discretized Navier-Stokes equations at Runge-Kutta internal stages, we use the solution algorithm introduced in Section \ref{section4}; and employ the Gauss-Seidel method and the conjugate gradient method for the solutions of the linearized momentum equation and the pressure correction equation at each PISO iteration, respectively. Also note that we use logarithmic scales all error plots demonstrated later, so that the curves appear as straight lines of slope $k$ whenever the corresponding errors are convergent of order $k$.

\subsection{Verification of spatial accuracy}

The verification of spatial accuracy concentrates on the spatial errors of the finite volume discretization at two stages, before and after the temporal discretization. In the first stage of the verification, we examine the spatial errors of the semi-discrete incompressible Navier-Stokes equations based on GFISDM schemes. To this end, we compute the pressure at $t=0$ by solving (\ref{NSconsistency}) with $u$ and $\bar{u}$ taking the values from (\ref{taylorGreen}). Then, we insert the resulting pressure into the first two relations of (\ref{system}) to get the time derivatives of velocities at cell centroids and face centroids. Numerical results are computed for different uniform grids ranging from $16\times16$ to $1024\times1024$ control volumes, and are then compared with the exact solution of the pressure and the time derivative of the velocity:

\begin{equation}
\begin{split}
u'_{(1)}(x_1,x_2,t) &= 2\nu\cos x_1 \sin x_2 e^{-2\nu t}, \\
u'_{(2)}(x_1,x_2,t)& = -2\nu\sin x_1 \cos x_2 e^{-2\nu t}.
\end{split}
\end{equation}

\noindent We demonstrate the spatial errors of the semi-discrete incompressible Navier-Stokes systems based on GFISDM-Z and GFISDM-H respectively in Figure \ref{figure0}. The results indicate that the errors in $u'$, $\bar{u}'$ and $p$ for both GFISDM-Z and GFISDM-H show the second order convergence. In the second stage of the verification, we examine the spatial errors of $u$, $\bar{u}$ and $p$ after the temporal discretization of the semi-discrete incompressible Navier-Stokes equations. For this purpose, we simulate the Taylor-Green vortex from within the time interval $[0,0.1/ \nu]$ using a total of $512$ time-steps. We use the third-order accurate SDIRK3 with sufficient numbers of iterations for the Picard linearization and the PISO algorithm so that the temporal errors are much smaller than the spatial errors. The convergence results shown in Figure \ref{figure2} demonstrate the second order convergence of the spatial errors for GFISDM-Z and GFISDM-H in all numerical setup. 

\begin{figure}[htb]
\centering
    \begin{subfigure}[b]{0.33\linewidth}
        \includegraphics[width=0.9\linewidth]{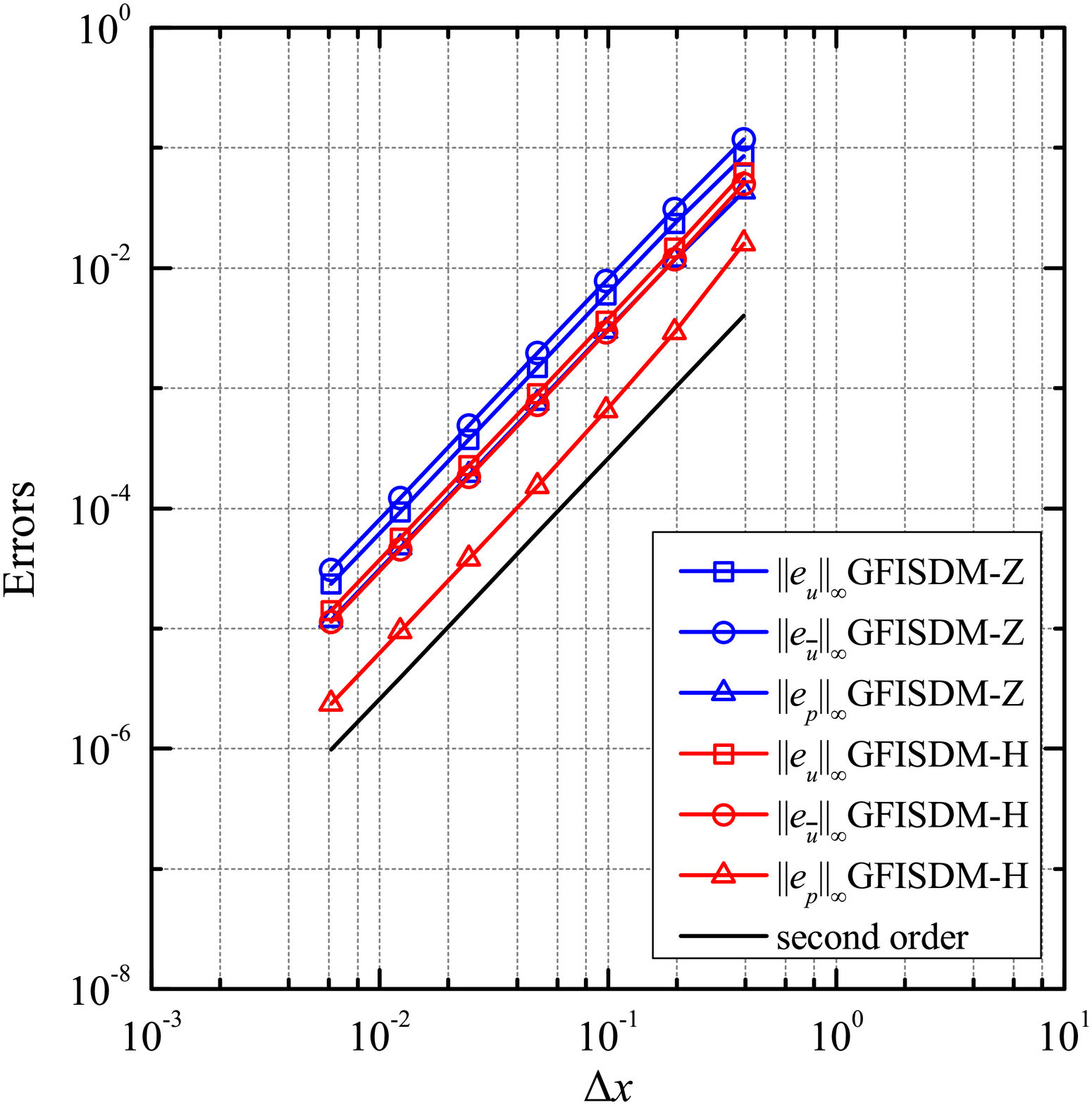}
        \caption{case \rom{1}}
     \end{subfigure}
    \begin{subfigure}[b]{0.33\linewidth}
        \includegraphics[width=0.9\linewidth]{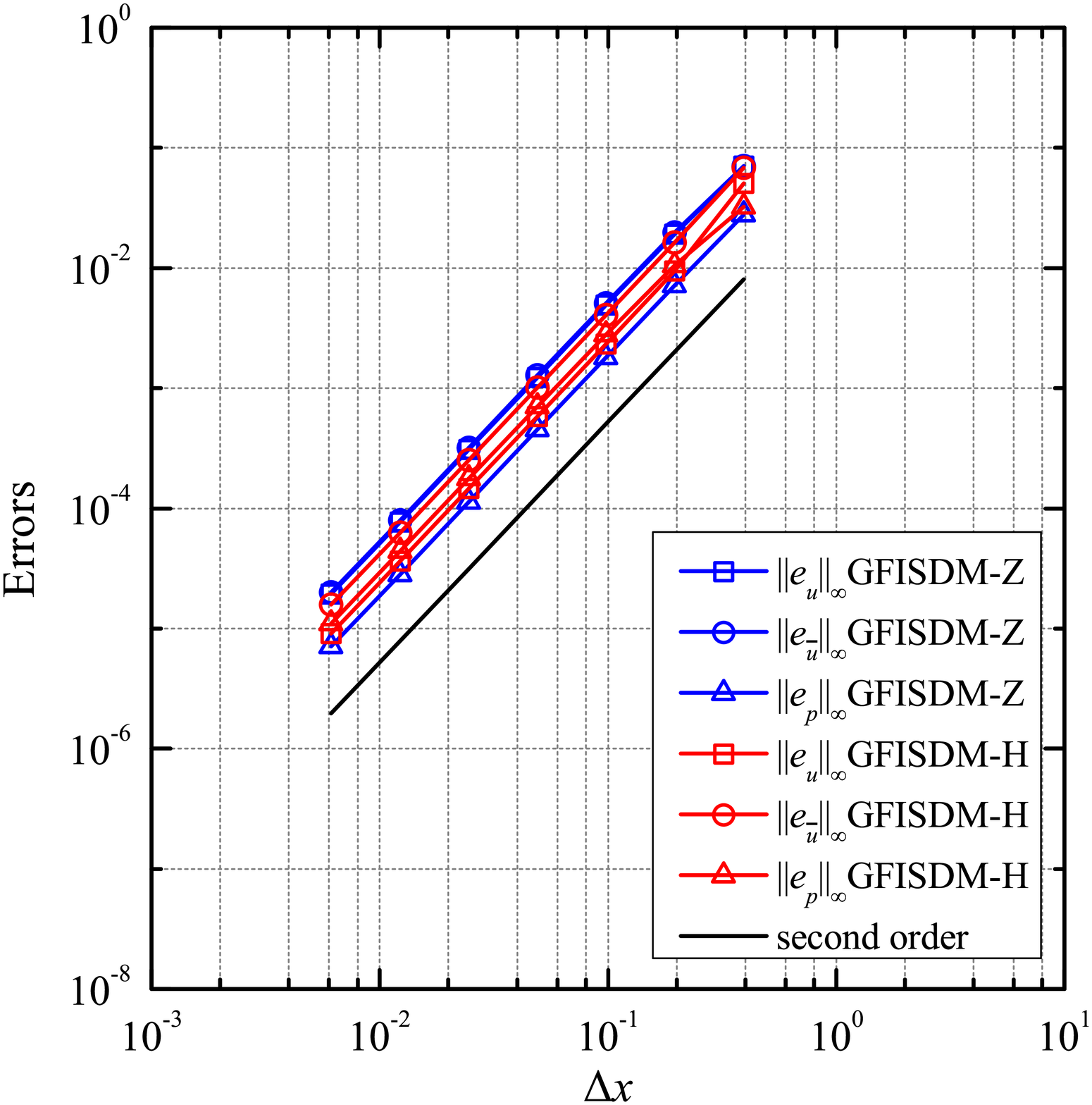}
        \caption{case \rom{2}}
    \end{subfigure}
     \begin{subfigure}[b]{0.33\linewidth}
        \includegraphics[width=0.9\linewidth]{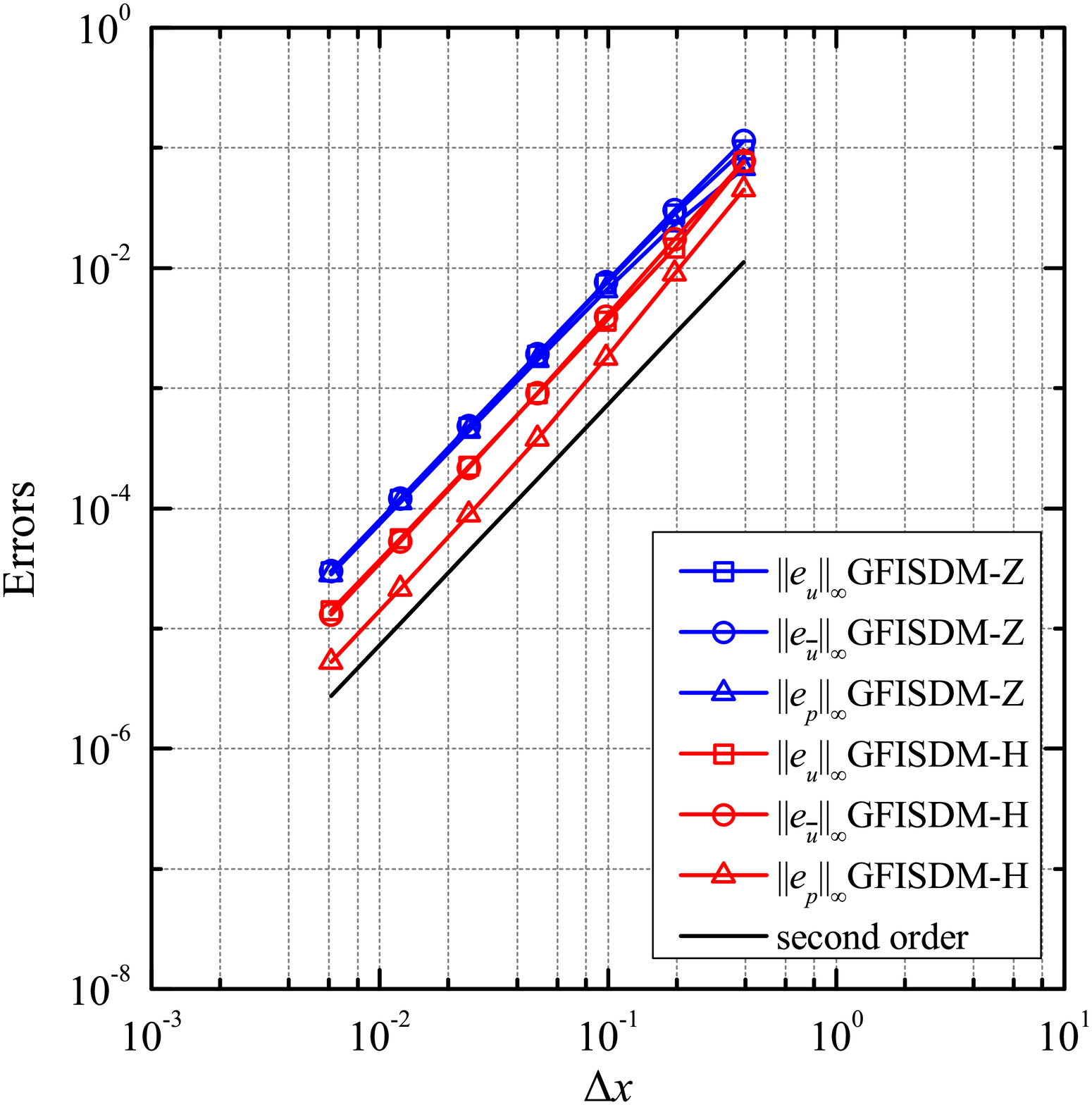}
        \caption{case \rom{3}}
    \end{subfigure}
    \begin{subfigure}[b]{0.33\linewidth}
        \includegraphics[width=0.9\linewidth]{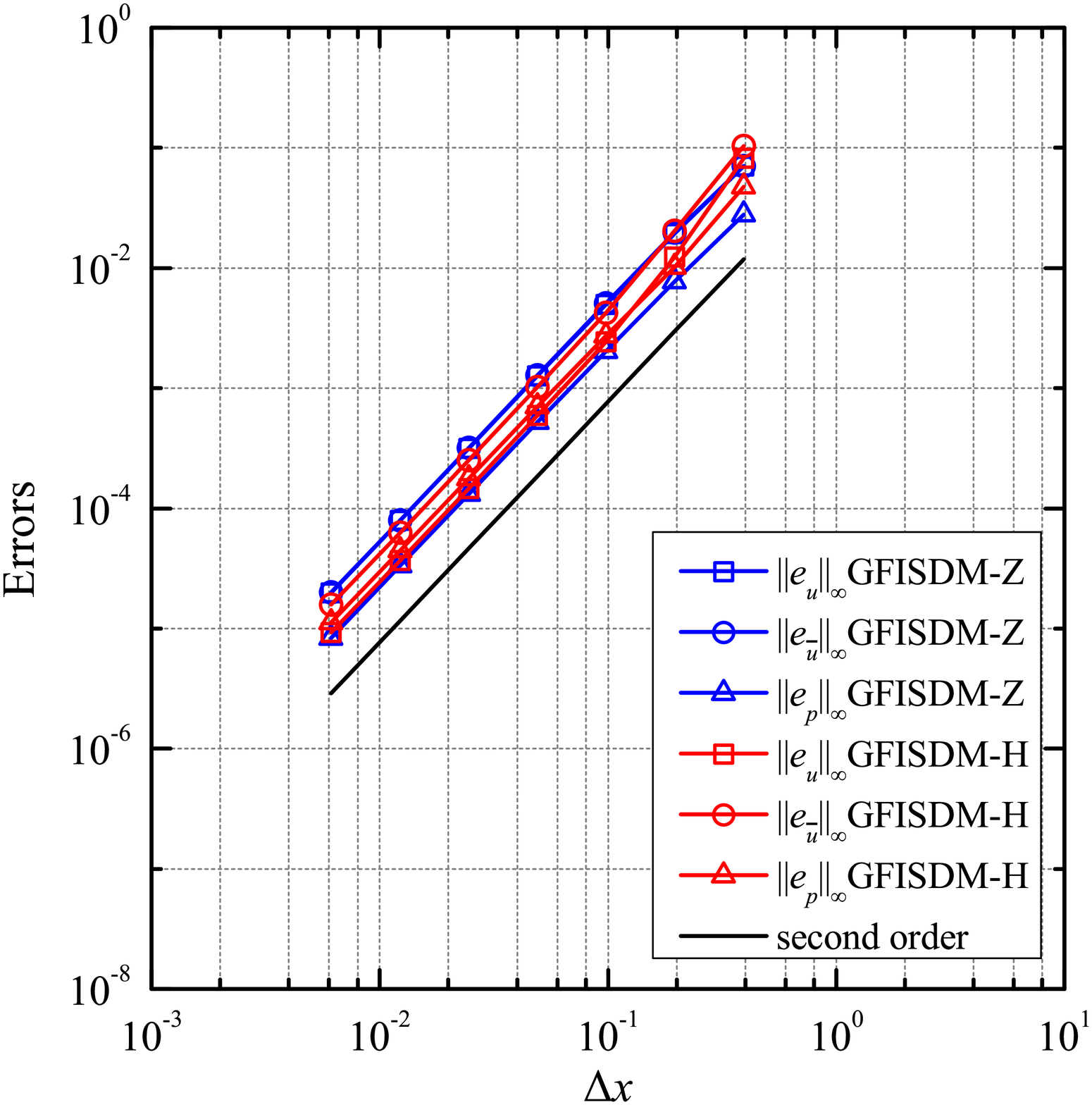}
        \caption{case \rom{4}}
    \end{subfigure}
      \caption{Convergence of the spatial discretization errors of the semi-discrete Navier-Stokes system for Taylor-Green Vortex at $t = 0$} \label{figure0}
\end{figure}

\begin{figure}[htb]
\centering
    \begin{subfigure}[b]{0.35\linewidth}
        \includegraphics[width=0.9\linewidth]{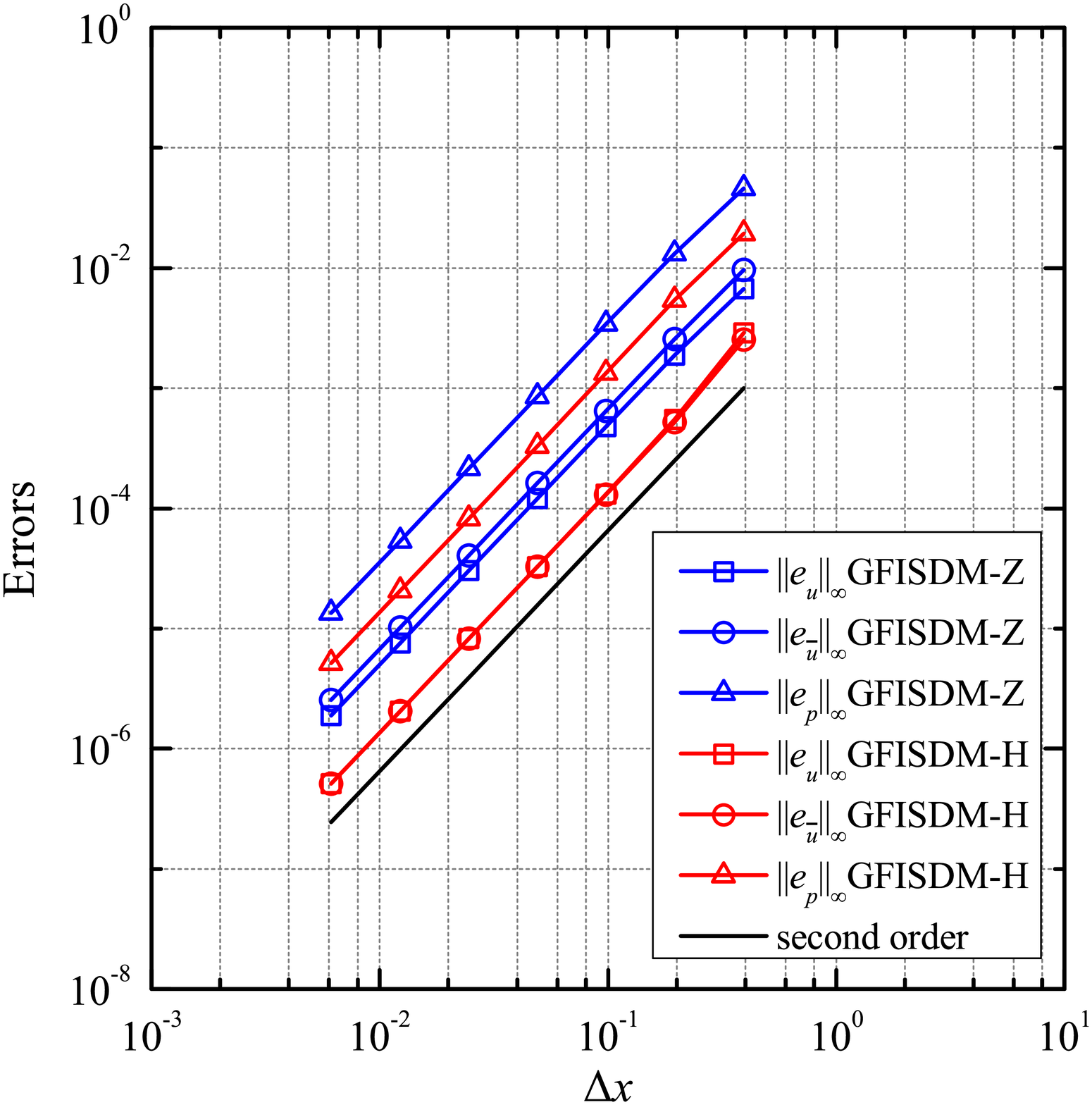}
        \caption{case \rom{1}}
     \end{subfigure}
    \begin{subfigure}[b]{0.35\linewidth}
        \includegraphics[width=0.9\linewidth]{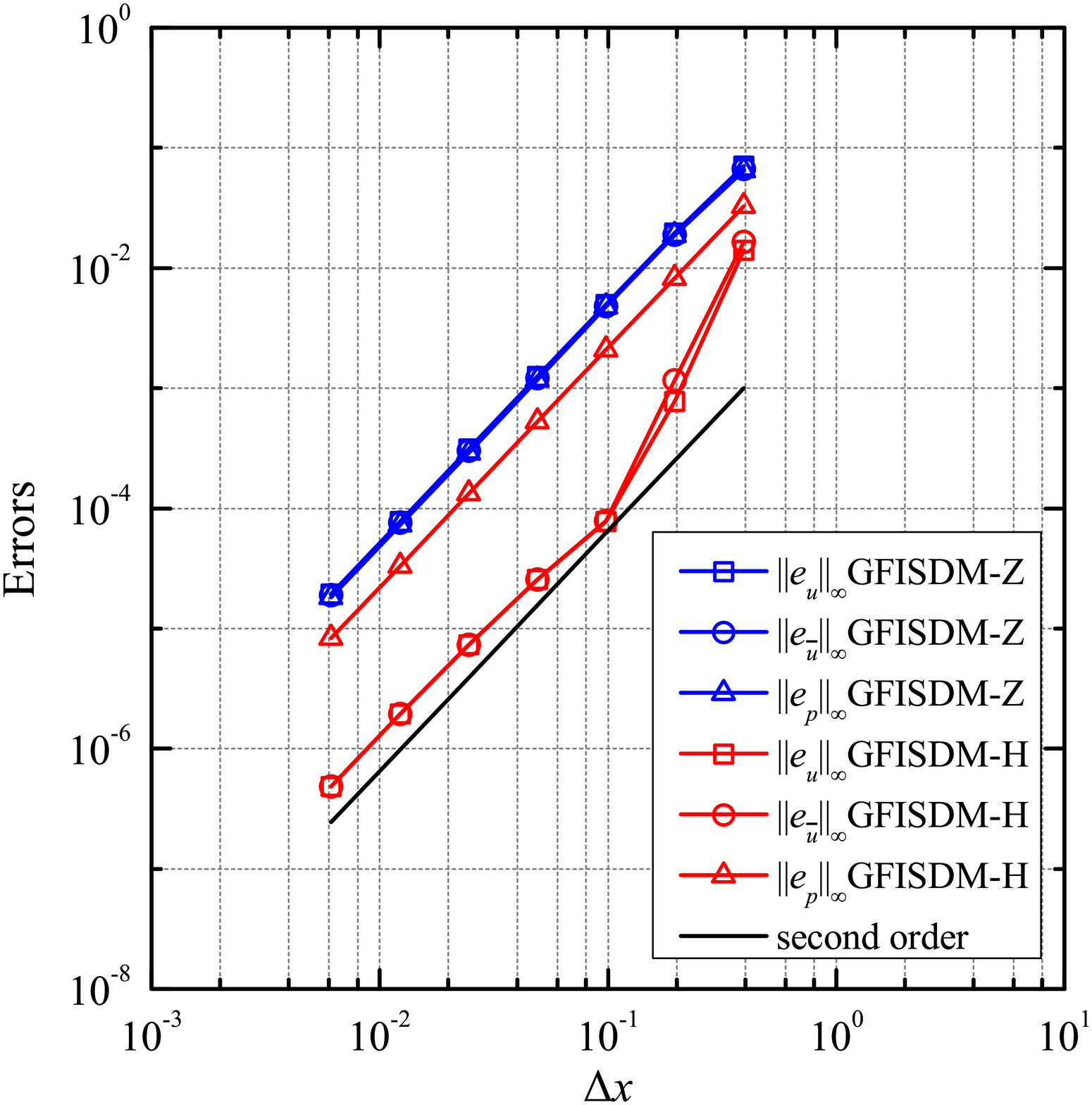}
        \caption{case \rom{2}}
    \end{subfigure}
     \begin{subfigure}[b]{0.35\linewidth}
        \includegraphics[width=0.9\linewidth]{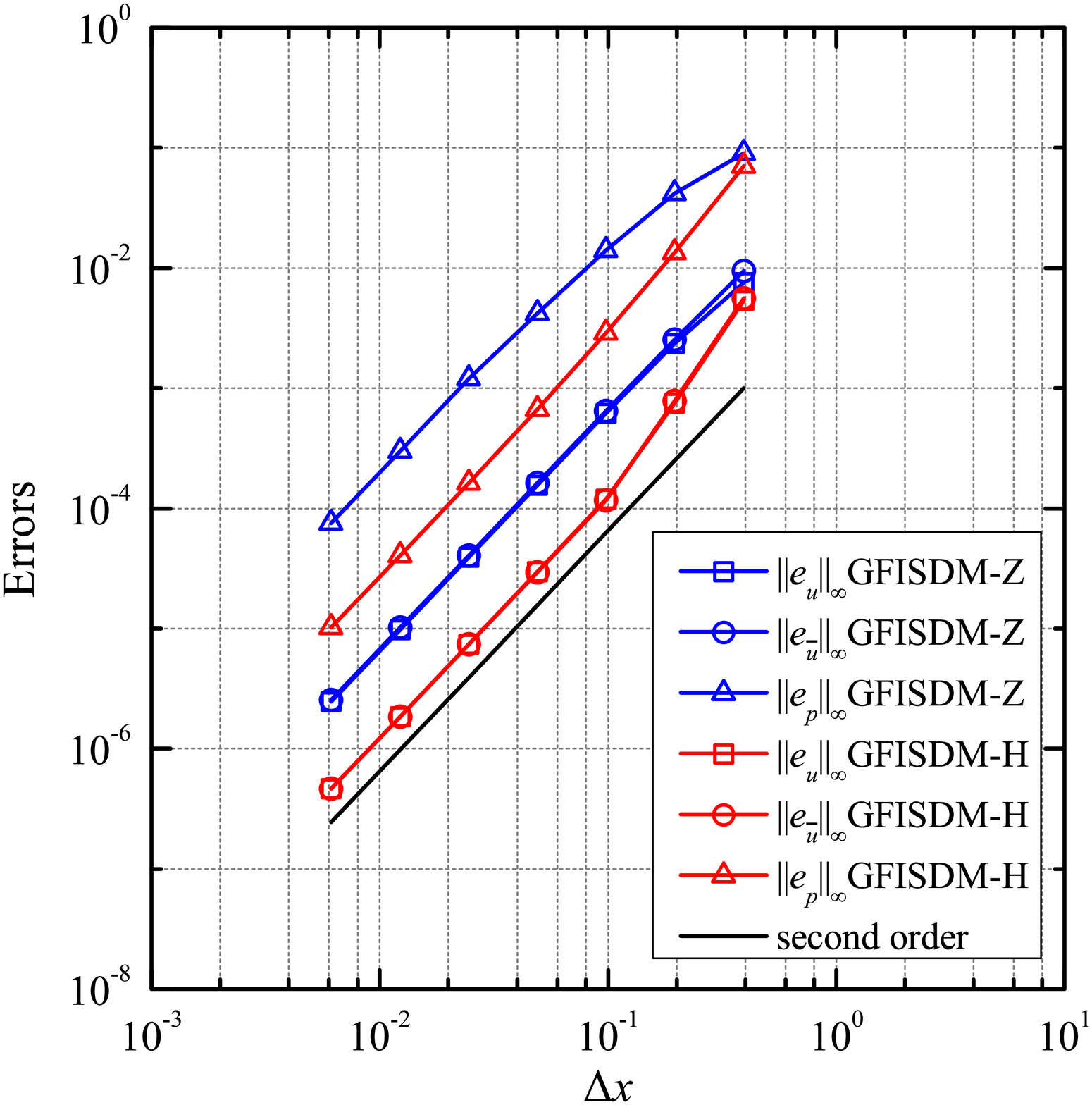}
        \caption{case \rom{3}}
    \end{subfigure}
    \begin{subfigure}[b]{0.35\linewidth}
        \includegraphics[width=0.9\linewidth]{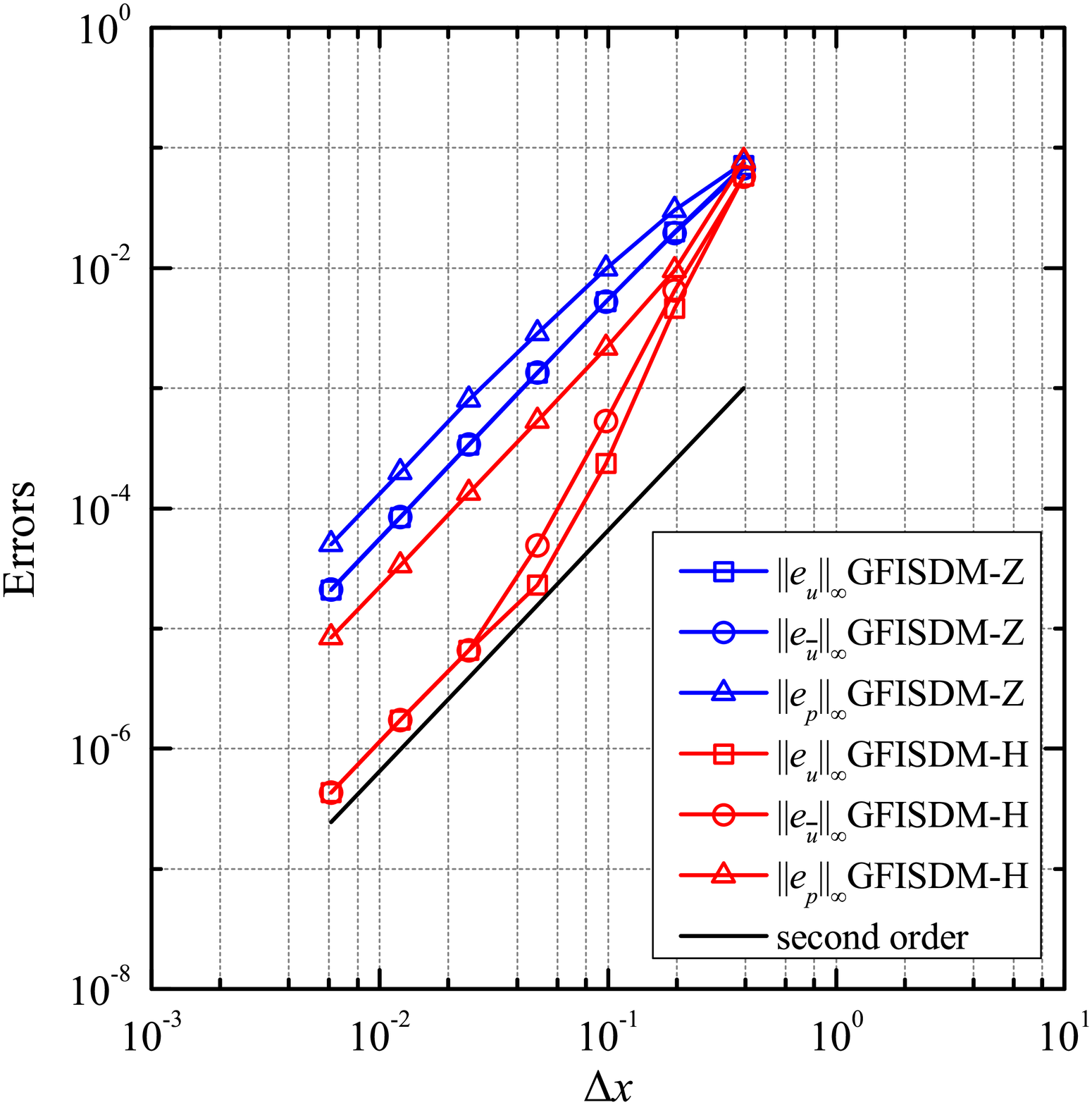}
        \caption{case \rom{4}}
    \end{subfigure}
      \caption{Convergence of the spatial errors for Taylor-Green Vortex at $t = 0.1/ \nu$} \label{figure2}
\end{figure}

\subsection{Verification of temporal accuracy}

To investigate the temporal accuracy of the proposed method, we simulate the Taylor-Green vortex on a uniform grid with $16\times16$ control volume and vary the total time-steps  from $8$ to $256$ steps. For simplicity, we only consider GFISDM-H for momentum interpolations. In all of the tests, we compute the temporal solutions of the semi-discrete incompressible Navier-Stokes system within the time interval $[0,0.1/ \nu]$. In this time duration, the magnitudes of velocities decrease by around 18\%. It is noted that the temporal error, considered here, represents the error in solving the semi-discrete incompressible Navier-Stokes system, i.e., the difference between the numerical and the exact solution of (\ref{system}). Considering the exact solution of (\ref{system}) is not available, it is approximated by the numerical solution from a simulation with SDIRK3 and a sufficiently small time-step size, i.e., $h = 0.1/2048\nu$.

To demonstrate the influence of the number of Picard iterations (denoted by $N_p$) on temporal accuracy, two different numbers of the Picard iteration are considered. For each Picard iteration loop, we use two pressure correction loops for the PISO algorithm. Assuming the computational cost of SDIRK2 and SDIRK3 for a single Runge-Kutta internal stage is close, the computational cost required by SDIRK3 to advance a time-step is approximately $1.5$ times larger than SDIRK2 since SDIRK2 and SDIRK3 are the two-stage and the three-stage schemes, respectively. Therefore, in order to compare the performance of SDIRK2 and SDIRK3 in terms of the temporal accuracy per computational cost, temporal error results are plotted against the time-step size divided by the Runge-Kutta stage numbers. 

First, we discuss case \rom{1} and case \rom{2} both of which employ periodic boundary conditions. The source term of the discrete continuity equation is $0$ in these two cases, and consequently the proposed method (\ref{RKformulation}) is identical to the direct approach. The temporal errors for case \rom{1} are demonstrated in Figure \ref{figure3}. The results indicate the classical order of convergence of $u$, $u'$ and $p$ for both SDIRK schemes with $N_p=$ 2 and 4. By comparing the error results in the plots, it is easy to conclude that SDIRK3 has a considerable advantage over SDIRK2 in terms of the temporal accuracy per computational cost over the entire arrange of time-step sizes considered in the tests. The temporal error results for case \rom{2} are demonstrated in Figure \ref{figure4}. While the velocity and pressure show the classical order of convergence for SDIRK2 with both Picard iteration numbers, the solutions of SDIRK3 attain the classical order of convergence only at the higher number of Picard iterations. The comparison results shown in Figure \ref{figure4} indicate that SDIRK2 and SDIRK3 achieve the very close accuracy per computational cost in temporal solutions when $N_p = 2$, while, SDIRK3 performs much better than SDIRK2 when $N_p = 4$.

\begin{figure}[htb]
\centering
    \begin{subfigure}[b]{0.35\linewidth}
        \includegraphics[width=0.9\linewidth]{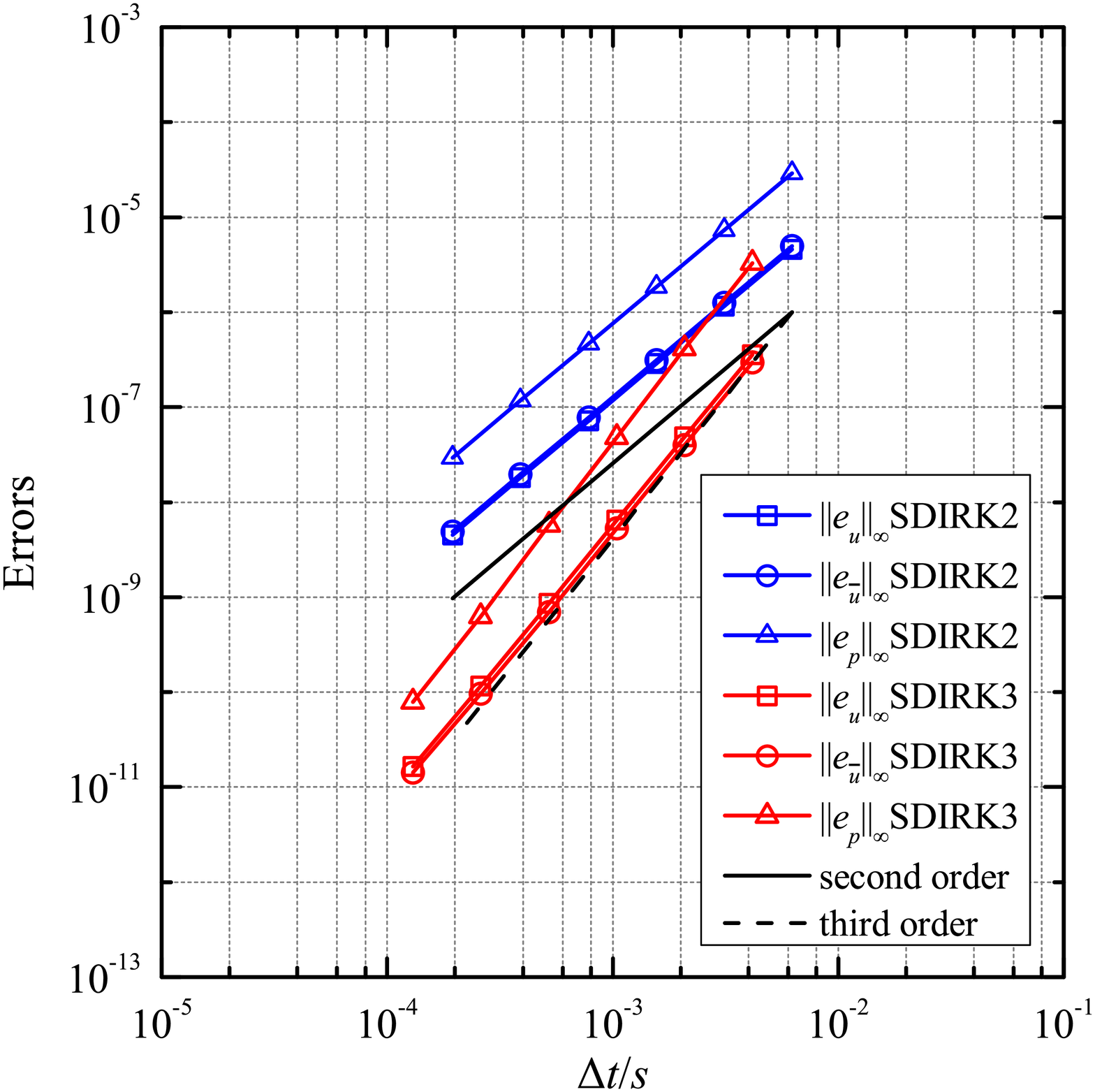}
        \caption{$N_p=2$}
     \end{subfigure}
    \begin{subfigure}[b]{0.35\linewidth}
        \includegraphics[width=0.9\linewidth]{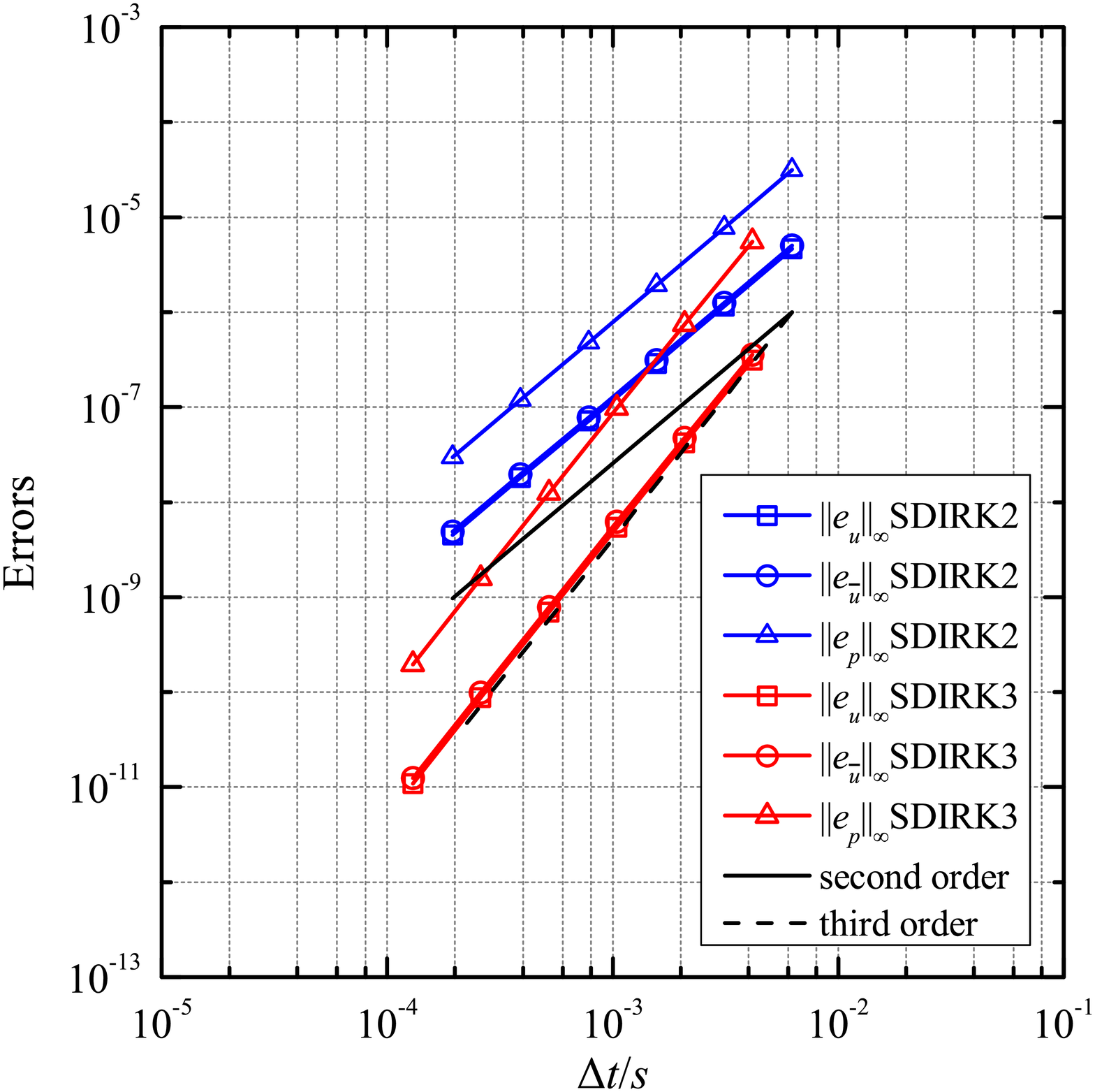}
        \caption{$N_p=4$}
    \end{subfigure}
      \caption{Convergence of the temporal errors for Taylor Green Vortex in case \rom{1} at $t = 0.1/ \nu$} \label{figure3}
\end{figure}

\begin{figure}[htb]
\centering
    \begin{subfigure}[b]{0.35\linewidth}
        \includegraphics[width=0.9\linewidth]{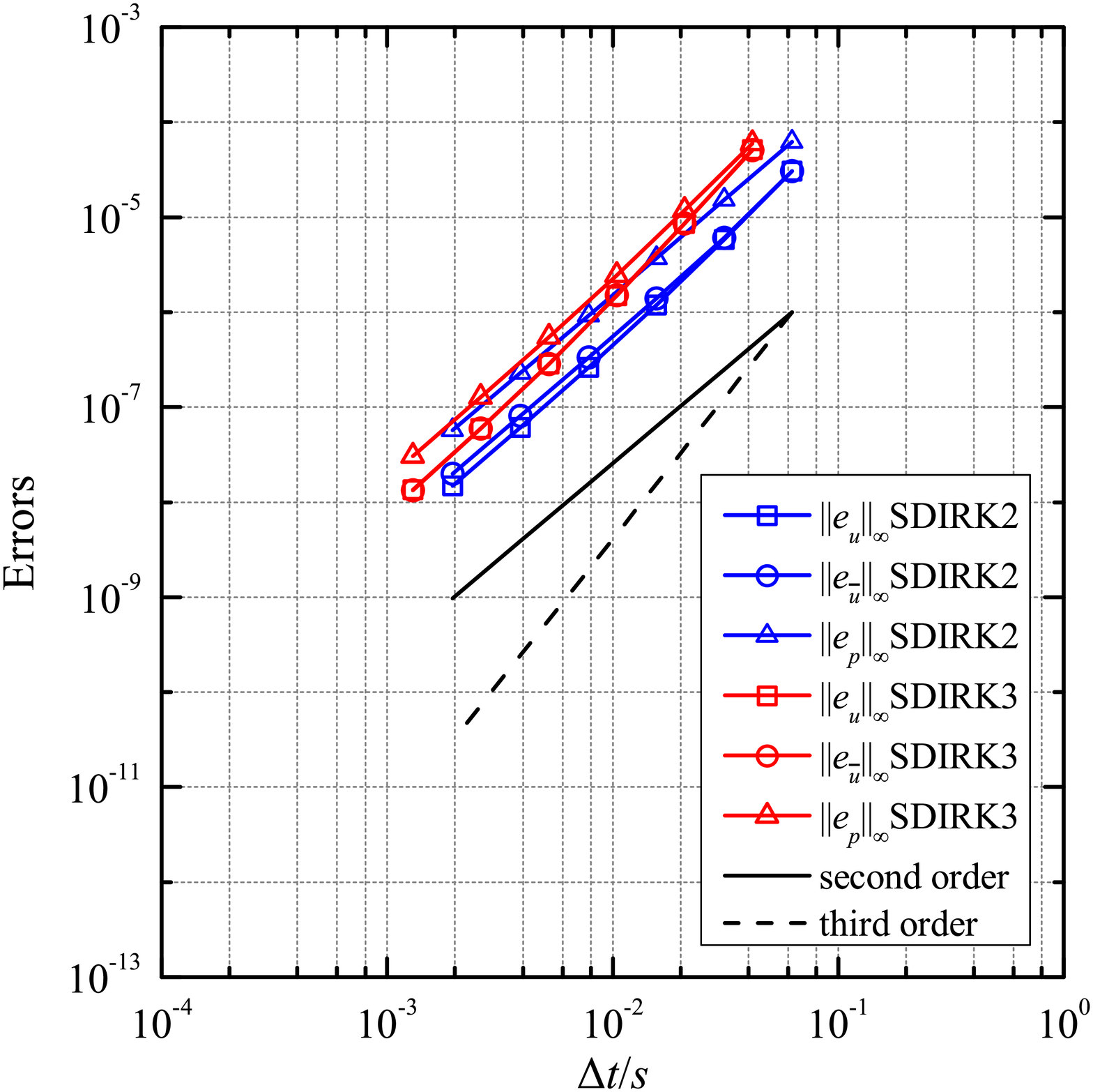}
        \caption{$N_p=2$}
     \end{subfigure}
    \begin{subfigure}[b]{0.35\linewidth}
        \includegraphics[width=0.9\linewidth]{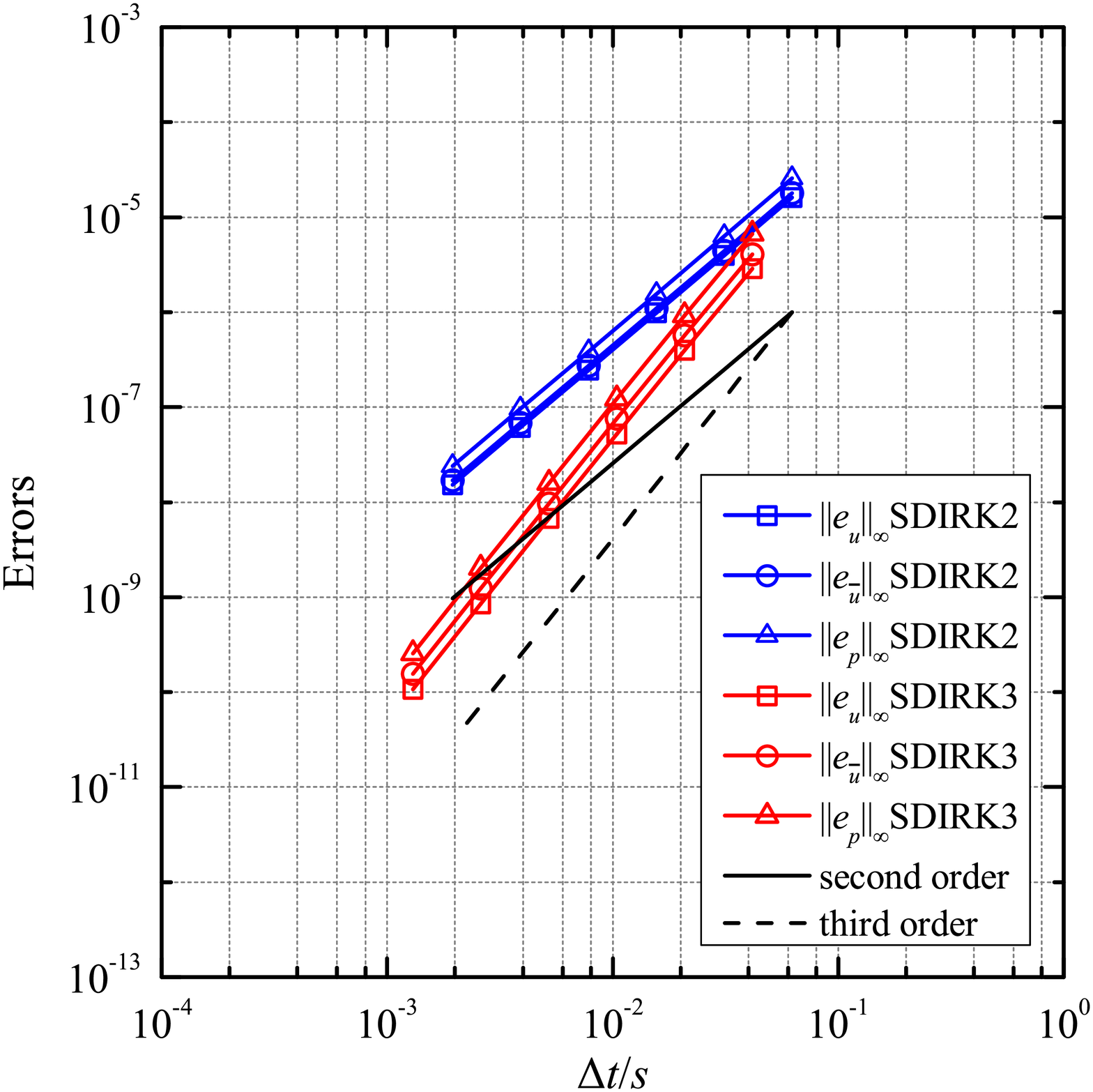}
        \caption{$N_p=4$}
    \end{subfigure}
      \caption{Convergence of the temporal errors for Taylor Green Vortex in case \rom{2} at $t = 0.1/ \nu$} \label{figure4}
\end{figure}

Then, let us focus on case \rom{3} and case \rom{4}. Both cases consider unsteady Dirichlet boundary conditions for velocities and zero-gradient boundary conditions for pressures. The source term of the discrete continuity equation is a vector of non-zero time-varying elements in those two cases, and consequently the proposed method is different from the direct approach. The temporal errors for case \rom{2} and case \rom{4} are shown in Figure \ref{figure5} and Figure \ref{figure6} respectively. The results indicate that SDIRK2 achieves its classical order (i.e., second-order) convergence for the velocity and pressure when $N_p = 2$, while SDIRK3 delivers the classical order (i.e., third-order) convergence until $N_p = 4$. This means SDIRK3, as the higher-order scheme, requires more Picard iterations to reach its expected order of convergence. For case \rom{3}, SDIRK3 has the higher accuracy per computational cost than SDIRK2 with $N_p=$ 2 or 4. However, for case \rom{4}, SDIRK2 and SDIRK3 are well matched in the accuracy per computational cost when $N_p = 2$, but SDIRK3 reclaims the accuracy advantage over SDIRK2 when $N_p = 4$. To examine the superiority of the proposed method in the term of convergence results, we demonstrate the temporal errors of the direct approach in Figure \ref{figure7} in which the velocity attains the classical order of convergence, while, the pressure is only first-order accurate for both SDIRK schemes. 

\begin{figure}[htb]
\centering
    \begin{subfigure}[b]{0.35\linewidth}
        \includegraphics[width=0.9\linewidth]{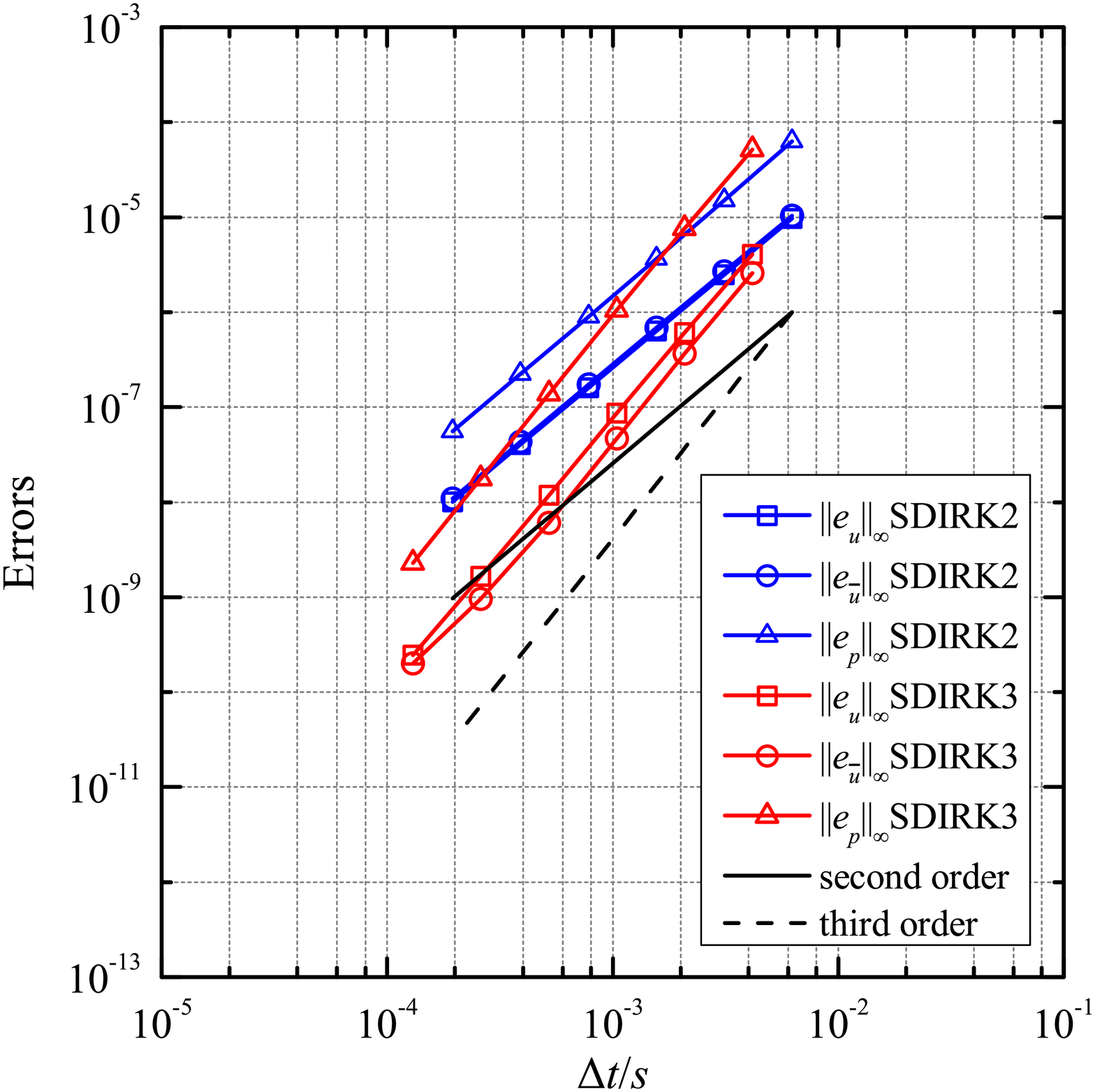}
        \caption{$N_p=2$}
     \end{subfigure}
    \begin{subfigure}[b]{0.35\linewidth}
        \includegraphics[width=0.9\linewidth]{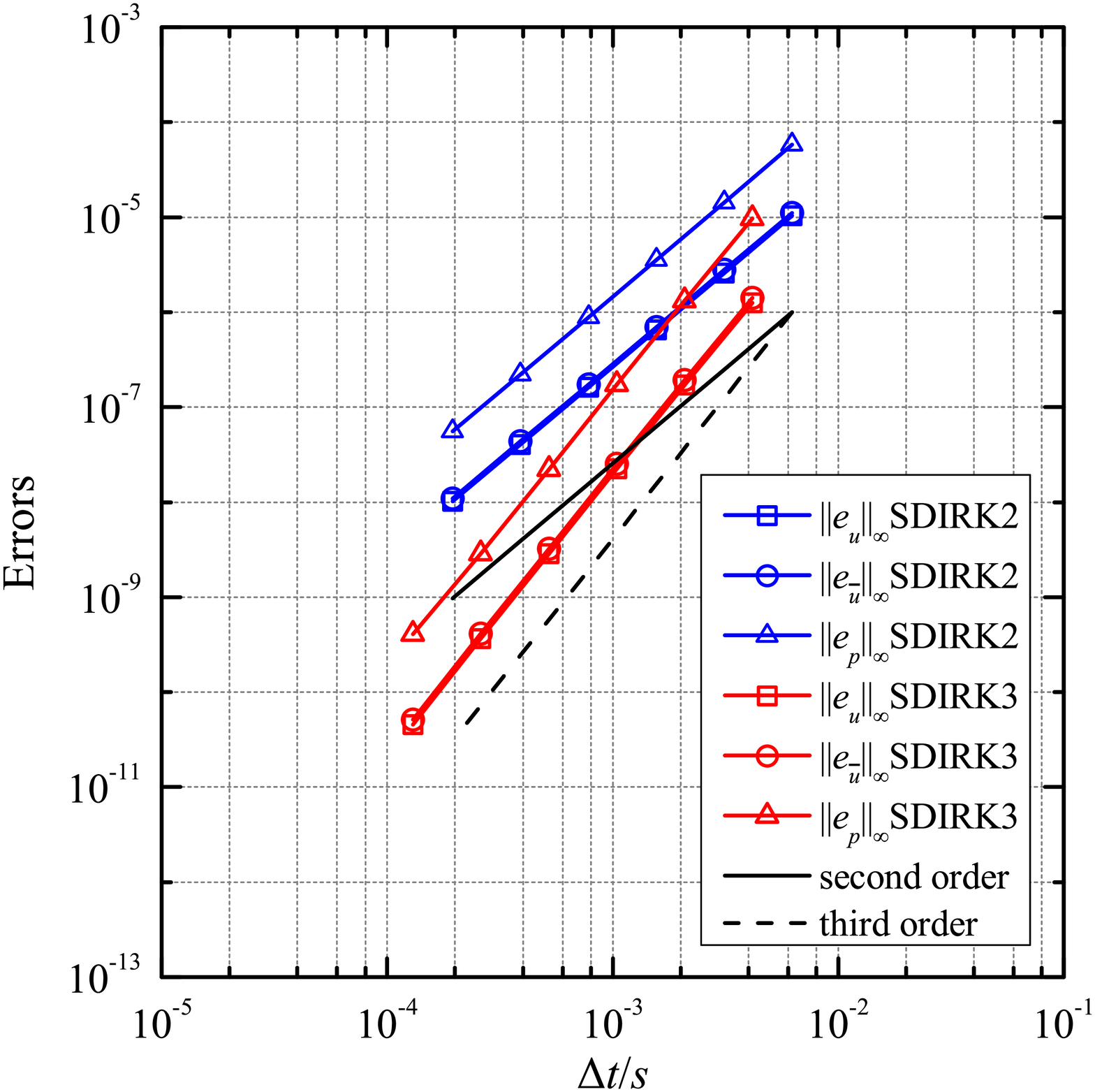}
        \caption{$N_p=4$}
    \end{subfigure}
      \caption{Convergence of the temporal errors for Taylor Green Vortex in case \rom{3} at $t = 0.1/ \nu$} \label{figure5}
\end{figure}

\begin{figure}[htb]
\centering
    \begin{subfigure}[b]{0.35\linewidth}
        \includegraphics[width=0.9\linewidth]{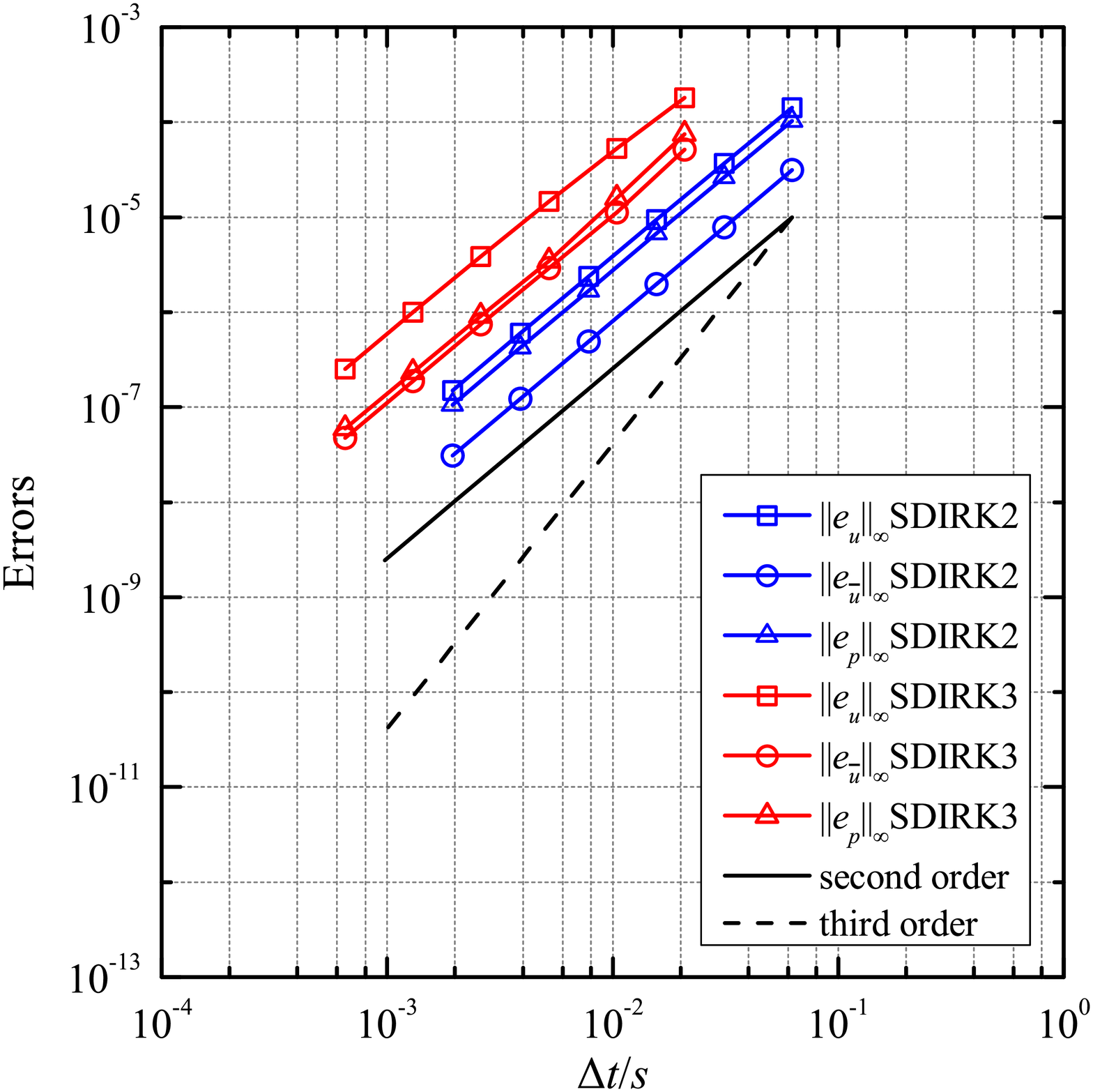}
        \caption{$N_p=2$}
     \end{subfigure}
    \begin{subfigure}[b]{0.35\linewidth}
        \includegraphics[width=0.9\linewidth]{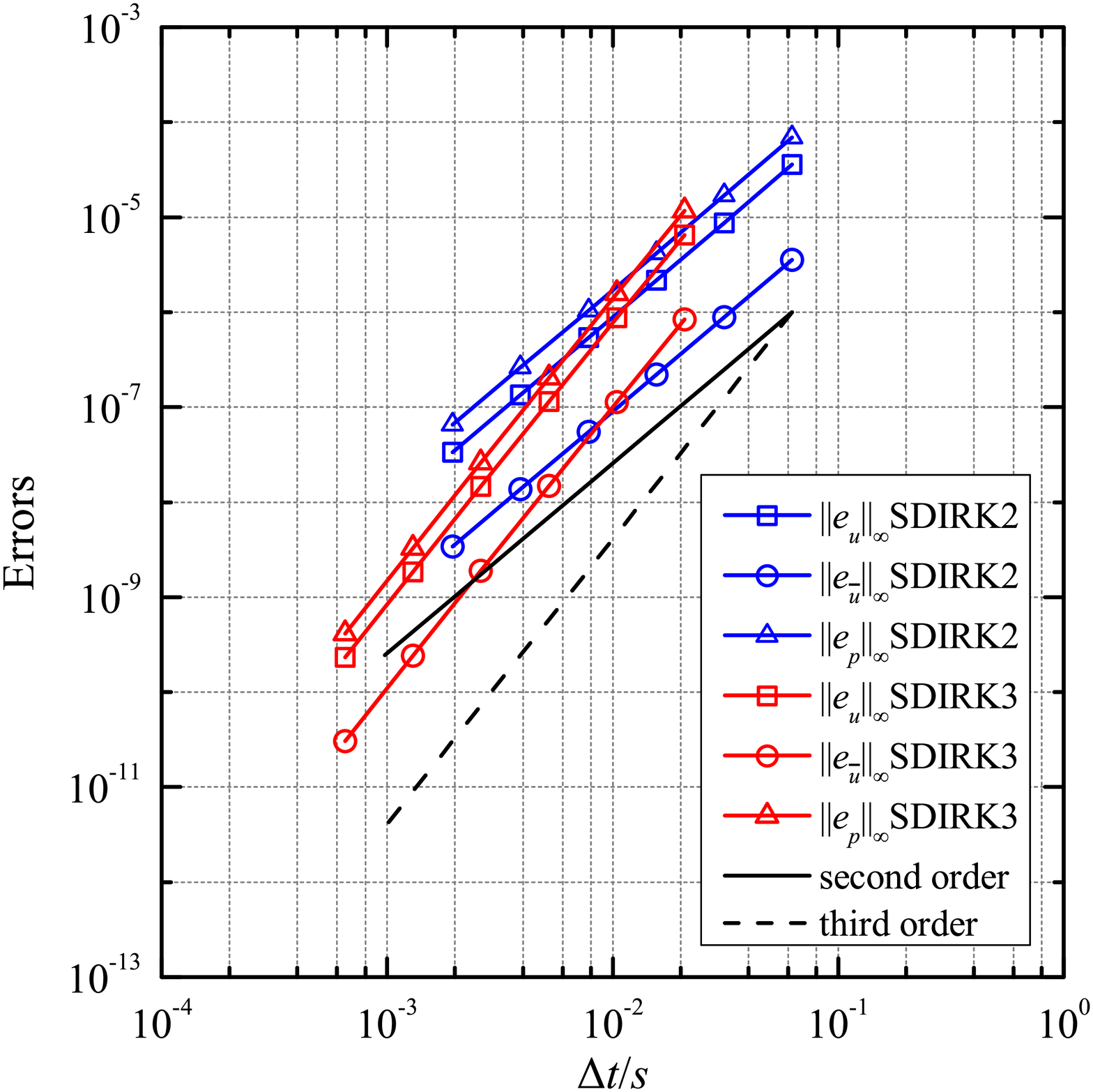}
        \caption{$N_p=4$}
    \end{subfigure}
      \caption{Convergence of the temporal errors for Taylor Green Vortex in case \rom{4} at $t = 0.1/ \nu$} \label{figure6}
\end{figure}

\begin{figure}[htb]
\centering
    \begin{subfigure}[b]{0.35\linewidth}
        \includegraphics[width=0.9\linewidth]{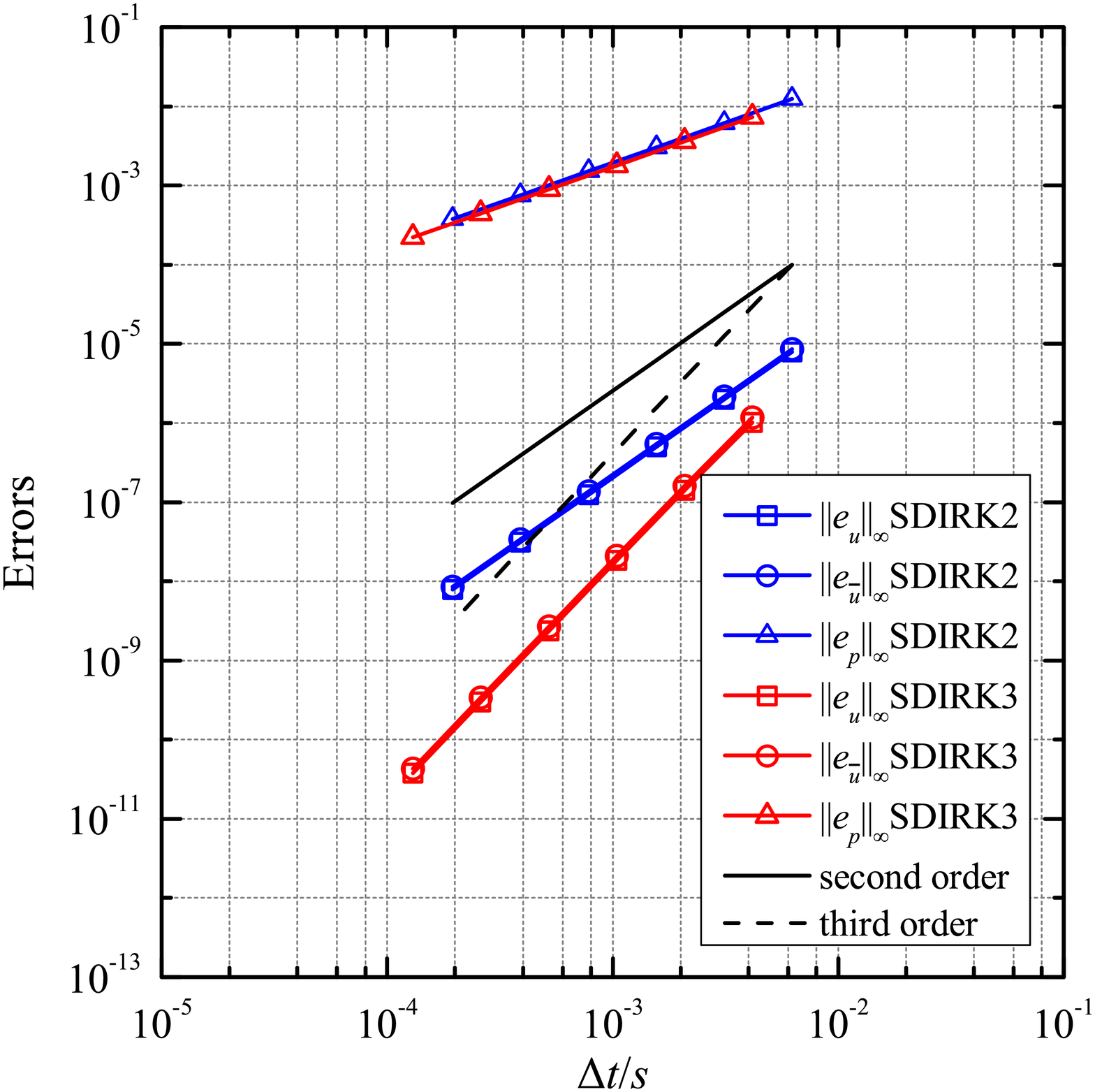}
        \caption{case \rom{3}}
     \end{subfigure}
    \begin{subfigure}[b]{0.35\linewidth}
        \includegraphics[width=0.9\linewidth]{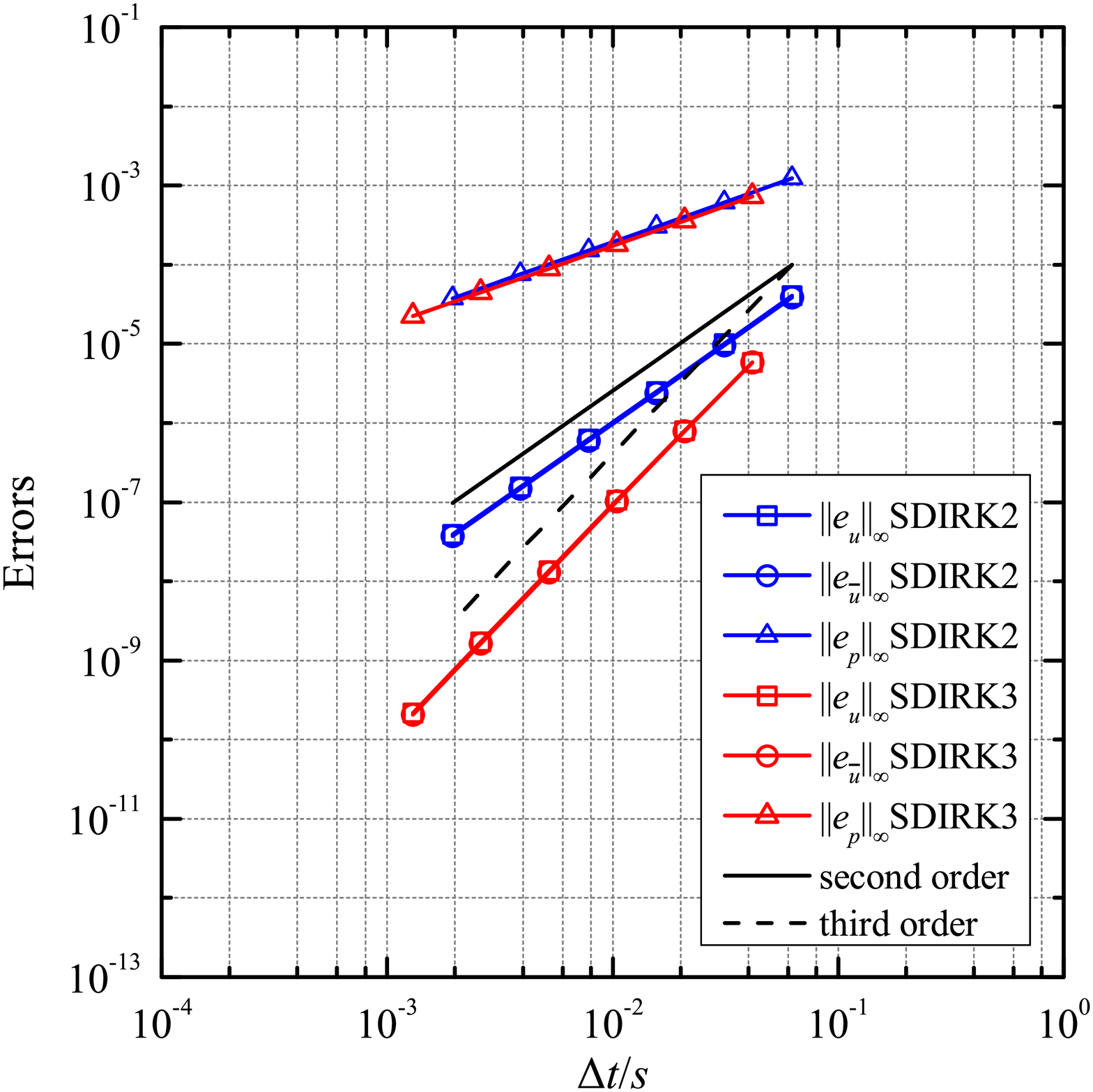}
        \caption{case \rom{4}}
    \end{subfigure}
      \caption{Convergence of the temporal errors when SDIRK schemes are implemented with the direct approach and $N_p =4$} \label{figure7}
\end{figure}

It should be emphasized that all temporal error results demonstrated in this subsection are obtained by solving the semi-discrete incompressible Navier-Stokes systems based on GFISDM-H. In addition, we have a statement presented in Section \ref{section2.3}, related to the order reduction of temporal errors for Choi's momentum interpolation method, need to be validated. Therefore, we re-simulate case \rom{2}$\sim$\rom{4}, in which the diagonal elements of the matrices $K$ and $N$,  using GFISDM-Choi for the evaluation of the time derivative of cell-face velocities. The error results shown in Figure \ref{figure8} demonstrate the first-order convergence of the temporal errors which validates the statement.

\begin{figure}[htb]
\centering
    \begin{subfigure}[b]{0.27\linewidth}
        \includegraphics[width=0.9\linewidth]{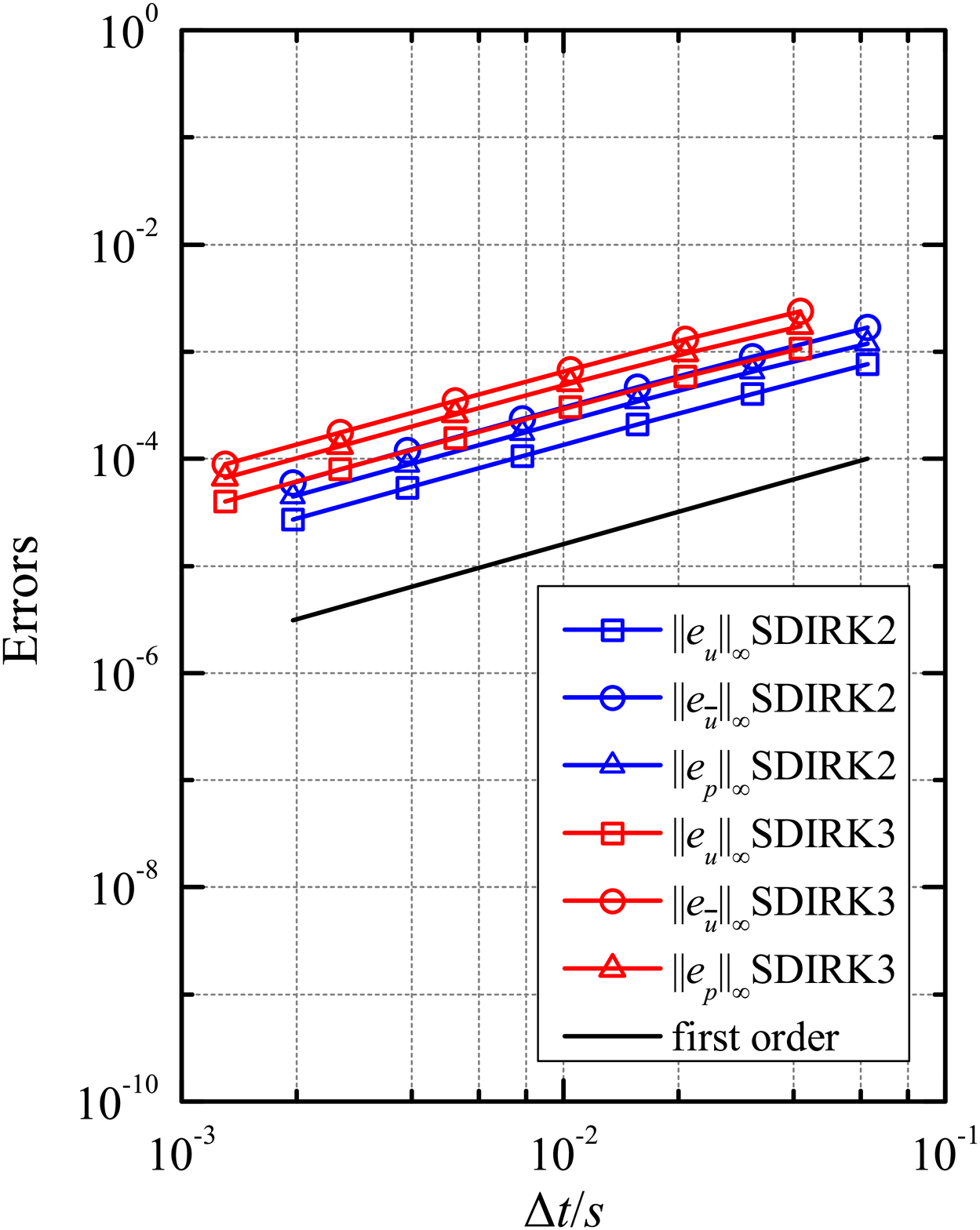}
        \caption{case \rom{2}}
     \end{subfigure}
    \begin{subfigure}[b]{0.27\linewidth}
        \includegraphics[width=0.9\linewidth]{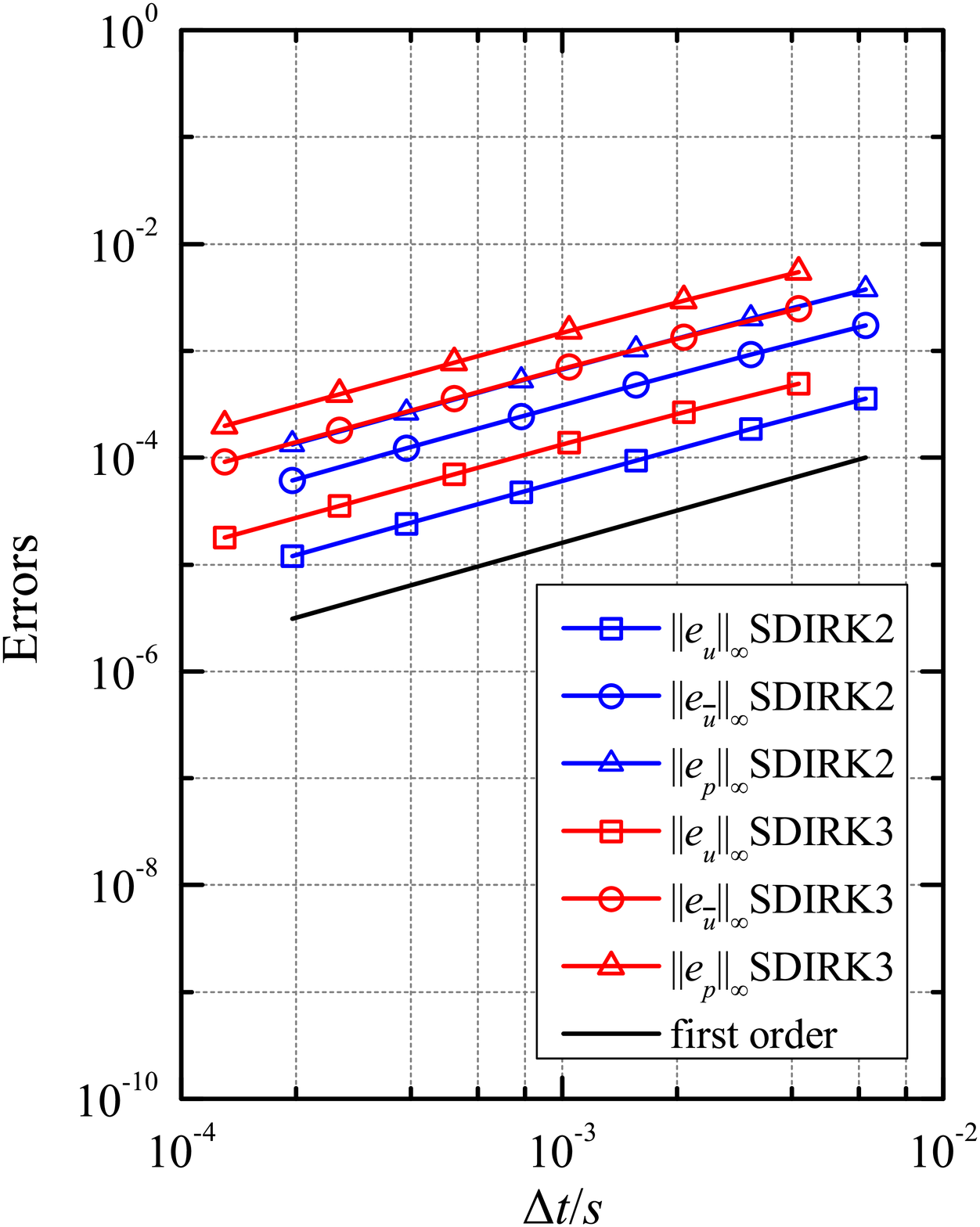}
        \caption{case \rom{3}}
    \end{subfigure}
    \begin{subfigure}[b]{0.27\linewidth}
        \includegraphics[width=0.9\linewidth]{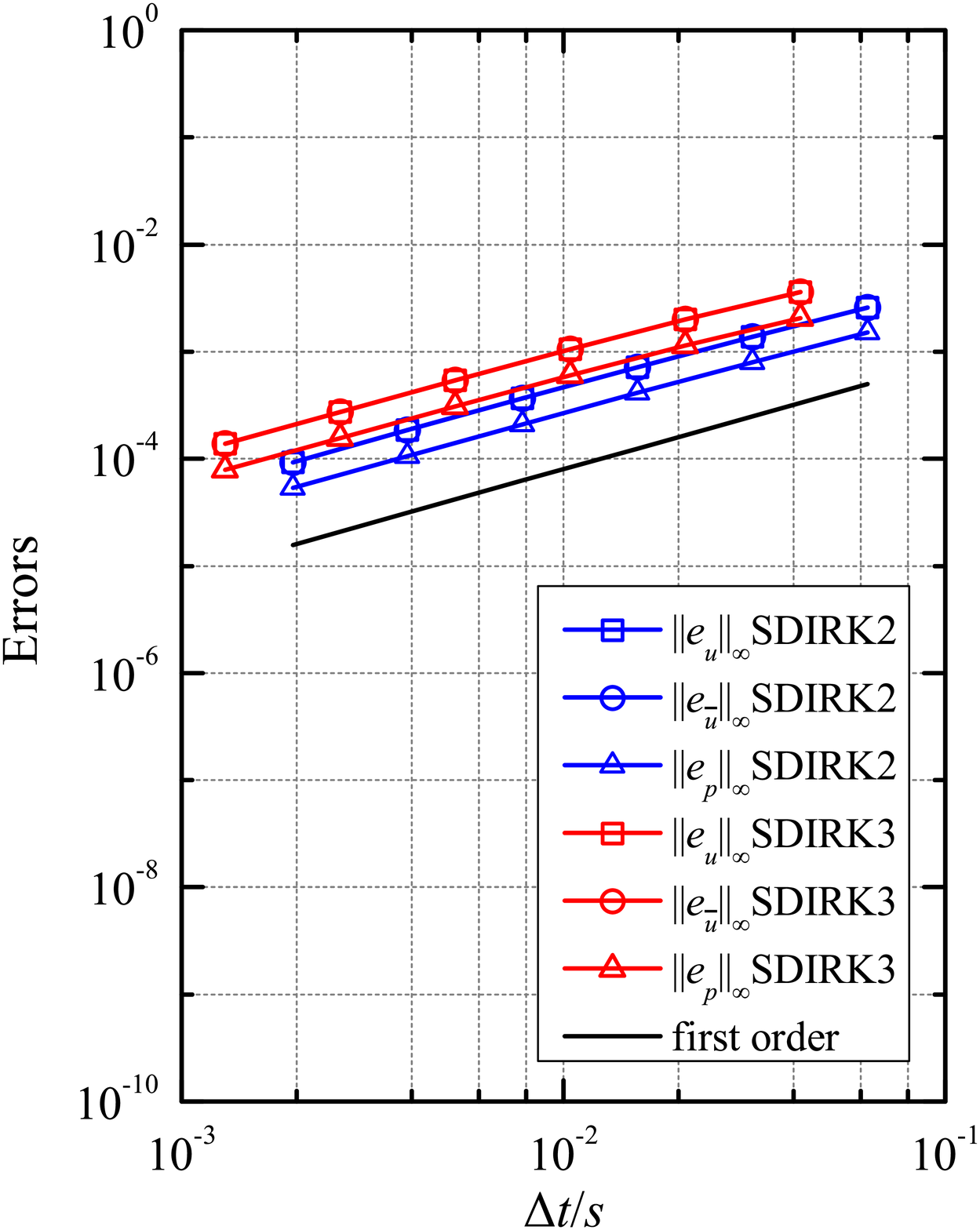}
        \caption{case \rom{4}}
    \end{subfigure}
      \caption{Convergence of the temporal errors when the momentum interpolation utilizing GFISDM-Choi} \label{figure8}
\end{figure}

\section{Discussion \& Concluding Remarks}

This study focuses on the formulation and solution of the incompressible Navier-Stokes equations on a collocated grid system. For the spatial discretization of the incompressible Navier-Stokes equations, we establish a new momentum interpolation framework for computing the velocities on cell-faces. The basic idea of the proposed framework is that we approximate the time derivative of cell-face velocities through a custom interpolation of the semi-discrete momentum equation of cell-center velocities, and discretize the approximation in time to obtain the values of cell-face velocities at discrete time points. It is noted that the finite volume discretization of the incompressible Navier-Stokes equations leads to a system of differential and algebraic equations with respect to the cell-center velocity and pressure. This system alone does not define a closed form solution and consequently cannot be solved. To address this issue, the momentum interpolation, in the proposed framework, introduces the differential equation of the cell-face velocity to the original system, by which the enlarged system has a closed form solution and can be strictly viewed as an index 2 differential-algebraic problem. The essential procedure of the proposed momentum interpolation framework is the evaluation of the time derivative of cell-face velocities. To accomplish the procedure, we develop a new momentum interpolation method named generalized face interpolation of the semi-discrete momentum equation, i.e., the GFISDM schemes.

In the other momentum interpolation methods available in literature, the original fully-discrete momentum equation of cell-center velocities normally goes through a series of transformation to deliver a rearranged equation. Then, the pseudo momentum equation of cell-face velocities is assumed to be consist of terms which are identical to those of the reorganized momentum equation of cell-center velocities. Finally, the unknown terms (which are face-centered scalar or vector field) in the momentum equation of cell-face velocities are approximated by custom interpolations of their counterparts (which are cell-centered scalar or vector field) in the momentum equation of cell-center velocities. In those momentum interpolation methods, the fields that interpolated are identical for each cell-face. 

In GFISDM schemes, we assume that by establishing a pseudo grid system regarding the values of flow variables at cell-face centroids, and discretizing the incompressible Navier-Stokes equation on it with a finite volume method, we obtain a semi-discrete momentum equation, i.e., (\ref{daem00}), of the face-centered velocity field $\bar{u}$. The semi-discrete equation of $\bar{u}$ is analogous to that of the cell-centered velocity field $u$. In both two semi-discrete momentum equations, the discretized diffusion and convection terms are expressed as the summation of the coefficient matrices (which are $K$ and $N$ for $u$, and $\bar{K}$ and $\bar{N}$ for $\bar{u}$) multiplied by the corresponding discretized velocities with the source terms (which are $b$ for $u$, and $\bar{b}$ for $\bar{u}$). In reality, the pseudo grid system does not exist, and consequently the $\bar{K}$, $\bar{N}$ and $\bar{b}$ terms of the $\bar{u}$ equation cannot be given explicitly. Therefore, we need to approximate the discretized diffusion and convection terms of cell-face velocities, i.e., the $\{\bar{K}+\bar{N}\}\bar{u}+\bar{b}$ field, through a custom face interpolation of cell-center velocities counterpart, i.e., the $\{K+N\}u+b$ field. Before the interpolation, we divide the $\{\bar{K}+\bar{N}\}\bar{u}+\bar{b}$ field into the terms $\{\text{diag}(\bar{A})\}\bar{u}$ and $\bar{H}$ where $\bar{A}$ and $\bar{H}$ are discrete face-centered scalar and vector fields, respectively. Each element of the $\bar{A}$ field is expressed as a linear combination of the corresponding element of $\text{diag}(K)$ and $\text{diag}(N)$ with the vectors $\alpha$ and $\beta$ storing the related coefficients. Hence, both the unknown $\bar{A}$ and $\bar{H}$ fields are $\alpha$- and $\beta$-dependent. 

To approximate the $i$-th elements of the $\bar{A}$ and $\bar{H}$ fields (i.e., $\bar{A}^i$ and $\bar{H}^{\boldsymbol{i}}$), we divide the $\{K+N\}u+b$ field into the terms $\{\text{diag}(A_i)\}u$ and $H_i$ with the expression (\ref{daem}) where $A_i$ and $H_i$ are discrete cell-centered scalar and vector fields, respectively. Then, we interpolate the $A_i$ and $H_i$ fields locally to face $i$, i.e., approximating the $\bar{A}^i$ and $\bar{H}^{\boldsymbol{i}}$ terms with $\eta^i(A_i)$ and $\boldsymbol{\eta}^{\boldsymbol{i}}(H_i)$, respectively. Therefore, different from the other momentum interpolation methods, the fields that interpolated to cell-faces are face-dependent in GFISDM schemes. The key factor leads us to enable such an utility in our method is because it is almost impossible, for a GFISDM scheme utilizing non-zero identical values for $\alpha$ and $\beta$, to attain the second order accuracy and convergence globally under all sorts of numerical setup. To validate this statement, we investigate the accuracy and convergence of the GFISDM schemes on Cartesian or curvilinear structured grid systems by considering the $\tilde{A}_i(\zeta)$ and $\tilde{H}_i(\zeta)$ functions which can be regarded as the continuous counterparts of the discrete $A_i$ and $H_i$ fields. The $\tilde{A}_i$ and $\tilde{H}_i$ functions are defined on a curvilinear axis $\zeta$ (which is established by linking the centroids of the cells involved in the interpolation regarding cell-face $i$ with a sufficiently smooth curve) with the domain $X_i$. We show that the smoothness and boundedness of the $\tilde{A}_i$ and $\tilde{H}_i$ functions on $X_i$ has a significant impact on the accuracy and convergence of the GFISDM schemes, and summarize two situations in which the schemes suffer from the loss of accuracy and order reduction.

The first situation is the $\tilde{A}_i$ function is discontinuous or non-differentiable on $X_i$. Once happens, the local interpolation of the $A_i$ and $H_i$ fields to the centroid of face $i$ cannot deliver accurate approximations for $\bar{A}^i$ and $\bar{H}^{\boldsymbol{i}}$. Hence, the accuracy and convergence of the GFISDM schemes will be damaged regardless of how accurate the interpolation scheme used for $\eta$ is. The second situation is the spatial derivatives of the $\tilde{A}_i$ function are unbounded as the grid spacing approaches to zero. The unboundedness of the $\tilde{A}_i$ and $\tilde{H}_i$ functions indicates the order reduction of the orders of convergence for the face interpolations of the $A_i$ and $H_i$ fields. Thus, the commonly used second-order linear interpolation scheme may not be accurate enough to ensure the second-order convergence of the GFISDM schemes.

The discussion in Section \ref{section2.4} indicates that the easiest approach for a GFISDM scheme to attain second-order convergence under all sorts of numeral setup is employing the $\alpha$ and $\beta$ vectors with all elements being zero (i.e., GFISDM-Z). This scheme approximates the $\{\bar{K}+\bar{N}\}\bar{u}+\bar{b}$ field through the face interpolation of the $\{H\}u+b$ field directly, therefore neither of the aforementioned situation occurs in it as along as the $\tilde{A}_i\tilde{u}+\tilde{H}_i$ term defined for each cell-face is sufficiently differentiable on its domain $X_i$. Consequently, GFISDM-Z delivers a second order convergence if the finite volume discretization of the momentum equation itself is convergent of order two and an interpolation scheme with order of accuracy no less than two is utilized in it. In contrast, it is much more difficult to achieve the same performance in order results when the GFISDM has non zero-elements in $\alpha$ or $\beta$. For this type of GFISDM schemes to attain the second order convergence, the first priority is to compensate order reduction caused by the unboundedness of the $\tilde{A}_i$ or $\tilde{H}_i$ functions using higher order interpolation schemes, as this issue affects the convergence results globally. Then, the GFISDM should make sure the $\tilde{A}_i$ and $\tilde{H}_i$ functions are sufficiently differentiable on their domain $X_i$ for all cell-faces. To achieve this goal, we just need to set $\alpha^i$ or $\beta^i$ to constant zero if the $\tilde{A}_{i,K}$ or $\tilde{A}_{i,N}$ function is discontinuous or non-differentiable on $X_i$ within the solution time domain. Based on the above requirements, we develop a GFISDM scheme (i.e., GFISDM-H) which preserves the characteristics of GFISDM-Yu as much as possible without suffering from the order reduction.

The proposed momentum interpolation framework offers an alternative approach to examine if a momentum interpolation method can attain the time-step size-independent converged solution for steady-state flows. To carry out the examination, we first seek for particular $\alpha$, $\beta$ and $\eta$ with which the GFISDM once discretized in time can reproduce the scheme of the momentum interpolation method that needs to be examined. Then, by examining the time-step size-dependency of $\eta$, we can easily determine if the converged solutions are time-step size-dependent. Using this approach, we examined the momentum interpolation methods of \cite{choi1999}, \cite{yu2002} and \cite{pascau2011} and denote the GFISDM schemes corresponding to those methods as GFISDM-Choi, GFISDM-Yu and GFISDM-Pascau respectively. The results indicate that, unlike the other two momentum interpolation methods, Choi's scheme is not unconditionally time-step size-independent. More specifically, when the diagonal elements of $K$ or $N$ are not identical, GFISDM-Choi is time-step size-dependent. Moreover, GFISDM-Choi can be regarded as a perturbed system of GFISDM-Yu with the perturbations of $O(h)$, and it converges to GFISDM-Yu as the time-step size $h$ approaches to zero. This finding further indicates that if Choi's momentum interpolation scheme is used for unsteady flow problems and the elements of $\text{diag}(K)$ or $\text{diag}(N)$ are not identical, the errors in temporal solutions will show first order convergence. This issue on Choi's scheme ought to be concerned as it is still being used in the latest public version (i.e., version 6.0) of the OpenFOAM{\textregistered}. 

For the time-marching of the semi-discrete incompressible Navier-Stokes equations, we develop a new method of applying implicit Runge-Kutta schemes, i.e., (\ref{RKformulation}), and analyze its convergence results mathematically. Compared to the standard direct approach, the new one significantly reduces the cost of mathematical calculation in the momentum interpolation and delivers the higher-order pressure without requiring an additional computational effort. The direct approach and the proposed method, when applied respectively to the semi-discrete Navier-Stokes system, take different values of the source term in the internal stage continuity equation. Instead of taking the explicitly prescribed values, the source term in the internal stage continuity equation given by the proposed method consists of two components, a Runge-Kutta approximation of the source term and a pre-determined variable which ensures the continuity equation is strictly satisfied at the new time-step. 

Theorem \ref{theorem1} gives the convergence results of the proposed method. In the theorem, an implicit Runge-Kutta scheme with a non-singular coefficient matrix is applied to an index 2 differential-algebraic problem (\ref{I2DAE}) (which can be regarded as a generalization of the semi-discrete incompressible Navier-Stokes system on a stationary mesh) using the scheme (\ref{RKformulation}). The theorem implies that the orders of convergence for the differential and algebraic components delivered by the proposed method are identical to those of the considered Runge-Kutta scheme for general index 1 problems. The proof the theorem is done by three steps: first, we compare the system of equations of the proposed method with the system of equations obtained by applying the considered Runge-Kutta scheme to an index 1 problem which is equivalent to the problem (\ref{I2DAE}), and demonstrate that the former system can be regarded a perturbed system of the latter; second, we investigate the perturbations and their accumulation over the whole time domain (based on Lemma \ref{lemma2}); finally, based on the influence of perturbations and the convergence results of the considered Runge-Kutta scheme for index 1 problems, we therefore obtain those results of the proposed method for the problem (\ref{I2DAE}). Note that we also consider the effect of the residual errors, resulting from solving the non-linear Runge-Kutta internal stage equations, on the global convergence results in our proof. Such a consideration is very necessary since the solutions of the non-linear equations require a fixed-point iteration method and the iterations are usually terminated before the exact solutions are achieved. Theorem \ref{theorem1} also suggests that stiff-accurate schemes are desirable in the proposed method since they can attain the classical order of convergence for both the velocity and pressure. Another advantage of the stiff-accurate schemes over the others is that the velocity and pressure obtained in a new time-step do not relate to the pressure of the previous time-step.

We also investigate the storage consumption required by stiff-accurate DIRK schemes. Inspired by the work of \cite{van1972} and \cite{kennedy2000}, two families of low storage stiff-accurate DIRK schemes, i.e., (\ref{lowStorage1}) and (\ref{lowStorage2}) with their coefficient tableau of the forms (\ref{DIRKfamily1}) and (\ref{DIRKfamily2}) respectively, are developed to reduce the storage consumption during the implementation. The first family of low storage schemes requires two registers for the velocity and one register for the pressure. The second family of low storage schemes requires three registers for the velocity and one register for the pressure. Based on the degrees of freedom allowable in their coefficient tableau and the order conditions of Runge-Kutta methods, the two families of low storage stiff-accurate DIRK schemes require two Runge-Kutta internal stages for order 2 and three internal stages for order 3. We then demonstrate a two-stage second-order SDIRK scheme (denoted by SDIRK2) and a three-stage third-order SDIRK scheme (denoted by SDIRK3) as the examples for the two families of low storage schemes. It should be emphasized that the coefficient tableau of the two SDIRK schemes were first presented by \cite{alexander1977} for solving ordinary differential equations. For the solution of the discretized non-linear Navier-Stokes system at Runge-Kutta internal stages, we present a solution algorithm which employs the Picard linearization to deal with the non-linearity of the momentum equation and uses the PISO algorithm to handle the pressure-velocity coupling of the linearized Navier-Stokes system. This solution algorithm ensures the enforcement of the discretized continuity equation at every Runge-Kutta internal stage even in the case that the Picard iterations are terminated before achieving the exact solutions. 

For the examination of the accuracy and convergence of the presented methods in momentum interpolation and temporal discretization, we consider a benchmark problem, the Taylor-Green Vortex. Computations are conducted on two dimensional Cartesian grids. Different boundary conditions and numerical schemes are considered to demonstrate the performance of the proposed methods under all sorts of situations. The verification of spatial accuracy is carried out in two stages. In the first stage, we examine the convergence of the semi-discrete incompressible Navier-Stokes systems based on GFISDM-Z and GFISDM-H respectively. In the second stage, we demonstrate the errors of the velocity and pressure after the time-marching of the semi-discrete incompressible Navier-Stokes systems. In both stages, GFISDM-Z and GFISDM-H attain second-order of convergence for both the velocity and pressure.

In the verification of temporal accuracy, we only consider GFISDM-H for simplicity. The main concerns of the temporal accuracy verification, in addition to the convergence results, are the accuracy per computational cost and the influence of the Picard iteration. The results indicate that, for both SDIRK2 and SDIRK3, the temporal solutions of the velocity and pressure show the classical order of convergence with a sufficient number of Picard iterations. In addition, SDIKR3, as the higher-order scheme, requires more Picard iterations than SDIRK2 to reach its expected order of convergence, especially for the semi-discrete incompressible Navier-Stokes system with strong non-linearity. Moreover, when the classical order of convergence is reached, SDIRK3 performs better than SDIRK2 in terms of the accuracy per computational cost with the same number of Picard iterations. To demonstrate the superiority of the proposed method over the standard one in convergence results, we also compute the temporal errors of the direct approach and the result shows SDIRK2 and SDIRK3 just deliver first-order temporal-accurate pressures for the cases with unsteady Dirichlet boundary conditions. At the end of the verification, we examine the negative effect of Choi's momentum interpolation scheme on temporal convergence results, and numerical results validate our prediction.

Finally, it should be emphasized again the analysis of the accuracy and convergence of GFISDM schemes presented in Section \ref{section2.4} is limited to the cases in which Cartesian or curvilinear structured grids are considered and polynomial-based methods are adopted in the interpolation schemes. Also note that the temporal convergence results of the proposed method for the time-marching of the semi-discrete incompressible Navier-Stokes equations are based on the assumption that the mesh system is static. Those results do not hold for deforming mesh problems. Meanwhile, the implementation of the proposed momentum interpolation framework on deforming mesh problems is very easy, and the order results still apply. The only issue needs to be concerned during the implementation is the enforcement of the geometric conversation law. We will pick up this problem along with a detailed discussion on the scheme of GFISDM-H in our future work.

\section*{Acknowledgement}
The authors gratefully acknowledge the support provided in part by the NSF Grant No. 1562244 and No. 1612843.

\bibliographystyle{elsarticle-harv}
\bibliography{main}

\end{document}